\documentclass[a4paper,11pt]{article}

\usepackage{amsmath}
\usepackage{amsfonts}
\usepackage{amssymb}
\usepackage{amsthm}
\usepackage{mathrsfs}
\usepackage{graphicx}
\usepackage{color}
\usepackage[authoryear,sort,comma]{natbib}
\usepackage[margin=0.8in]{geometry}
\usepackage[titletoc]{appendix}
\usepackage{bm}
\usepackage{algorithm}
\usepackage{algpseudocode}
\usepackage{algpascal}
\usepackage{url}
\usepackage{afterpage}
\usepackage{enumitem} 
\usepackage{multirow}
\usepackage{booktabs}
\usepackage{subfigure}

\newlength{\vslength}
\numberwithin{equation}{section}

\newtheorem{theorem}{Theorem}[section]
\newtheorem{lemma}[theorem]{Lemma}
\newtheorem{proposition}[theorem]{Proposition}
\newtheorem{corollary}[theorem]{Corollary}
\newtheorem*{remark}{\bf Remark}

\def\cA{\mathcal A}

\def\cF{\mathcal F}

\def\cN{\mathcal N}

\def\cQ{\mathcal Q}

\def\cS{\mathcal S}
\def\cT{\mathcal T}
\def\cU{\mathcal U}

\newcommand{\bA}{{\bf A}}

\newcommand{\bC}{{\bf C}}
\newcommand{\bD}{{\bf D}}

\newcommand{\bF}{{\bf F}}

\newcommand{\bH}{{\bf H}}
\newcommand{\bI}{{\bf I}}

\newcommand{\bV}{{\bf V}}

\newcommand{\bX}{{\bf X}}

\newcommand{\bY}{{\bf Y}}

\newcommand{\bbE}{{\mathbb E}}

\newcommand{\bbN}{{\mathbb N}}
\newcommand{\bbP}{{\mathbb P}}
\newcommand{\bbR}{{\mathbb R}}

\newcommand{\bSigma}{\bm{\Sigma}}

\newcommand{\bOmega}{\bm{\Omega}}

\newcommand{\eg}{{\it e.g.}}

\newcommand{\bc}{\begin{center}}
\newcommand{\ec}{\end{center}}
\newcommand{\be}{\begin{equation}}
\newcommand{\ee}{\end{equation}}
\newcommand{\ba}{\begin{array}}
\newcommand{\ea}{\end{array}}
\newcommand{\bean}{\setlength\arraycolsep{1pt}\begin{eqnarray*}}
\newcommand{\eean}{\end{eqnarray*}}
\newcommand{\bea}{\setlength\arraycolsep{1pt}\begin{eqnarray}}
\newcommand{\eea}{\end{eqnarray}}
\newcommand{\ben}{\begin{enumerate}}
\newcommand{\een}{\end{enumerate}}
\newcommand{\bed}{\begin{itemize}}
\newcommand{\eed}{\end{itemize}}

\DeclareMathOperator*{\argmin}{argmin}

\RequirePackage{amsthm,amsmath,amsfonts,amssymb}
\RequirePackage[authoryear]{natbib}
\RequirePackage[colorlinks,citecolor=blue,urlcolor=blue]{hyperref}
\RequirePackage{graphicx}
\usepackage{comment}
\usepackage{anyfontsize}

\usepackage{newtxtext}
\usepackage{newtxmath}
\usepackage{mathrsfs}
\usepackage{microtype}
\UseMicrotypeSet[protrusion]{basicmath}

\usepackage{amsthm}
\usepackage{mathrsfs}
\usepackage{graphicx}
\usepackage{color}
\usepackage{bm}
\usepackage{algorithm}
\usepackage{algpseudocode}
\usepackage{algpascal}
\usepackage{url}
\usepackage{afterpage}
\usepackage{enumitem} 
\usepackage{multirow}
\usepackage{booktabs}
\usepackage{subfigure}

\usepackage{mathtools}
\usepackage{dsfont}
\usepackage{alphabeta}

\usepackage{xr}

\usepackage{import}

\usepackage{mathtools}
\usepackage{dsfont}
\usepackage[utf8]{inputenc}
\usepackage[greek,english]{babel}
\usepackage{alphabeta}
\usepackage{caption}

\usepackage{aligned-overset}

\newcommand{\rmd}{\mathrm{d}}
\newcommand{\rmx}{\mathrm{x}}

\newcommand{\bbF}{\mathbb F}

\newcommand{\OVB}{\textnormal {\rm \texttt{OVB}}}

\newcommand{\symmPD}{\cS_{\texttt{++}}^{p}}
\newcommand{\symmPSD}{\cS_{\texttt{+}}^{p}}

\newcommand{\app}{{\rm app}}

\newcommand{\ML}{\textnormal {\rm \texttt{ML}}}
\newcommand{\MAP}{\textnormal {\rm \texttt{MAP}}}

\newcommand{\BvM}{\textnormal {\rm \texttt{BvM}}}
\newcommand{\SGD}{\textnormal {\rm \texttt{SGD}}}
\newcommand{\ASGD}{\textnormal {\rm \texttt{ASGD}}}
\newcommand{\wSGD}{\textnormal {\rm \texttt{wSGD}}}

\newcommand{\BOO}{\textnormal {\rm \texttt{BOO}}}

\newcommand{\exit}{\textnormal {\rm exit}}

\newcommand{\scrE}{\mathscr{E}}
\newcommand{\logdet}{{\rm logdet}}

\theoremstyle{definition}
\newtheorem*{assumption*}{Assumption}

\title{Bayesian online learning in the one-pass regime: Frequentist validity and uncertainty quantification}
\author{Jeyong Lee \footnote{Weierstrass Institute for Applied Analysis and Stochastics (WIAS)}  
\quad Junhyeok Choi\footnote{Department of Industrial and Management Engineering, Pohang University of Science and Technology (POSTECH). \\ 
Emails: \texttt{lee@wias-berlin.de}, \ \ \texttt{(cjunh4810, dgkim, mchae)@postech.ac.kr}
}
\quad Dongguen Kim$^\dagger$
\quad Minwoo Chae$^\dagger$}
\date{\today}

\begin{document}

\maketitle
    \addtocontents{toc}{\protect\setcounter{tocdepth}{-1}}

\begin{abstract}
    Bayesian online learning provides a coherent framework for sequential inference. However, its theoretical understanding remains limited, particularly in the one-pass setting. Existing theoretical guarantees typically require the mini-batch sample size to diverge, a condition that fails in the one-pass regime. In this paper, we propose a new Bayesian online learning algorithm tailored to the one-pass setting, which incorporates a warm-start phase to ensure stable sequential updates. For this algorithm, we show that the sequentially updated posterior attains the optimal convergence rate. Building on this, we establish an online analogue of the Bernstein--von Mises theorem, which guarantees valid uncertainty quantification without diverging mini-batch sample sizes. Our analysis is based on a novel theoretical framework that differs fundamentally from existing approaches in the online learning literature. Numerical experiments on generalized linear models show that the proposed method matches the performance of the batch estimator while outperforming existing online procedures.
\end{abstract}

\section{Introduction} 

Online learning refers to a statistical framework for settings where data arrive sequentially, typically arising in large-scale or streaming environments. Unlike batch learning, which constructs an estimator using the entire dataset at once, online learning updates parameters sequentially as new observations become available. Each update relies only on the previous estimate and the current observation without revisiting past data. 
In practice, online updates are often performed using small mini-batches of data. A particularly stringent formulation is the one-pass setting, in which each observation is processed individually and contributes exactly once to the update. 

Online learning methods have been extensively developed across a wide range of statistical models. Representative examples include topic models \citep{hoffman2010online, wang2011online, kim2016online}, Dirichlet process mixtures \citep{jeong2023online}, matrix factorization \citep{mairal2010online}, survival analysis \citep{xue2020online, wu2021online, choi2025online}, and quantile regression \citep{chen2019quantile, lee2024fast}.

From a theoretical perspective, however, the online learning setting imposes structural constraints that fundamentally distinguish it from classical batch learning. In particular, since data arrive sequentially, the total sample size is typically unknown, making it difficult to tune hyperparameters that depend on a fixed sample size. Moreover, as past observations are not revisited, errors or biases introduced at early stages may propagate through subsequent updates and affect the entire trajectory of the estimator. Consequently, the statistical analysis of online estimators departs from the classical M-estimation framework, necessitating new analytical tools. Understanding the theoretical behavior of online estimators has therefore become an important problem.

In the frequentist literature, stochastic gradient descent (SGD) is a standard approach to online learning. Due to its recursive structure and computational efficiency, SGD, originally introduced by \citet{robbins1951stochastic}, and its variants have been extensively studied. As pioneering works, \citet{ruppert1988efficient} and \citet{polyak1992acceleration} independently proposed averaging schemes for SGD iterates, with \citet{polyak1992acceleration} establishing asymptotic normality of the resulting estimator. More recently, \citet{toulis2017SGD} introduced stabilized variants with improved finite-sample performance, and subsequent works have focused on establishing valid uncertainty quantification for SGD-based estimators \citep{chen2020SGD, zhu2023online, lee2022fast, fang2018online}. Importantly, these results extend naturally to the one-pass setting while retaining their theoretical guarantees.

On the other hand, the Bayesian framework provides a conceptually appealing alternative, as it updates a distribution over the parameter space in a sequential manner. The posterior distribution at each step can be used as the prior for the next step, allowing for a coherent learning procedure in the online setting. However, except in conjugate models, exact posterior updates are generally intractable, necessitating approximation at each step \citep{opper1999bayesian, solla1999online}. This has led to the development of Bayesian online learning methods that approximate the posterior within a tractable family, often using variational techniques \citep{broderick2013streaming, lin2013online, nguyen2018variational, lambert2022recursive, lambert2023limited}.

Despite these methodological advances, the theoretical understanding of Bayesian online learning remains limited. Most existing works focus on algorithmic design, while rigorous statistical guarantees are relatively scarce. To address this gap, recent work by \citet{lee2026online} has shown that sequentially updated distributions can be asymptotically equivalent to the full posterior based on the entire dataset. However, such results rely on diverging mini-batch sample sizes and therefore do not fully capture the one-pass setting. Moreover, empirical evidence in \citet{lee2026online} indicates that statistical efficiency can deteriorate substantially when the mini-batch size is small. These observations highlight the need for Bayesian online learning methods tailored to the one-pass setting, together with a rigorous theoretical analysis of their statistical properties.

Motivated by this gap, we propose a Bayesian online learning procedure tailored to the one-pass setting and establish its theoretical validity. The proposed method is inspired by variational approximation, where the expected log-likelihood is replaced by a local quadratic approximation, leading to a tractable update with a closed-form expression. It is worth noting that \citet{lee2026online} rely on an iterative optimization procedure with Monte Carlo approximations at every update, which can become computationally burdensome over long sequences of online updates. In contrast, our method admits an analytic update formula, enabling rapid updates by avoiding an iterative procedure at each step. This substantially reduces the computational cost and makes the procedure particularly attractive for one-pass learning. In addition, the proposed method does not require a user-specified learning-rate schedule, unlike SGD-based methods, thereby avoiding a potentially sensitive tuning parameter and making it easier to deploy in practice.

However, the local quadratic approximation is valid only in a neighborhood of the true parameter, and errors arising in early updates may propagate through subsequent updates in the one-pass setting. This motivates initializing the recursion with a sufficiently accurate estimate. To this end, the method incorporates a warm-start phase that ensures the estimator enters a local neighborhood before the online recursion begins.

We further establish an online analogue of the Bernstein--von Mises (BvM) theorem by analyzing the asymptotic behavior of the sequentially updated posterior distributions. This result shows that, even in the fully online regime, the proposed posterior achieves asymptotic efficiency and provides valid frequentist uncertainty quantification, without requiring diverging mini-batch sizes, in contrast to \citet{lee2026online}. To establish this result, we develop a theoretical framework adapted to the algorithmic structure of the procedure, enabling a refined non-asymptotic analysis. This framework differs substantially from existing techniques in the SGD literature and avoids the strong structural assumptions typically imposed in their analyses. See Section \ref{sec:comparison} for details.

Finally, we complement our theoretical findings with extensive numerical experiments, demonstrating that the proposed method achieves strong empirical performance compared to existing approaches across a range of settings. In particular, the experiments highlight the importance of the warm-start phase, as omitting this phase can lead to substantially slower convergence in certain settings. Nevertheless, the results also indicate that stable performance can often be achieved with relatively short warm-start phases, and in logistic regression settings the proposed procedure continues to perform well even without a warm-start phase. See Section \ref{sec:experiments} for details.

The remainder of this paper is organized as follows. In the following subsections, we introduce our setup and the basic notation used throughout the paper. Section \ref{sec:Bayesian_online_learning} presents the general framework of Bayesian online learning and introduces the proposed one-pass algorithm. Section \ref{sec:theory} develops the main theoretical results, including the convergence analysis of the estimator and the establishment of the online BvM theorem. Numerical experiments are presented in Section \ref{sec:experiments}, and concluding remarks are given in Section 5. All proofs and additional technical details are deferred to the Appendix.


\subsection{Setup and notations}

Let $(Y_t)_{t \in \bbN}$ be a sequence of observations. Let $\bbF = (\cF_t)_{t \in \bbN}$ be the filtration given by $\cF_t = \sigma(Y_1, ..., Y_t)$ for any $t \in \bbN$.
Let $\Theta = \bbR^p$.
For $t \in \bbN$, let $p_{t, \theta}(\cdot)$ be the probability density function for $Y_t$ parametrized by $\theta \in \Theta$. Let $\ell_t(\theta) = -\log p_{t, \theta}(Y_t)$ be the negative log-likelihood function. 
Let $\bbP_{\theta}^{(t)}$ be the joint probability measure corresponding to the product density function 
$
    (y_1, ..., y_t) \mapsto \prod_{s=1}^{t} p_{s ,\theta}(y_s).
$
Here, we assume that our model is correctly specified; that is, for any $t \in \bbN$, $(Y_s)_{s =1}^{t}$ is generated from $\bbP_{\theta_{\star}}^{(t)}$ for some true parameter $\theta_{\star} \in \Theta$. For notational convenience, we denote $\bbP_{\theta_{\star}} = \bbP_{\theta_{\star}}^{(t)}$ whenever the value of $t$ is clear from the context. Also, let $\bbE$ denote the expectation under $\bbP_{\theta_{\star}}$.

Let $\bbN_0 = \bbN \cup \{ 0 \}$.
For $a, b \in \bbR$, $a \vee b$ and $a \wedge b$ denote the maximum and minimum of $a$ and $b$, respectively.
For $x \in \mathbb{R}$, let $\lceil x \rceil$ denote the smallest integer greater than or equal to $x$.
For two positive sequences $(a_n)$ and $(b_n)$, $a_n \lesssim b_n$ (or $a_n = O(b_n)$) means that $a_n \leq C b_n$ for some constant $C \in (0, \infty)$.
Also, $a_n \asymp b_n$ indicates that $a_n \lesssim b_n$ and $b_n \lesssim a_n$.
The notation $a_n \ll b_n$ (or $a_n = o(b_n)$) implies that $a_n / b_n \rightarrow 0$ as $n \rightarrow \infty$. We use the standard $O_P$ and $o_P$ notation for stochastic order symbols; see \citet{van2000asymptotic}.
With a slight abuse of notation, for $\nu_n \in \bbR^p$ satisfying $\| \nu_n \|_2 = o(1)$, we write $\nu_n = o(1)$.
For $1 \leq q \leq \infty$, $\|\cdot\|_{q}$ indicates the $\ell_q$-norm of a vector. 
For $m \in \bbN$, let $[m] = \{1, 2, ..., m\}$. 

For a two times differentiable function $f : \bbR^{p} \rightarrow \bbR$, let
\begin{align*}
    &\nabla f(\theta) = \left( \dfrac{\partial}{\partial \theta_{i_1}} f(\theta) \right)_{i_1 \in [p]} \in \bbR^{p},  \quad 
    \nabla^2 f(\theta) = \left( \dfrac{\partial^2}{\partial \theta_{i_1} \partial \theta_{i_2}} f(\theta) \right)_{i_1, i_2 \in [p]} \in \bbR^{p \times p}.
\end{align*}
For $z = (z_j)_{j \in [p]} \in \bbR^{p}$ and $k \geq 2$, let
\begin{align*}
     z^{\otimes k} = \left( z_{i_1} \cdots z_{i_k} \right)_{i_1, ..., i_k \in [p]} \in \bbR^{p^{k}}.
\end{align*}
For two $k$-order tensors $\bA = (A_{i_1, ..., i_k})_{i_1, ..., i_k \in [p]} \in \bbR^{p^{k}}$ and $\mathbf{B} = (B_{i_1, ..., i_k})_{i_1, ..., i_k \in [p]} \in \bbR^{p^{k}}$, let
\begin{align*}
    \langle \bA, \mathbf{B} \rangle = \sum_{i_1, ..., i_k \in [p]} A_{i_1, ..., i_k} B_{i_1, ..., i_k}.
\end{align*}

Let $\symmPD$ and $\symmPSD$ denote the sets of $p \times p$ symmetric positive definite and symmetric positive semi-definite matrices, respectively. Let $\bI_{p} \in \bbR^{p \times p}$ denote the identity matrix.
For a matrix $\mathbf{A} = (a_{ij}) \in \bbR^{n \times p}$, let $\lambda_{\operatorname{min}}(\bA)$ and $\lambda_{\operatorname{max}}(\bA)$ denote the smallest and largest singular values of $\bA$, respectively.
For simplicity, $\| \bA \|_{2}$ will often be used interchangeably with $\lambda_{\operatorname{max}}(\bA)$.
Let $\| \bA \|_{\rm F} = (\sum_{ij} a_{ij}^{2})^{1/2}$ be the Frobenius norm.
For two distinct matrices $\bA, \mathbf{B} \in \bbR^{n \times n}$, $\bA \succeq \mathbf{B}$ means $\bA - \mathbf{B}$ is positive semi-definite matrix.
For a $k$-th order tensor $\bA = (A_{i_1, ..., i_k})_{i_1, ..., i_k \in [p]} \in \bbR^{p^{k}}$, define the operator norm of $\bA$ by
\begin{align*}
    \left\| \bA \right\|_{2} = \sup_{u_{1}, \ldots, u_{k} \in \cU} \left| \langle \bA, u_{1} \otimes ... \otimes u_{k} \rangle \right|,
\end{align*}
where $\cU = \{ u \in \bbR^{p} : \| u \|_{2} = 1 \}$.

For $\sigma^{2} \geq 0$, a random vector $X \in \bbR^{p}$ is said to be $\text{SubG}(\sigma^2)$ if
\begin{align*}
    \log \bbE \exp \big( \alpha^{\top} (X - \bbE X) \big) \leq \sigma^{2} \left\| \alpha \right\|_{2}^{2}/2, \quad \forall \alpha \in \bbR^{p}
\end{align*}
We adopt the convention that $\inf \emptyset = \infty$.

\begin{table}[t!]
\centering
\caption{Summary of important notations}
\label{table:notations_main}
\renewcommand{\arraystretch}{1} 
\begin{tabular}{c c || c c || cc}
\hline
\textbf{Notation} & \textbf{Location} & \textbf{Notation} & \textbf{Location} & \textbf{Notation} & \textbf{Location} \\
\hline\hline
$d_{V}(\cdot,\cdot)$
  & \eqref{def:TV_distance}
  & $V_t, \, \Delta_t$
  & \eqref{def:V_t_Delta_t} 
  & $\theta_t, \bOmega_t$
  & \eqref{Alg} \\
  $\hat \theta_t^{\MAP}, \hat \theta_t^{\ML}$
  & \eqref{def:batch_map}, \eqref{def:MLE} 
  & $\underline \lambda_0, \overline \lambda_0$
  & \eqref{def:prior_precision}
  & $\Theta_r, \, r$
  & \eqref{def:local_set} \\
$t_{0}, t_{\exit}$
  & \eqref{Alg}, \eqref{def:exit_time}
  & $g_t, \bH_t$
  & \eqref{Alg} 
  & $L_{1:t}(\cdot), \widetilde L_{s}(\cdot)$
  & \eqref{def:batch_loss} \\
  $\bF_{t, \theta}$
  & \eqref{A3-b}
  & $\epsilon_{\app, t}$
  & \eqref{def:epsilon_app}
  & $\widetilde \ell_t(\cdot)$
  & \eqref{def:tilde_loss} \\
\hline
\end{tabular}
\end{table}

\section{Bayesian online learning algorithm} \label{sec:Bayesian_online_learning}


\subsection{Bayesian online learning framework}

For $N \in \bbN$, we denote by $\bD_{1:N} = (Y_1, \ldots, Y_N)$ the collection of all samples up to the first $N$ observations. 
Let $\Pi_0$ denote the initial prior on $\Theta$. The corresponding posterior distribution is $\Pi(\cdot \mid \bD_{1:N})$, which we refer to as the $N$-th full posterior. In online learning, we assume that the data become available sequentially in the order $\bD_1, \bD_2, \ldots, \bD_T$, where $\bD_t = (Y_{N_{t-1}+1}, Y_{N_{t-1}+2}, \ldots, Y_{N_t})$ represents the $t$-th mini-batch, with $N_0 = 0$. For simplicity, we assume that each mini-batch has the same size, denoted by $n$, so that $N_t = nt$ for all $t$. 

In practice, many Bayesian online learning methods can be characterized as recursive processes that sequentially apply an update step followed by an approximation step to maintain computational tractability. This general framework with a variational approximation was analyzed in \citet{lee2026online}, which considered a sequence $(\Pi_{t}^{\OVB})_{t\le T}$ of approximated posterior distributions starting from $\Pi_{0}^{\OVB}=\Pi_{0}$. For $t\ge1$, let $\widetilde{\Pi}_{t}(\cdot \mid \bD_{t})$ be the posterior distribution obtained by updating the prior $\Pi_{t-1}^{\OVB}$ with the data $\bD_{t}$ via Bayes' rule. Specifically, the updated posterior is given by
\begin{align} \label{eq:online-Bayes}
    \widetilde{\Pi}_{t}( \cA \mid \bD_t) = \dfrac{
    \int_{\cA} \exp\left\{ - \sum_{i \in I_t} \ell_i(\theta) \right\} \rmd \Pi_{t-1}^{\OVB} (\theta)
    }{
    \int_{\Theta} \exp\left\{ - \sum_{i \in I_t} \ell_i(\theta) \right\} \rmd \Pi_{t-1}^{\OVB} (\theta)
    }
\end{align}
for any measurable $\cA \subset \Theta$, where $I_t = \{ N_{t-1}+1, \ldots, N_t \}$ is the index set for $\bD_t$.
Let $\cQ$ be a variational family of probability measures on $\Theta$. 
Since $\widetilde{\Pi}_{t}(\cdot \mid \bD_{t})$ is computationally intractable in general, we approximate the updated posterior by projecting it onto the class $\cQ$. The approximated posterior is defined as
\begin{align} \label{eq:KL-projection}
    \Pi_t^{\OVB} = \argmin_{Q \in \cQ} K \big( Q; \: \widetilde{\Pi}_{t}(\cdot \mid \bD_t) \big),
\end{align}
where
$ 
    K(P; Q) = \int \log\left( \rmd P / \rmd Q \right) \rmd P
$
denotes the Kullback--Leibler (KL) divergence. Thus, we obtain the sequence $(\Pi_t^{\OVB})_{t=1}^{T}$ of approximated posterior distributions, where $\Pi_t^{\OVB} \in \cQ$ for every $t \geq 1$.
This general procedure has been applied to many popular Bayesian models; see
\cite{broderick2013streaming, lin2013online, nguyen2018variational, choi2025online}.

It is well known that minimizing the KL divergence as in \eqref{eq:KL-projection} is equivalent to minimizing the negative evidence lower bound (ELBO); that is,
\begin{align} \label{max_ELBO}
    \Pi_t^{\OVB} = \argmin_{Q \in \cQ } \left\{ \mathbb{E}_{Q} \Big[ \sum_{i \in I_t} \ell_i(\theta) \Big] + K \big(Q ;\: \Pi_{t-1}^{\OVB}\big) \right\},
\end{align}
where $\bbE_{Q}$ denotes the expectation with respect to $Q$. Hence, with the Gaussian variational family $\cQ = \{ Q = N(\mu, \bOmega^{-1}) : \mu \in \bbR^p, \bOmega \in \symmPD \}$, the online learning procedure for obtaining $(\Pi_t^{\OVB})_{t=1}^{T}$ reduces to successive optimization problems over $\bbR^p \times \symmPD$. \citet{lee2026online} analyzed theoretical properties of $(\Pi_t^{\OVB})_{t=1}^{T}$ under the Gaussian variational family $\cQ$, establishing asymptotic guarantees when the mini-batch size $n = |\bD_t |$ is sufficiently large.

\subsection{Bayesian online one-pass algorithm} \label{ssec:alg}

In the one-pass setting ($n = | \bD_t | = 1$), solving the ELBO optimization at each iteration incurs a substantial computational cost. To address this, we replace the log-likelihood term in \eqref{max_ELBO} with its second-order local approximation. This leads to a modified version of \eqref{max_ELBO}:
\begin{align} \label{new_ELBO}
\Pi_t = \argmin_{Q \in \cQ} \left\{ \mathbb{E}_{Q} \big[ \widetilde \ell_t(\theta)\big] + K \big(Q ;\: \Pi_{t-1}\big) \right\},
\end{align}
where $\Pi_{t-1} = N(\theta_{t-1}, \bOmega_{t-1}^{-1})$ and
\begin{align} \label{def:tilde_loss}
    \widetilde \ell_t(\theta) 
    = \ell_t( \theta_{t-1} ) 
    + \langle \nabla \ell_t( \theta_{t-1} ), \theta - \theta_{t-1} \rangle  
    + \dfrac{1}{2} \langle \nabla^2 \ell_t(\theta_{t-1}), (\theta - \theta_{t-1})^{\otimes 2} \rangle.
\end{align}

Under the Gaussian variational family $\cQ$, the problem in \eqref{new_ELBO} admits a closed-form solution. 
Specifically, the quadratic structure of $\widetilde{\ell}_t(\cdot)$ yields the following updates:
\begin{align}
\begin{aligned} \label{eqn:updating_formula}
    \bOmega_t &= \bOmega_{t-1} + \nabla^2 \ell_t(\theta_{t-1}), \\
    \theta_t &= \theta_{t-1} - \bOmega_t^{-1} \nabla \ell_t  (\theta_{t-1}).
\end{aligned}    
\end{align}
Thus, the one-pass update takes a fully explicit recursive form, making it computationally appealing.
Moreover, if $\nabla^2 \ell_t(\theta_{t-1})$ is a rank-$1$ matrix of the form $u_t v_t^{\top}$ for some $u_t, v_t \in \bbR^p$, then the inverse $\bOmega_t^{-1}$ admits the recursive representation
\begin{align*}
    \bOmega_{t}^{-1} = \bOmega_{t-1}^{-1} - \dfrac{ \bOmega_{t-1}^{-1} u_t v_t^{\top} }{ 1 + v_t^{\top} \bOmega_{t-1}^{-1} u_t}  \bOmega_{t-1}^{-1},
\end{align*}
which allows us to avoid computationally expensive matrix inversions.

However, the validity of this update relies on a local quadratic approximation and may deteriorate under poor initialization. The update in \eqref{eqn:updating_formula} can be interpreted as a stochastic analogue of Newton's method, which is well-known to be sensitive to initialization. 
To illustrate this, consider the objective function $f(x) = \sqrt{x^2 + 1}$ with minimizer $x_{\star} = 0$.
The corresponding Newton update is given by
\begin{align*}
    x_{t+1} 
    = x_{t} - \dfrac{\nabla f(x_t)}{\nabla^2 f(x_t)}
    = x_{t} - \dfrac{ x_t (x_t^2 + 1)^{-1/2} }{(x_t^2 + 1)^{-3/2}} = -x_t^{3}, \quad t \in \bbN_0,
\end{align*}
which converges to $0$ if and only if $|x_0| < 1$. 
Hence, a careful initialization is crucial for the stability of Newton-type methods.

To alleviate such dependence on initializations, we adopt a warm-start strategy, which ensures that the estimator is not excessively far from $\theta_{\star}$ before the online recursion begins. To ensure such localization, we employ a batch estimator. To formalize this idea, for the initial prior $\Pi_0 = N(\theta_0, \bOmega_0^{-1})$, we introduce the following batch maximum a posteriori (MAP) estimator:
\begin{align} \label{def:batch_map}
    \hat \theta_t^{\MAP} = \argmin_{\theta \in \Theta} \left\{ \sum_{s=1}^{t} \ell_s(\theta) + \dfrac{1}{2} \big\| \bOmega_0^{1/2} (\theta - \theta_0) \big\|_2^{2} \right\}, \quad  t \in \bbN.
\end{align}
We now describe the proposed Bayesian online one-pass algorithm.

\ben[label=(\textbf{Alg}), ref = \textbf{Alg}]
\item Let $t_{0} \in \bbN$.
If $t \in \bbN$ with $t \leq t_{0}$, we use batch updates:
\begin{align*}
    \bOmega_t = \bOmega_0 + \sum_{s=1}^{t} \nabla^2 \ell_s( \hat \theta_{t}^{\MAP} ), \quad
    \theta_t = \hat \theta_{t}^{\MAP}.
\end{align*}
If $t \in \bbN$ with $t > t_{0}$, we use online updates
\begin{align*}
    \bOmega_t = \bOmega_{t-1} + \bH_t, \quad 
    \theta_t = \theta_{t-1} - \bOmega_t^{-1} g_t,
\end{align*}
where $g_t = \nabla \ell_t(\theta_{t-1}) \in \bbR^p$ and $\bH_t = \nabla^2 \ell_t(\theta_{t-1}) \in \symmPSD$.
\vspace{-\belowdisplayskip}
\label{Alg}
\een
The above scheme consists of a warm-start phase followed by an online update phase.
In Algorithm \eqref{Alg}, the phase transition time $t_{0}$ determines the length of the initial warm-start phase.
For $t \leq t_{0}$, we compute the batch estimator $\hat \theta_t^{\MAP}$ based on the full sample $\bD_{1:t}$ and set $\theta_t = \hat \theta_t^{\MAP}$. For $t > t_{0}$, we switch to a one-pass update. This yields the sequence $(\Pi_t)_{t \in \bbN}$ with $\Pi_t = N(\theta_t, \bOmega_t^{-1})$. 

The choice of $t_{0}$ plays an important role in both theoretical and empirical performance. In Section \ref{sec:theory}, we provide a sufficient condition on $t_0$ under which stable online updates are possible; see \eqref{A5}. Moreover, empirical results in Section \ref{sec:experiments} illustrate the necessity of the warm-start phase, showing that, in certain settings, omitting it leads to substantially slower convergence. 

This does not imply, however, that stable performance always requires a long warm-start phase. For example, in logistic regression settings, the proposed procedure continues to perform well even with very small values of $t_0$ (\eg, $t_0=0$). For practical use, Section \ref{sec:experiments} provides explicit guidelines for choosing $t_0$ that are not overly conservative, and the corresponding empirical results demonstrate stable behavior across all considered settings. Moreover, since larger values of $t_0$ generally lead to more stable behavior, tuning $t_0$ is relatively straightforward in practice.

While the online updates in Algorithm \eqref{Alg} resemble those of existing algorithms \citep[\eg,][]{bercu2020efficient, godichon2025adaptive}, the ordering of the updates is slightly different. Specifically,  although the precise forms vary across methods, existing approaches typically adopt the following update order:
\begin{align*}
    \theta_t = \theta_{t-1} - \bOmega_{t-1}^{-1} g_t, \quad 
    \bOmega_t = \bOmega_{t-1} + \bH_t.
\end{align*}
The key difference lies in the update of $\theta_t$: whether one uses $\bOmega_{t-1}$ or $\bOmega_t$.
This subtle discrepancy results in a significant difference in the theoretical analysis.
Since $\bOmega_{t-1}$ is $\cF_{t-1}$-measurable, existing approaches exploit the identity $\bbE [ \bOmega_{t-1}^{-1} g_t \mid \cF_{t-1} ] = \bOmega_{t-1}^{-1} \bbE [ g_t \mid \cF_{t-1} ]$, which facilitates the analysis. 
In contrast, this property no longer holds when $\bOmega_t$ is used.
For this reason, \citet{bercu2020efficient} modified the ordering of the updates to facilitate the theoretical analysis, while initially considering an update scheme similar to ours; see (3.1)–(3.6) in \citet{bercu2020efficient}. This motivates the development of a new theoretical framework tailored to our algorithm, which we present in detail in Section \ref{sec:rate_conv}.

\section{Theoretical results} \label{sec:theory}

In this section, we establish a theoretical justification for $\Pi_t$ proposed in Section~\ref{ssec:alg}. 
Our goal is to show that $\Pi_t$ becomes asymptotically equivalent to the full posterior $\Pi(\cdot \mid \bD_{1:t})$ in an appropriate sense. We refer to this result as the \textit{online BvM theorem}.

\subsection{Rates of convergence and online Bernstein--von Mises theorem} \label{sec:rate_online_BvM}

\subsubsection{Assumptions} \label{sec:assumptions}
We present the assumptions required for our theoretical analysis. 
To accommodate a broad class of statistical models, the conditions are stated in a high-level form. 
Detailed insights and motivations for each assumption will be provided in Appendix~\ref{sec:proof_strategy}; for clarity of presentation, we focus here on their formal statements and overall structure.
More specifically, the assumptions are organized into four groups: basic setup \eqref{A0}, local regularity \eqref{A1}--\eqref{A3}, fluctuation control \eqref{A4}, and the warm-start phase \eqref{A5}.
Before stating the assumptions, we introduce some notation.
Throughout the paper, we consider fixed $r > 0$ and $\rmx > 0$. We define the local neighborhood of $\theta_{\star}$ by
\begin{align} \label{def:local_set}
    \Theta_{r} = \Theta_{r}(\theta_{\star}) = \{ \theta \in \Theta : \| \theta - \theta_{\star} \|_2 \leq r \}.
\end{align}
The radius $r$ is chosen such that the regularity conditions, to be specified later, hold on $\Theta_r$. In particular, $r$ need not vanish; any fixed choice of $r$ is admissible provided these conditions are satisfied. The parameter $\rmx$ controls the level of high-probability events. All statements involving $\Theta_r$ and $\rmx$ are understood with respect to this choice of $r$ and $\rmx$.
With this notation in place, we now state the assumptions.

\paragraph*{Basic setup}
This condition imposes basic regularity, identifiability, and conditions on the initial prior parameters.
\ben[label=(\textbf{A0}), ref = \textbf{A0}] 
\item For each $t \in \bbN$, the map $\theta \mapsto \ell_t(\theta)$ is at least twice differentiable. Moreover, $\theta_{\star}$ is the unique minimizer of the map $\theta \mapsto \bbE \ell_t(\theta)$ for any $t \in \bbN$. The initial prior parameters $(\theta_0, \bOmega_0)$ are non-random quantities satisfying
\begin{align} \label{def:prior_precision}
    0 < \underline \lambda_0 \leq \lambda_{\min}(\bOmega_0) \leq \lambda_{\max}(\bOmega_0) \leq \overline \lambda_0 \leq K, 
\end{align} 
where $K > 0$ is a universal constant.
\label{A0}
\een

\paragraph*{Local regularity}
These conditions impose local regularity on $\Theta_r$.
\ben[label=(\textbf{A1}), ref=\textbf{A1}]
\item There exist constants $K_1, K_2 > 0$ such that, on an event $\scrE_1(\rmx)$ with $\bbP_{\theta_{\star}}(\scrE_1(\rmx)) \geq 1 - e^{-\rmx}$, for all $t \in \bbN$ with $t \geq K_1 (p + \rmx)$
\begin{align*}
    \inf_{ \{ \overline \theta_s \}_{s=1}^{t} \subset \Theta_{r} }
    \lambda_{\min} \bigg( \sum_{s=1}^{t} \nabla^2 \ell_s (\overline \theta_s) \bigg) 
    \geq 
    K_2 t.
\end{align*}
\label{A1}
\een    
\ben[label=(\textbf{A2}), ref=\textbf{A2}, topsep = 1pt]
\item For $t \in \bbN$ and $\theta, \theta' \in \Theta_r$, there exist $\phi_{t}(\theta, \theta') \geq 0$ and constant $L_1 \geq 0$ such that
\begin{align*}
    \nabla^2 \ell_t ( \theta )
    \preceq
    \big( 1 + \phi_{t}(\theta, \theta') \big) \nabla^2 \ell_t ( \tilde \theta ), \quad 
    \phi_{t}(\theta, \theta') \leq L_1 \| \theta - \theta' \|_2
\end{align*}
for some $\tilde \theta$ on the line segment between $\theta$ and $\theta'$ satisfying
\begin{align*}
    \ell_t(\theta)
    =
    \ell_t(\theta')
    + \langle \nabla \ell_t(\theta'), \theta- \theta' \rangle
    + \dfrac{1}{2} \langle \nabla^2 \ell_t ( \tilde \theta ), (\theta- \theta')^{\otimes 2} \rangle.
\end{align*}
Also, there exists a constant $L_2 > 0$ such that
$
    \max_{t \in \bbN} \sup_{\theta \in \Theta_r} \| \nabla^2 \ell_t ( \theta ) \|_2 \leq L_2.
$
\label{A2}
\een

\ben[label=(\textbf{A3}), ref=\textbf{A3}]
\item
There exists $L_3 > 0$ such that for any $\theta, \theta' \in \Theta_r$
\begin{align} \label{A3-a}
    \max_{t \in \bbN} \| \nabla^2 \ell_t ( \theta ) - \nabla^2 \ell_t ( \theta' ) \|_2 
    \leq L_3 \| \theta - \theta' \|_2.
\end{align}
Furthermore, there exists a constant $K_6 > 0$ such that for any $\theta, \theta + h \in \Theta_r$
\begin{align} \label{A3-b}
    (1 + e^{K_6 \| h \|_2})^{-1} \bF_{t, \theta}
    \preceq \bF_{t, \theta+h}
    \preceq (1 + e^{ K_6 \| h \|_2}) \bF_{t, \theta}, \quad \forall t \in \bbN,
\end{align}
where $\bF_{t, \theta} \overset{\rm def}{=} \sum_{s=1}^{t} \nabla^2 \ell_s(\theta)$.
\label{A3}
\een

\paragraph*{Fluctuation control}
This condition controls stochastic fluctuations along the online updates.
To localize the analysis, we introduce the following exit time from $\Theta_r$:
\begin{align} \label{def:exit_time}
t_{\exit} = \inf \{ t \in \bbN : t > t_0, \ \theta_t \notin \Theta_r \}.
\end{align}

\ben[label=(\textbf{A4}), ref=\textbf{A4}]
\item Fix a phase transition time $t_0 \in \bbN$ satisfying $t_0 \geq K_1(p + \rmx)$, where $K_1$ is specified in \eqref{A1}.
There exists a constant $K_3 > 0$ such that, on an event $\scrE_2(\rmx)$ satisfying $\bbP_{\theta_{\star}}(\scrE_2(\rmx)) \ge 1 - e^{-\rmx}$, for all $t \in \bbN$ with $t > t_{0}$
\begin{align} \label{A4-a}
    \sum_{s=t_{0}+1}^{t}
    \big\{
        \ell_s(\theta_{\star})
        -
        \ell_s( \theta_{s-1})
    \big\}
    \le
    K_3 \rmx.
\end{align}
Furthermore, there exist constants $K_4 > 0$ and $K_5 \geq 0$ such that, on an event $\scrE_3(\rmx)$ with $\bbP_{\theta_{\star}}(\scrE_3(\rmx)) \geq 1 - e^{-\rmx}$, for any $t \in \bbN$ with $t_{0} < t \leq t_{\exit}$ 
\begin{align} \label{A4-b}
    \sum_{s = t_{0}+1}^{t} \langle g_s, \bOmega_{s}^{-1} g_s \rangle
    \leq 
    K_4 p \log t + K_5 \big( \rmx + \log t \big).
\end{align}
The constants $K_3, K_4$, and $K_5$ are independent of $t_0$.
\label{A4}
\een

\paragraph*{Warm-start phase}
This condition ensures stability of the online updates through a sufficient requirement on $t_0$. 
To formulate an assumption on $t_0$, we introduce some notation.
For $1 \leq s \leq t \leq 6$, let $K_{s:t} = (K_j)_{j=s}^{t}$ and $L_{s:t} = (L_j)_{j=s}^{t}$, where these constants are specified in \eqref{A1}-\eqref{A4}.
For $t \in \bbN$, define the maximum likelihood estimator (MLE) $\hat \theta_t^{\ML}$ by
\begin{align} \label{def:MLE}
    \hat \theta_t^{\ML} = \argmin_{\theta \in \Theta} \sum_{s=1}^{t} \ell_s(\theta).
\end{align}

\ben[label=(\textbf{A5}), ref=\textbf{A5}]
\item There exists a sufficiently large constant $M = M( K_{1:5}, L_{1:2}, r) > 0$ such that, on an event $\scrE_4(\rmx)$ with $\bbP_{\theta_{\star}}(\scrE_4(\rmx)) \geq 1 - e^{-\rmx}$, if
$    t_0 \geq \lceil M (p\log (p \vee 3) + \rmx) \rceil, $
then $\big\| \bOmega_{t_0}^{1/2}( \hat \theta_{t_0}^{\MAP} - \theta_{\star}) \big\|_2 \leq C (p + \rmx)$ and
\begin{align*} 
    \big\| \hat \theta_t^{\MAP} - \theta_{\star} \big\|_2 \vee \big\| \hat \theta_t^{\ML} - \theta_{\star} \big\|_2
    \leq C \sqrt{\dfrac{p + \log t + \rmx}{t}}, \quad 
    \forall t \in \bbN \text{ with } t \geq t_0,
\end{align*}
where $C > 0$ is a universal constant. 
\label{A5}
\een

By \eqref{A5}, on $\scrE_{4}(\rmx)$, it follows that $\hat \theta_{t_0}^{\MAP} = \theta_{t_0} \in \Theta_{r}$ because $M = M( K_{1:5}, L_{1:2}, r)$ is large enough.

\subsubsection{Main results} \label{sec:main_results}
We now present the main results: the $\ell_2$-convergence rate and the online BvM theorem. 
The estimation rate of $\theta_t$ serves as a key ingredient in establishing the online BvM theorem. 
We therefore begin by deriving a non-asymptotic $\ell_2$-convergence rate of $\theta_t$ to $\theta_{\star}$ for $t \geq t_0$. To this end, for $t \in \bbN_0$, we define
\begin{align} \label{def:V_t_Delta_t}
    V_t = \big\| \bOmega_t^{1/2} ( \theta_t - \theta_{\star} ) \big\|_2^2, \quad 
    \Delta_t = \theta_t - \theta_{\star}.
\end{align}
The quantity $V_t$ measures the deviation $\Delta_t$ under the local geometry induced by $\bOmega_t$.
Since Algorithm \eqref{Alg} involves the coupled evolution of $(\theta_t, \bOmega_t)$, analyzing $V_t$ rather than $\| \Delta_t \|_2$ directly is natural in this context. 

For $\rmx > 0$, let $\scrE_{1:4}(\rmx) = \scrE_1(\rmx) \cap \scrE_2(\rmx) \cap \scrE_3(\rmx) \cap \scrE_4(\rmx)$, which collects the high-probability events specified in Section \ref{sec:assumptions}. In Theorem \ref{thm:rate_conv}, we show on $\scrE_{1:4}(\rmx)$ that the growth of $V_t$ remains sufficiently controlled, which in turn yields the optimal $\ell_2$-rate up to logarithmic factors.

\begin{theorem}[Rates of convergence] \label{thm:rate_conv}
    For given $\rmx, r > 0$, suppose that \eqref{A0}–\eqref{A2} and \eqref{A4}-\eqref{A5} hold.
    Then, on $\scrE_{1:4}(\rmx)$, we have for any $t \in \bbN$ with $t \geq t_0$
    \begin{align*} 
        V_t \leq K \left( p \log t + \rmx  \right), \quad 
        \| \theta_t - \theta_{\star} \|_2 \leq K \sqrt{\dfrac{p \log t + \rmx}{t}}, \quad 
        \theta_t \in \Theta_r,
    \end{align*}
    where $K = K(K_{2:5}) > 0$.
\end{theorem}

We now turn to the central result of this paper: the \textit{online BvM theorem}. Our goal is to establish that
\begin{align*}
    d_V \Big( \Pi_t, N \big( \hat \theta_t^{\ML}, \bF_{t, \theta_{\star}}^{-1} \big) \Big) \vee
    d_V \Big( \Pi_t, \Pi(\cdot \mid \bD_{1:t}) \Big) = o_{P}(1),
\end{align*}
where $d_V(\cdot, \cdot)$ denotes the total variation distance
\begin{align} \label{def:TV_distance}
    d_{V} \left( P,  Q \right) = \sup_{\cA} \left| P(\cA) - Q(\cA) \right|,
\end{align}
with the supremum taken over all measurable sets $\cA$, and $P$ and $Q$ denoting probability distributions.
Let $Q_{\BvM, t} = N \big( \hat \theta_t^{\ML}, \bF_{t, \theta_{\star}}^{-1} \big)$. By the classical BvM theorem for batch data, it is well-known that
\begin{align*}
    d_V \big( Q_{\BvM, t}, \Pi(\cdot \mid \bD_{1:t}) \big) = o_{P}(1).
\end{align*}
For detailed statements, we refer to the classical asymptotic theory \citep[Section 10.2]{van2000asymptotic} and its non-asymptotic counterpart \citep[Section 7]{lee2026online}. Since 
\begin{align*}
    d_V \big( \Pi_t, \Pi(\cdot \mid \bD_{1:t}) \big)
    &\leq 
    d_V \big( \Pi_t, Q_{\BvM, t} \big)
    +
    d_V \big( Q_{\BvM, t}, \Pi(\cdot \mid \bD_{1:t}) \big) \\
    &=
    d_V \big( \Pi_t, Q_{\BvM, t} \big)
    +
    o_{P}(1),
\end{align*}
it suffices to bound $d_V ( \Pi_t, Q_{\BvM, t} )$, which constitutes our main contribution. 
The following theorem establishes this bound.

\begin{theorem}[Online BvM] \label{thm:online_BvM}
    For given $\rmx, r > 0$, suppose that \eqref{A0}–\eqref{A5} hold.
    Then, there exists a constant $M_2 = M_2(K_6, L_3, \overline \lambda_0) > 0$ such that, on $\scrE_{1:4}(\rmx)$, for any $t \in \bbN$ with $t > M_2 t_0$,
    \begin{align*}
        d_{V} \big( Q_{\BvM, t}, \Pi_t \big)
        &\leq 
        K \epsilon_{\app, t}, \\
        d_{V} \big( \Pi_t, \Pi(\cdot \mid \bD_{1:t}) \big)
        &\leq 
        K \epsilon_{\app, t}
        +
        d_{V} \big( Q_{\BvM, t}, \Pi(\cdot \mid \bD_{1:t}) \big),
    \end{align*}   
    where $K = K(K_{1:5}, L_{1:3}, \overline \lambda_0, r) > 0$ and 
    \begin{align} \label{def:epsilon_app}
        \epsilon_{\app, t} = \epsilon_{\app, t}(t_0, p, \rmx) = \sqrt{ \dfrac{\log^2(t/t_0)(p \log t + \rmx)^2}{t} }.
    \end{align}    
\end{theorem}

Theorem \ref{thm:online_BvM} ensures that $\Pi_t$ provides valid uncertainty quantification with explicit finite-sample error rates. For a fixed $\alpha \in (0, 1)$, consider the following Wald-type confidence and credible sets:
\begin{align*}
    \widehat{C}_{1:t}(\alpha) &= \left\{ \theta \in \Theta : \big\| \bF_{t, \theta_{\star}}^{1/2} ( \theta - \hat \theta_t^{\ML} ) \big\|_2^2 \leq \chi_{p, \alpha}^2 \right\},
    \\
    \widehat{C}_{t}(\alpha) &= \left\{ \theta \in \Theta : \big\| \bOmega_{t}^{1/2} ( \theta - \theta_t ) \big\|_2^2 \leq \chi_{p, \alpha}^2 \right\},
\end{align*}
where $\chi_{p, \alpha}^2$ denotes the $(1-\alpha)$ quantile of the $\chi_p^2$ distribution. Here, $\widehat{C}_{1:t}(\alpha)$ represents the standard frequentist confidence set based on the batch MLE, while $\widehat{C}_{t}(\alpha)$ is the credible set based on the sequentially updated posterior $\Pi_t$.
As a consequence of Theorem \ref{thm:online_BvM}, the following inclusions hold on $\scrE_{1:4}(\rmx)$
\begin{align*}
  \widehat{C}_{1:t}(\alpha + \epsilon_t)
  \subset
  \widehat{C}_{t}(\alpha)
  \subset
  \widehat{C}_{1:t}(\alpha - \epsilon_t)
\end{align*}
for some $\epsilon_t \in \bbR_+$ with $\epsilon_t = O(\epsilon_{\app,t})$. Therefore, $\widehat{C}_{t}$ inherits the frequentist validity of $\widehat{C}_{1:t}$.

\subsection{Comparison with existing analyses} \label{sec:comparison}

We now compare our approach with existing analyses of second-order methods. 
Consider the following update $\widetilde \theta_{t+1} = \widetilde \theta_{t} - \bC_{t} g_{t+1}$ with $t \in \bbN$,
where $\bC_t \in \symmPD$ is $\cF_t$-measurable. A key difficulty in analyzing such methods lies in the dependence of $\bC_t$ on the past iterates $(\widetilde \theta_s)_{s < t}$. The sequence $(\bC_t)$, which serves as a data-driven ``preconditioning'' matrix, induces a circular dependence: stability of $(\widetilde \theta_t)$ requires regular behavior of $\bC_t$, while controlling $\bC_t$ in turn relies on the convergence of $(\widetilde \theta_t)$. Our analysis resolves this circular dependence, without relying on the stringent assumptions commonly imposed in existing analyses. To highlight this contrast, we first review existing approaches and their limitations.

Existing analyses typically bypass this difficulty by imposing structural conditions on $\bC_t$. 
For instance, \citet{toulis2017SGD} consider a deterministic sequence $(\bC_t)$ that converges to a fixed limit and satisfies suitable spectral bounds throughout the iteration. Extensions to random matrices \citep{leluc2023asymptotic, godichon2025adaptive} similarly assume almost sure convergence of $\bC_t$. These assumptions effectively decouple the recursion from the evolution of $(\widetilde \theta_t)$, allowing the two sequences to be analyzed independently.

However, these conditions are difficult to verify in practice, since $\bC_t$ is typically data-dependent and must be estimated jointly with $(\widetilde \theta_t)$. To address this, existing works often first establish almost sure convergence of $(\widetilde \theta_t)$ via stochastic approximation techniques, such as the Robbins--Siegmund theorem \citep{duflo1997random, neri2026quantitative}. Under additional conditions, the convergence of $(\widetilde \theta_t)$ in turn yields regular behavior of $(\bC_t)$. 

Nevertheless, this approach does not fully resolve the circular dependence. Establishing almost sure convergence of $(\widetilde \theta_t)$ itself typically requires a priori control of the eigenvalues of $(\bC_t)$. While these eigenvalue conditions are weaker than direct structural assumptions on $(\bC_t)$, they remain difficult to verify. As a result, existing works often resort to non-trivial modifications of the update rule, such as truncation schemes \citep{bercu2020efficient} or diagonal loading techniques \citep{godichon2025adaptive}, to enforce the desired spectral control.

Returning to our approach, we analyze the recursion directly through a carefully constructed inductive argument combined with an exit time analysis. This allows us to control the coupled system $(\theta_t, \bOmega_t)$ simultaneously, thereby resolving the circular dependence without relying on asymptotic arguments, almost sure spectral control, or algorithmic modifications of $\bOmega_t$. 
Consequently, we obtain a fully non-asymptotic characterization with uniform high-probability control over the entire trajectory, reflecting the intrinsic dynamics of the algorithm rather than externally imposed regularity conditions.

\subsection{Verification of assumptions} \label{sec:verification}

In this subsection, we demonstrate that the assumptions \eqref{A0}--\eqref{A5} hold under a natural statistical model. To this end, we focus on the logistic regression model. 
Let $(Y_t)_{t \in \bbN}$ and $(x_t)_{t \in \bbN}$ denote the sequences of response variables 
and $p$-dimensional covariates, respectively. We consider the logistic regression 
model with a fixed design, that is, the covariate $x_t$ is non-random for all $t \in \bbN$.

For $t \in \bbN$ and $\theta \in \Theta$, the negative log-likelihood is given by
\begin{align} \label{def:logit_loss}
    \ell_{t}(\theta) = b( x_t^{\top} \theta ) - Y_t x_t^{\top} \theta,
\end{align}
where $b(\cdot) = \log (1 + \exp(\cdot))$. Note that $b(\cdot)$ is three times differentiable with derivatives $b'$, $b''$ and $b'''$, respectively. It is well-known that $\bbE Y_t = b'(x_t^{\top} \theta_{\star})$ and $\operatorname{Var}(Y_t) = b''(x_t^{\top} \theta_{\star})$. Since
\begin{align*}
    \bbE \ell_{t}(\theta) = b( x_t^{\top} \theta ) - b'(x_t^{\top} \theta_{\star}) x_t^{\top} \theta
\end{align*}
and $b''(\cdot) > 0$, both maps $\theta \mapsto \ell_t(\theta)$ and $\theta \mapsto \bbE \ell_t(\theta)$ are strictly convex. In particular, $\bbE \ell_t(\theta)$ has the unique minimizer $\theta_{\star}$. Hence, \eqref{A0} is satisfied.

We now verify \eqref{A1}--\eqref{A5}. We first focus on \eqref{A5}, which concerns the theoretical properties of (penalized) M-estimators. The behavior of M-estimators has been extensively studied in the literature. 
In particular, \citet{sur2019modern} established the existence and uniqueness of the MLE $\hat \theta_t^{\ML}$ when $p/n$ is below a certain value as $n, p \to \infty$. \citet{ostrovskii2021finite} showed that $n \gg p$ is sufficient to ensure the convergence rate in \eqref{A5} under a self-concordant smoothness condition. Notably, the logistic regression model falls within their framework \citep[Appendix D]{ostrovskii2021finite}. 

For the penalized estimator $\hat \theta_t^{\MAP}$, \citet{lee2026online} showed that, under a normal prior $N(\theta_0, \bOmega_0^{-1})$, if $\| \bOmega_0 \|_2 = O(1)$ and $\| \bOmega_0^{1/2} (\theta_{\star} - \theta_0) \|_2^2 \lesssim p$, then $\hat \theta_t^{\MAP}$ is asymptotically indistinguishable from the MLE $\hat \theta_t^{\ML}$. Consequently, the bias induced by the prior $\Pi_0$ becomes negligible, and hence $\| \hat \theta_t^{\MAP} - \theta_{\star} \|_2$ in \eqref{A5} attains the same optimal rate as $\hat \theta_t^{\ML}$. See Section 7 of \citet{lee2026online}; see also \citet{spokoiny2017penalized} for further details.

In Section \ref{sec:assume_example}, we provide detailed proofs that \eqref{A1}--\eqref{A4} hold under appropriate conditions. For clarity, we present a streamlined summary of these results. 
Before presenting the assumptions, we introduce the following notation: for $\tau > 0$ and $t \in \bbN$, let $I_t(\tau) = \{ s \in [t] : |x_s^{\top} \theta_{\star}| \leq \tau \}$.

\ben[label=(\textbf{Ex}), ref=\textbf{Ex}]
\item
There exist constants $K_{\min}, K_{\max} > 0$ such that
\begin{align*}
    (1 \vee r) \max_{t \in \bbN} \| x_t \|_2 \leq K_{\max}, \quad
\end{align*}
and for any $t \in \bbN$ with $t \geq K_{\max} p$
\begin{align*}
    \lambda_{\min} \left( \sum_{s \in I_t(K_{\max})} x_s x_s^{\top} \right) \geq K_{\min} t, \quad 
    \lambda_{\max} \left( \sum_{s \in [t]} x_s x_s^{\top} \right) \leq K_{\max} t.
\end{align*}
Furthermore, for all $t_0 \in \bbN$ with $t_0 \geq K_{\max} (p + \rmx)$, $\bbP_{\theta_{\star}} ( \hat \theta_{t_0}^{\MAP} \in \Theta_r ) \geq 1 - e^{-\rmx}$.
\label{Ex}
\een
Under \eqref{Ex}, we verify \eqref{A1}--\eqref{A4}.

\begin{corollary} \label{coro:verification}
    For given $\rmx, r > 0$, suppose that \eqref{Ex} holds. Then, there exist a constant $K = K(K_{\min}, K_{\max}) > 0$ and an event $\scrE(\rmx)$ with $\bbP_{\theta_{\star}}(\scrE(\rmx)) \geq 1 - 3e^{-\rmx}$ such that, on $\scrE(\rmx)$, for any $t \in \bbN$ with $t > t_0$,
    \begin{align*}
        \inf_{ \{ \overline \theta_s \}_{s=1}^{t} \subset \Theta_r }  \lambda_{\min} \left( \sum_{s = 1}^{t} \nabla^2 \ell_s(\overline \theta_s) \right) &\geq K^{-1} t, \\
        \sum_{s = t_0 +1}^{t} \big\{ \ell_s(\theta_{\star}) - \ell_s(\theta_{s-1}) \big\} &\leq K \rmx, \\
        \sum_{s = t_0 +1}^{t \wedge t_{\exit}} \langle g_s, \bOmega_s^{-1} g_s \rangle &\leq K \big( p \log (t \wedge t_{\exit}) + \rmx \big),
    \end{align*}
    and for any $\theta, \theta + h, \theta_1, \theta_2 \in \Theta_r$ and $t \in \bbN$
    \begin{align*}
        &\max_{t \in \bbN} \sup_{\theta \in \Theta_r} \| \nabla^2 \ell_t(\theta) \|_2 \leq K, & 
        &\max_{t \in \bbN} \sup_{\theta, \theta' \in \Theta_r}
        \dfrac{ \| \nabla^2 \ell_t(\theta) - \nabla^2 \ell_t(\theta')  \|_2 }{ \| \theta - \theta' \|_2 }
        \leq K, \\
        &e^{-K \| h \|_2 } \bF_{t, \theta} 
        \preceq
        \bF_{t, \theta+h} 
        \preceq
        e^{K \| h \|_2 } \bF_{t, \theta}, & 
        &\nabla^{2} \ell_t (\theta_1) 
        \preceq ( 1 + K \| \theta_1 - \theta_2 \|_2 ) \nabla^{2} \ell_t (\tilde \theta)
    \end{align*}
    for some $\tilde \theta = \tilde \theta(\theta_1, \theta_2) \in \Theta_r$ satisfying
    \begin{align*}
        \ell_t(\theta_1)
        =
        \ell_t(\theta_2)
        + \langle \nabla \ell_t(\theta_2), \theta_1- \theta_2 \rangle
        + \dfrac{1}{2} \langle \nabla^2 \ell_t ( \tilde \theta ), (\theta_1- \theta_2)^{\otimes 2} \rangle.
    \end{align*}
\end{corollary}
\begin{proof}
   See the proofs of Propositions \ref{prop:A1}, \ref{prop:A2}, \ref{prop:A3-b}, \ref{prop:A5} and \ref{prop:A3-a}.
   Corollary \ref{coro:verification} follows as a special case.
\end{proof}


Although we verify \eqref{A1}--\eqref{A4} in Corollary \ref{coro:verification}, \eqref{Ex} involves a somewhat strong assumption, namely that $\sup_{t \in \bbN} \| x_t \|_2 \leq K_{\max}$. This requirement arises from controlling the smoothness parameters $L_1, L_2$, and $L_3$, which may depend on the dimension $p$. 
For example, a natural bound for $L_2$ is of the form
\begin{align*}
    \sup_{\theta \in \Theta_r} \| \nabla^2 \ell_t(\theta) \|_2 \lesssim p.
\end{align*}
Under the batch learning regime, such dimension dependence can be mitigated. As a simple illustration, suppose that the covariates $(X_t)_{t \in \bbN}$ are i.i.d. with $X_t \sim N(0, \bI_p)$ for each $t \in \bbN$. 
In this case, one can show that if $t \geq C (p + \rmx)$ for a sufficiently large constant $C > 0$, then
\begin{align*}
    \sup_{\theta \in \Theta_r} \big\| \nabla^2 L_{1:t}(\theta) \big\|_2 \lesssim t,
\end{align*}
which no longer depends on $p$; see \citet{lee2025advances} and \citet{lee2026online} for related results. However, in one-pass procedures, the smoothness of the empirical risk $\ell_t$ typically needs to be controlled at each iteration. For this reason, we impose a stronger condition in \eqref{Ex}.

A similar phenomenon has been observed in other one-pass procedures in the literature. Many previous works treat the dimension $p$ as a fixed constant \citep{chen2024online, bercu2020efficient, godichon2025adaptive, toulis2017SGD}. 
Notably, \citet{chen2020SGD} developed a non-asymptotic analysis that accounts for dimension dependence. However, for the step size $\zeta_t = \zeta_0 t^{-0.501}$, which exhibits strong empirical performance, they require $t \gg p^{7.984}$ to ensure valid uncertainty quantification; see Theorem 4.2 in \citet{chen2020SGD}. 
This requirement highlights a gap between one-pass procedures and batch learning. We conjecture that this gap can be relaxed and plan to investigate this direction in future work.

\section{Numerical experiments} \label{sec:experiments}

In this section, we present numerical experiments to complement our theoretical findings and to evaluate the empirical performance of the proposed Bayesian online one-pass (BOO) algorithm. Our primary goal is to assess both the estimation accuracy and the validity of uncertainty quantification in an online setting. 

We compare our method against several benchmark methods: first-order stochastic gradient descent (SGD), averaged SGD (ASGD), and two variants of the proposed method—BOO without warm-start ($t_0 = 0$) and BOO with $t_0 = \lceil p \log (p \vee 3) + 5 \rceil$ (corresponding to $M = 1$ and $\rmx = 5$ in \eqref{A5}). Here, the ASGD algorithm corresponds to the running average of the iterates $(\hat \theta_t^{\SGD})$ from SGD; that is, we consider the following:
\begin{align*}
    \hat \theta_t^{\ASGD} = t^{-1} \sum_{s=1}^{t} \hat \theta_s^{\SGD}.
\end{align*}
To ensure a fair comparison in terms of initialization, we additionally consider warm-started versions of SGD and ASGD, denoted by wSGD and wASGD. These estimators share the same initialization $t_0 = \lceil p \log (p \vee 3) + 5 \rceil$ as BOO with warm-start, and are initialized using the batch estimator $\hat \theta_{t_0}^{\MAP}$. Specifically, wSGD is defined by
\begin{align*}
\hat \theta_t^{\wSGD} &= \hat \theta_{t-1}^{\MAP}, & &\forall t \in \bbN \text{ with } t \leq t_0 \\
\hat \theta_t^{\wSGD} &= \hat \theta_{t-1}^{\wSGD} - \zeta_t g_t, & &\forall t \in \bbN \text{ with } t > t_0
\end{align*}
with step-size $\zeta_t \in \bbR_+$. The wASGD is defined analogously as the running average of $(\hat \theta_t^{\wSGD})$.

We consider generalized linear models, specifically logistic and Poisson regression. For the simulation, we use a total sample size $N = 10^4$ with varying dimensions. For the data generating process, we consider a sequence of observations $\bY = (Y_t)_{t \in [N]}$ whose joint probability measure is $\bbP_{\theta_{\star}}$, where $\theta_{\star} \in \mathbb{R}^p$ is a unit-norm vector defined as
$\theta_{\star} = \frac{v}{\|v\|_2}$ with $v = (v_j)_{j \in [p]}$ and $v_j = (-1)^{(j-1)} j$.
All reported results are obtained by averaging over 500 independent repetitions.
For SGD, we set the step-size $\zeta_t = \zeta_0 t^{-\alpha}$ and tune $(\zeta_0, \alpha)$. We found that the performance of SGD is highly sensitive to the choice of $(\zeta_0, \alpha)$. 
Accordingly, we fix $\alpha = 0.501$, which was empirically observed to yield strong performance, and set $\zeta_0 = 1/2$ for both logistic and Poisson regression.

We focus on two types of evaluation metrics. For estimation, we evaluate the $\ell_2$-error $\| \theta_t - \theta_{\star} \|_2$. For inference, we examine the frequentist coverage of the credible set constructed from $\Pi_N$. 
To conduct statistical inference for SGD, we adopt the \textit{plug-in} approach in \citet{chen2020SGD} utilizing $\bF_{N, \SGD} = \sum_{t=1}^{N} \nabla^2 \ell_t( \hat \theta_{t-1}^{\SGD} )$ and $\bV_{N, \SGD} = \sum_{t=1}^{N} \nabla \ell_t( \hat \theta_{t-1}^{\SGD} ) \nabla \ell_t( \hat \theta_{t-1}^{\SGD} )^{\top}$. For a confidence level $\alpha \in (0, 1)$ and coordinate index $j \in [p]$, the confidence and credible sets are given by
\begin{align*}
    \bC_j^{\SGD}(\alpha) &= \left\{ \vartheta_j : 
    \hat \theta_{N,j}^{\ASGD} - z_{\alpha/2} \cdot \hat \sigma_{j, \SGD} \leq \vartheta_j \leq 
    \hat \theta_{N,j}^{\ASGD} + z_{\alpha/2} \cdot \hat \sigma_{j, \SGD}  \right\}, \\
    \bC_j^{\BOO}(\alpha) &= \left\{ \vartheta_j : 
    \theta_{N,j} - z_{\alpha/2} \cdot \hat \sigma_{j, \BOO} \leq \vartheta_j \leq
    \theta_{N,j} + z_{\alpha/2} \cdot \hat \sigma_{j, \BOO}  \right\},
\end{align*}
where $z_{\alpha/2}$ denotes the upper $\alpha/2$ quantile of $N(0, 1)$, and $\hat \sigma_{j, \SGD}$ and $\hat \sigma_{j, \BOO}$ denote the $j$-th diagonal component of $\bF_{N, \SGD}^{-1} \bV_{N, \SGD} \bF_{N, \SGD}^{-1}$ and $\bOmega_{N}^{-1}$, respectively. 
Setting $\alpha = 0.05$, we report the coverage probability (CP) and the length (Len) of intervals in Tables \ref{tab:coverage_indep} and \ref{tab:coverage_corr}.
Finally, we also investigate the sensitivity of the algorithm to the phase transition time $t_0$ (Figure \ref{fig:sensitivity_M}) and the initial distance between $\theta_0$ and $\theta_{\star}$ (Figure \ref{fig:sensitivity_init_dist}).


\begin{figure}[t]
    \centering
    \includegraphics[width=\textwidth]{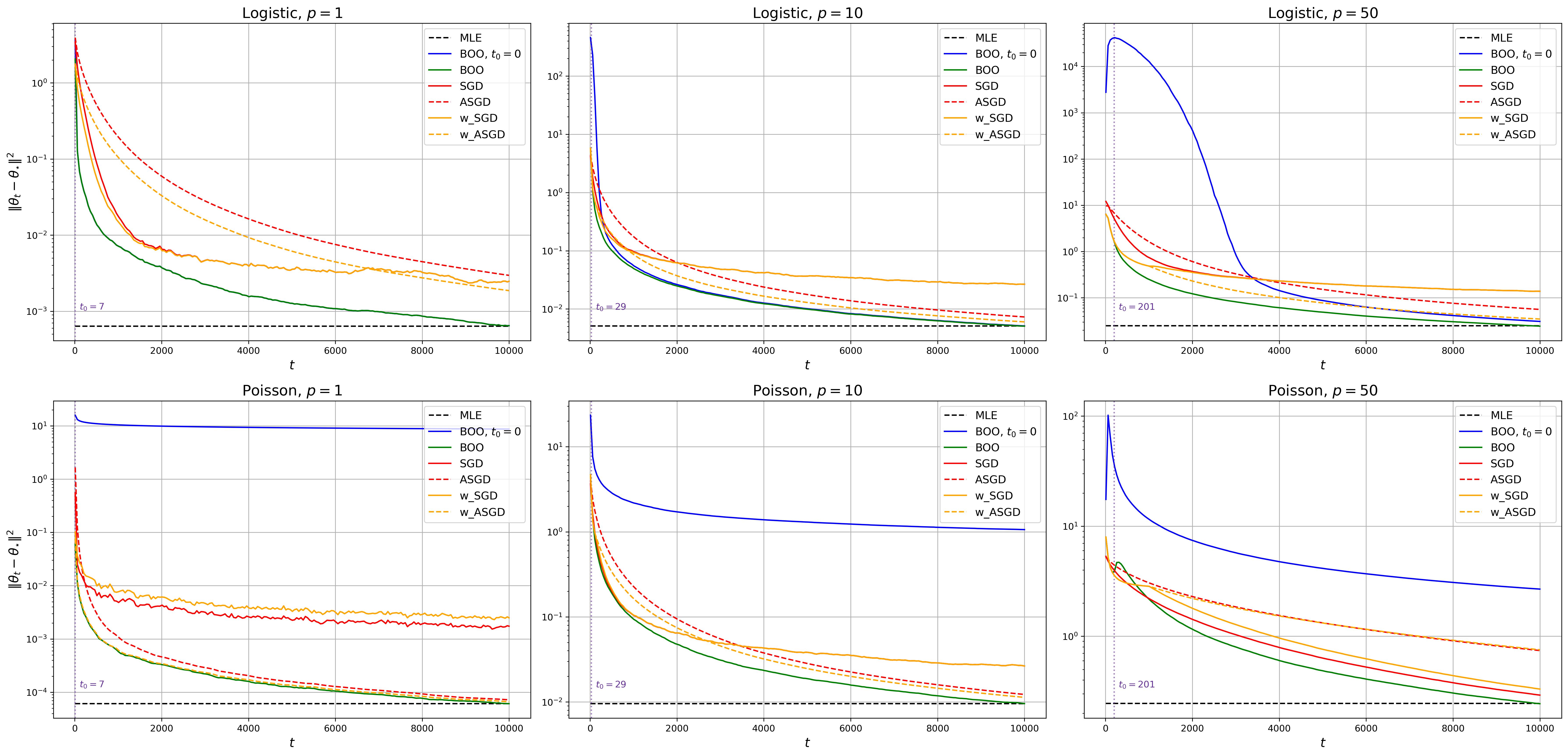}
    \caption{$\ell_2$-error across different parameter dimensions $p$.}
    \label{fig:est_varying_dim}
\end{figure}


\subsection{Estimation accuracy with varying dimensions}

We first examine the estimation accuracy of each method across varying parameter dimensions $p \in \{ 1, 10, 50 \}$. According to \eqref{A5} with $M = 1$ and $\rmx = 5$, we consider $t_0 \in \{ 7, 29, 201 \}$, respectively. Figure \ref{fig:est_varying_dim} shows the $\ell_2$-error of the estimators for both logistic and Poisson regression models. For this simulation, we generate the covariates $\bX = (X_t)_{t \in [N]}$ as $X_t \sim N(0, \bI_p)$ in the logistic model and $X_t = Z_t / \|Z_t\|_2$ with $Z_t \sim N(0, \bI_p)$ in the Poisson model, where the normalization is used to improve numerical stability.

Overall, the BOO estimator with $t_0 = \lceil p \log (p \vee 3) + 5 \rceil$ exhibits stable performance across a wide range of dimensions. In particular, it consistently outperforms SGD and ASGD in terms of $\ell_2$ error, and this remains true even when SGD and ASGD are equipped with a warm-start phase. 
The gap becomes more pronounced in higher dimensions, where the second-order structure of BOO yields more favorable scaling with respect to $p$. We also observe that the BOO estimator with $t_0 = 0$ exhibits slightly degraded performance compared to its $t_0 > 0$ counterpart, particularly in the early stage of the iteration. This effect persists and leads to worse final performance in the Poisson model. This reflects the role of the initial phase in ensuring that the iterates enter a suitable local region, as suggested in Section \ref{sec:theory}.

\subsection{Estimation and coverage under independent and correlated designs}

\begin{figure}[ht!]
    \centering
    \includegraphics[width=\textwidth]{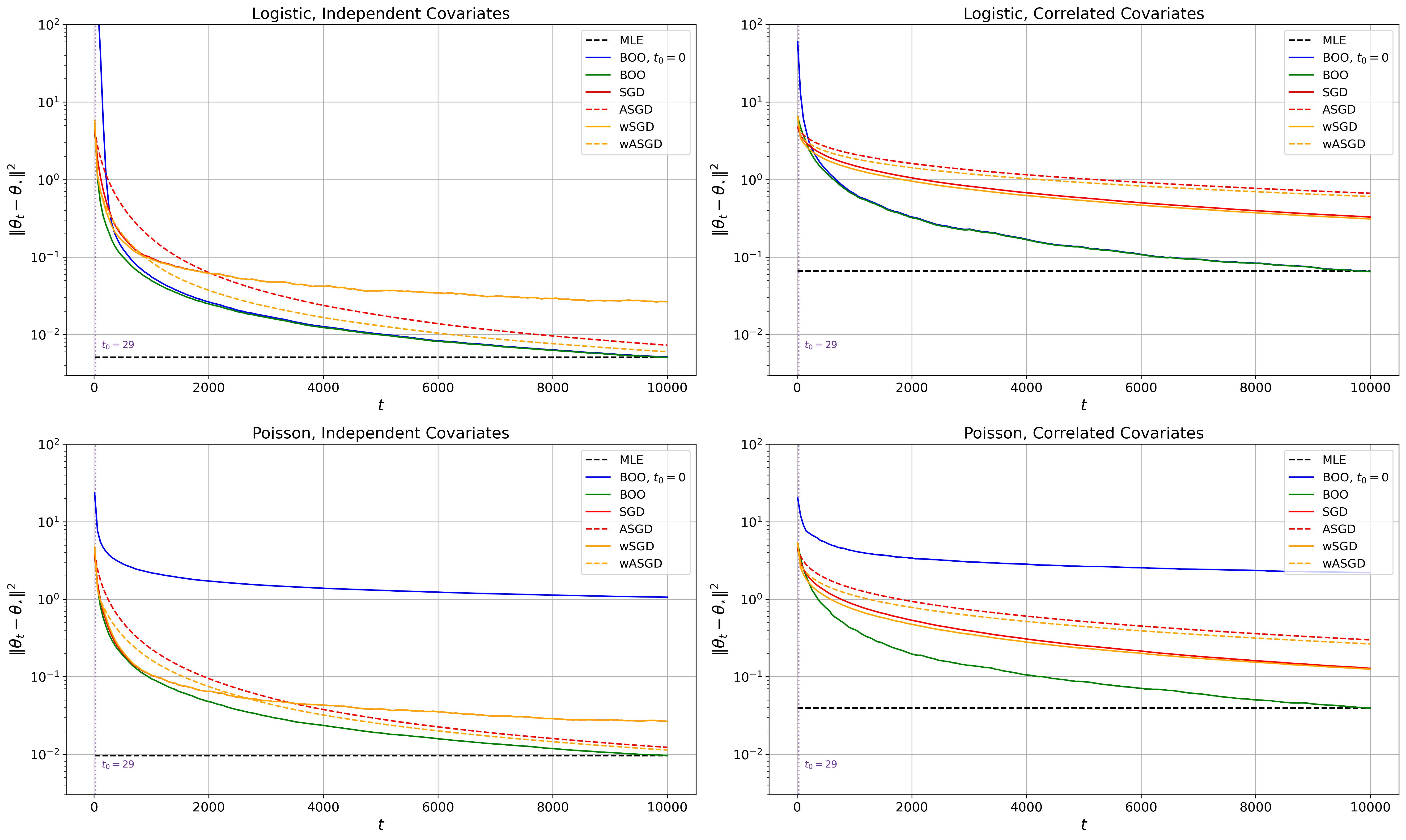}
    \caption{$\ell_2$-error across independent and correlated design.}
    \label{fig:corr}
\end{figure}


We next investigate both estimation accuracy and the validity of uncertainty quantification under different design settings. To this end, we consider two covariance settings, namely independent and correlated, with $p = 10$ for both logistic and Poisson regression models. Specifically, we generate $\bX = (X_t)_{t \in [N]}$ such that $X_t \overset{\rm i.i.d.}{\sim} N(0, \bSigma)$ in the logistic model and $X_t = Z_t / \|Z_t\|_2$ with $Z_t \overset{\rm i.i.d.}{\sim} N(0, \bSigma)$ in the Poisson model. 
For the independent design, we take $\bSigma = \bI_p$, whereas for the correlated design we set
\begin{align*}
\bSigma = \bA \operatorname{diag} \left( ( j^2 \slash p^2 )_{j=1}^{p} \right) \bA^{\top},
\end{align*}
where $\bA$ is a randomly generated orthogonal matrix. This correlated design is inspired by \citet{boyer2023asymptotic} and allows us to assess the robustness of each method with respect to the covariance structure.

We first consider point estimation performance. Under the independent design, the first-order methods---SGD and ASGD---perform slightly worse than BOO with $t_0 = 29$. In contrast, under the correlated design, their performance deteriorates substantially, indicating that first-order procedures are more sensitive to the dependence structure of the covariates. The BOO estimators exhibit a different pattern. In the logistic model, both BOO with $t_0=0$ and BOO with $t_0=29$ achieve performance comparable to that of the MLE, regardless of the dependency structure. In the Poisson model, however, only BOO with $t_0 = 29$ remains comparable to the MLE, whereas BOO with $t_0 = 0$ shows a noticeable loss of accuracy. Overall, these results indicate that the second-order structure of BOO, combined with a suitable warm-start phase, leads to substantially improved robustness with respect to the design structure.

This pattern becomes more apparent when we turn to uncertainty quantification. Under the correlated design, the coverage probability of SGD collapses, due to the substantial bias of $\hat \theta_{N}^{\ASGD}$. A similar issue arises for BOO with $t_0 = 0$ in the Poisson regression model, 
which also exhibits invalid coverage. By contrast, BOO with $t_0 = 29$ continues to deliver coverage probabilities close to those of the MLE, together with comparable interval lengths. Consequently, BOO with $t_0 = 29$ is the only online method considered here that consistently matches the benchmark performance of the MLE in both estimation and inference.

These findings support two conclusions. First, the second-order BOO method provides robustness to the design structure that is not shared by the first-order alternatives. This robustness is already visible in point estimation, and becomes even more evident in coverage performance. Second, the warm-start phase is essential for stable one-pass Bayesian updating, especially in more challenging settings such as Poisson regression.


\begin{table}[t!]
    \centering
    \caption{Coverage and CI length comparison (independent design)}
    \label{tab:coverage_indep}
    
    \setlength{\tabcolsep}{4pt}
    \renewcommand{\arraystretch}{0.9}
    \fontsize{7}{10}\selectfont
    
    \begin{tabular}{c cc| cc |cc | cc | cc | cc| cc |cc}
    \toprule
    \multirow{3}{*}{Coef.}
    & \multicolumn{8}{c|}{Logistic}
    & \multicolumn{8}{c}{Poisson} \\
    \cline{2-17}
    
    & \multicolumn{2}{c}{MLE}
    & \multicolumn{2}{c}{BOO}
    & \multicolumn{2}{c}{BOO, $t_0=0$}
    & \multicolumn{2}{c|}{SGD}
    & \multicolumn{2}{c}{MLE}
    & \multicolumn{2}{c}{BOO}
    & \multicolumn{2}{c}{BOO, $t_0=0$}
    & \multicolumn{2}{c}{SGD} \\
    
    \cline{2-17}
    & CP & Len & CP & Len & CP & Len & CP & Len
    & CP & Len & CP & Len & CP & Len & CP & Len \\
    \hline
    
    $\theta_1$  & 0.950 & 0.086 & 0.946 & 0.086 & 0.950 & 0.087 & 0.928 & 0.089
                & 0.946 & 0.121 & 0.944 & 0.121 & 0.836 & 0.119 & 0.938 & 0.122 \\
    
    $\theta_2$  & 0.954 & 0.087 & 0.954 & 0.087 & 0.956 & 0.087 & 0.926 & 0.089
                & 0.956 & 0.121 & 0.960 & 0.121 & 0.834 & 0.119 & 0.920 & 0.122 \\
    
    $\theta_3$  & 0.948 & 0.087 & 0.954 & 0.087 & 0.956 & 0.087 & 0.920 & 0.089
                & 0.946 & 0.121 & 0.940 & 0.121 & 0.820 & 0.119 & 0.912 & 0.122 \\
    
    $\theta_4$  & 0.948 & 0.087 & 0.948 & 0.087 & 0.948 & 0.088 & 0.928 & 0.089
                & 0.956 & 0.121 & 0.956 & 0.121 & 0.820 & 0.118 & 0.930 & 0.122 \\
    
    $\theta_5$  & 0.946 & 0.088 & 0.950 & 0.088 & 0.950 & 0.088 & 0.932 & 0.090
                & 0.920 & 0.121 & 0.922 & 0.121 & 0.814 & 0.118 & 0.844 & 0.122 \\
    
    $\theta_6$  & 0.964 & 0.088 & 0.972 & 0.088 & 0.966 & 0.088 & 0.880 & 0.091
                & 0.942 & 0.121 & 0.936 & 0.121 & 0.798 & 0.118 & 0.904 & 0.121 \\
    
    $\theta_7$  & 0.954 & 0.089 & 0.962 & 0.089 & 0.964 & 0.089 & 0.910 & 0.091
                & 0.958 & 0.121 & 0.960 & 0.121 & 0.842 & 0.118 & 0.928 & 0.121 \\
    
    $\theta_8$  & 0.956 & 0.089 & 0.950 & 0.089 & 0.950 & 0.090 & 0.918 & 0.092
                & 0.960 & 0.121 & 0.958 & 0.120 & 0.840 & 0.118 & 0.938 & 0.121 \\
    
    $\theta_9$  & 0.956 & 0.090 & 0.956 & 0.090 & 0.954 & 0.091 & 0.888 & 0.093
                & 0.952 & 0.120 & 0.954 & 0.120 & 0.828 & 0.118 & 0.942 & 0.121 \\
    
    $\theta_{10}$ & 0.944 & 0.091 & 0.942 & 0.091 & 0.942 & 0.092 & 0.836 & 0.094
                  & 0.952 & 0.120 & 0.946 & 0.120 & 0.832 & 0.118 & 0.942 & 0.121 \\
    
    \bottomrule
    \end{tabular}
\end{table}

\begin{table}[t!] 
    \centering 
    \caption{Coverage and CI length comparison (correlated design)} 
    \label{tab:coverage_corr}  
    \setlength{\tabcolsep}{4pt} 
    \renewcommand{\arraystretch}{0.9} 
    \fontsize{7}{10}\selectfont  
    
    \begin{tabular}{c cc|cc|cc|cc|cc|cc|cc|cc} 
    \toprule 
    \multirow{3}{*}{Coef.} 
    & \multicolumn{8}{c|}{Logistic} 
    & \multicolumn{8}{c}{Poisson} \\ 
    \cline{2-17}  
    
    & \multicolumn{2}{c}{MLE} 
    & \multicolumn{2}{c}{BOO} 
    & \multicolumn{2}{c}{BOO, $t_0=0$} 
    & \multicolumn{2}{c|}{SGD} 
    
    & \multicolumn{2}{c}{MLE} 
    & \multicolumn{2}{c}{BOO} 
    & \multicolumn{2}{c}{BOO, $t_0=0$} 
    & \multicolumn{2}{c}{SGD} \\ 
    
    \cline{2-17}  
    
    & CP & Len & CP & Len & CP & Len & CP & Len 
    & CP & Len & CP & Len & CP & Len & CP & Len \\ 
    
    \hline  
    
    $\theta_1$   & 0.970 & 0.264 & 0.972 & 0.264 & 0.972 & 0.265 & 0.008 & 0.274 
    & 0.962 & 0.212 & 0.960 & 0.212 & 0.880 & 0.208 & 0.036 & 0.212 \\  
    
    $\theta_2$   & 0.938 & 0.293 & 0.942 & 0.293 & 0.940 & 0.293 & 0.010 & 0.304 
    & 0.956 & 0.233 & 0.952 & 0.232 & 0.884 & 0.228 & 0.058 & 0.233 \\  
    
    $\theta_3$   & 0.970 & 0.200 & 0.972 & 0.201 & 0.970 & 0.201 & 0.000 & 0.212 
    & 0.950 & 0.164 & 0.946 & 0.164 & 0.886 & 0.161 & 0.002 & 0.164 \\  
    
    $\theta_4$   & 0.954 & 0.552 & 0.950 & 0.551 & 0.954 & 0.551 & 1.000 & 0.568 
    & 0.952 & 0.429 & 0.956 & 0.427 & 0.890 & 0.419 & 0.992 & 0.430 \\  
    
    $\theta_5$   & 0.946 & 0.235 & 0.956 & 0.235 & 0.952 & 0.235 & 0.000 & 0.242 
    & 0.966 & 0.189 & 0.966 & 0.188 & 0.890 & 0.185 & 0.010 & 0.190 \\  
    
    $\theta_6$   & 0.952 & 0.218 & 0.950 & 0.219 & 0.950 & 0.219 & 0.000 & 0.232 
    & 0.962 & 0.177 & 0.966 & 0.177 & 0.888 & 0.173 & 0.004 & 0.176 \\  
    
    $\theta_7$   & 0.952 & 0.420 & 0.960 & 0.420 & 0.956 & 0.420 & 0.996 & 0.433 
    & 0.956 & 0.328 & 0.958 & 0.327 & 0.890 & 0.320 & 0.914 & 0.328 \\  
    
    $\theta_8$   & 0.956 & 0.397 & 0.960 & 0.397 & 0.956 & 0.397 & 0.850 & 0.408 
    & 0.954 & 0.311 & 0.952 & 0.310 & 0.890 & 0.304 & 0.434 & 0.312 \\  
    
    $\theta_9$   & 0.962 & 0.220 & 0.966 & 0.220 & 0.964 & 0.221 & 0.298 & 0.229 
    & 0.962 & 0.177 & 0.966 & 0.177 & 0.886 & 0.174 & 0.486 & 0.177 \\  
    
    $\theta_{10}$   & 0.950 & 0.215 & 0.952 & 0.215 & 0.950 & 0.215 & 0.758 & 0.223 
    & 0.938 & 0.173 & 0.942 & 0.172 & 0.866 & 0.169 & 0.936 & 0.173 \\  
    
    \bottomrule 
    \end{tabular} 
\end{table}


\subsection{Sensitivity analysis}

We next investigate the sensitivity of BOO to key hyperparameters, focusing on $t_0$ and $\theta_0$. 
Throughout this subsection, we fix $p = 10$. These parameters play an important role in the algorithm, as they determine whether the iterates enter and remain within a suitable local region where the online updates are stable. The theoretical condition in \eqref{A5} requires $M$ to be sufficiently large, 
which may be conservative in practice. Moreover, \eqref{A5} implicitly relies on the well-posedness of the prior parameters, making it important to understand how the empirical performance depends on the choice of $\theta_0$. We therefore examine the sensitivity of BOO with respect to $M$ and $\theta_0$.

\subsubsection{Sensitivity to the phase transition time}

\begin{figure}[t]
    \centering
    \includegraphics[width=\textwidth]{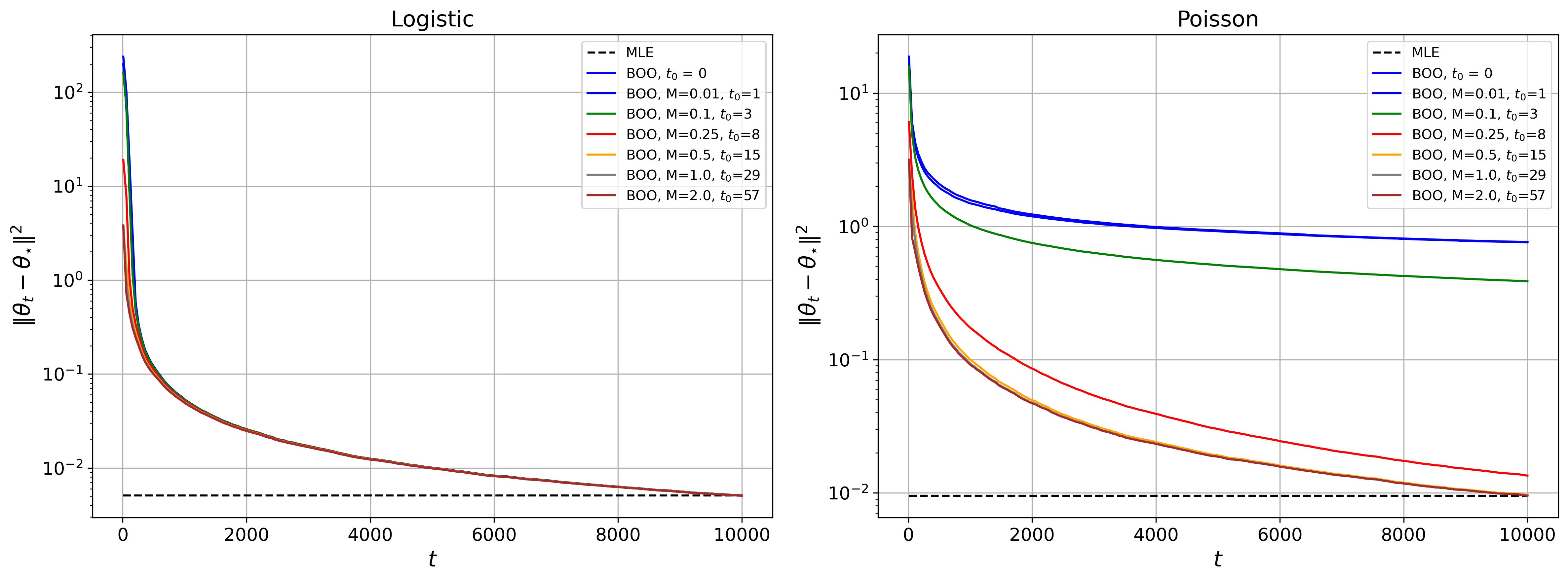}
    \caption{$\ell_2$-error across varying $M$.}
    \label{fig:sensitivity_M}
\end{figure}


We first examine the effect of $t_0$. Recall from \eqref{A5} that 
\begin{align*}
    t_0 = \lceil M( p \log (p \vee 3) + \rmx) \rceil.
\end{align*}
To assess this effect, we fix $\rmx = 5$ and vary $M$ across $\{0.01, 0.1, 0.25, 0.5, 1.0, 2.0\}$, evaluating the estimation accuracy of $\theta_t$. Although \eqref{A5} requires $M$ to be sufficiently large, we argue that $M = 1$, the value used throughout our other experiments, is sufficient for the desired empirical performance. 

Figure \ref{fig:sensitivity_M} shows that the choice of $t_0$ has a relatively mild impact once it exceeds a certain threshold. For the logistic regression model, even BOO with $t_0 = 0$ achieves performance comparable to the MLE. For the Poisson regression model, in contrast, the role of $t_0$ becomes more pronounced: BOO attains MLE-level performance in terms of $\ell_2$-error once $M$ exceeds approximately $0.25$.

\subsubsection{Sensitivity to initialization}

\begin{figure}[t]
    \centering
    \includegraphics[width=\textwidth]{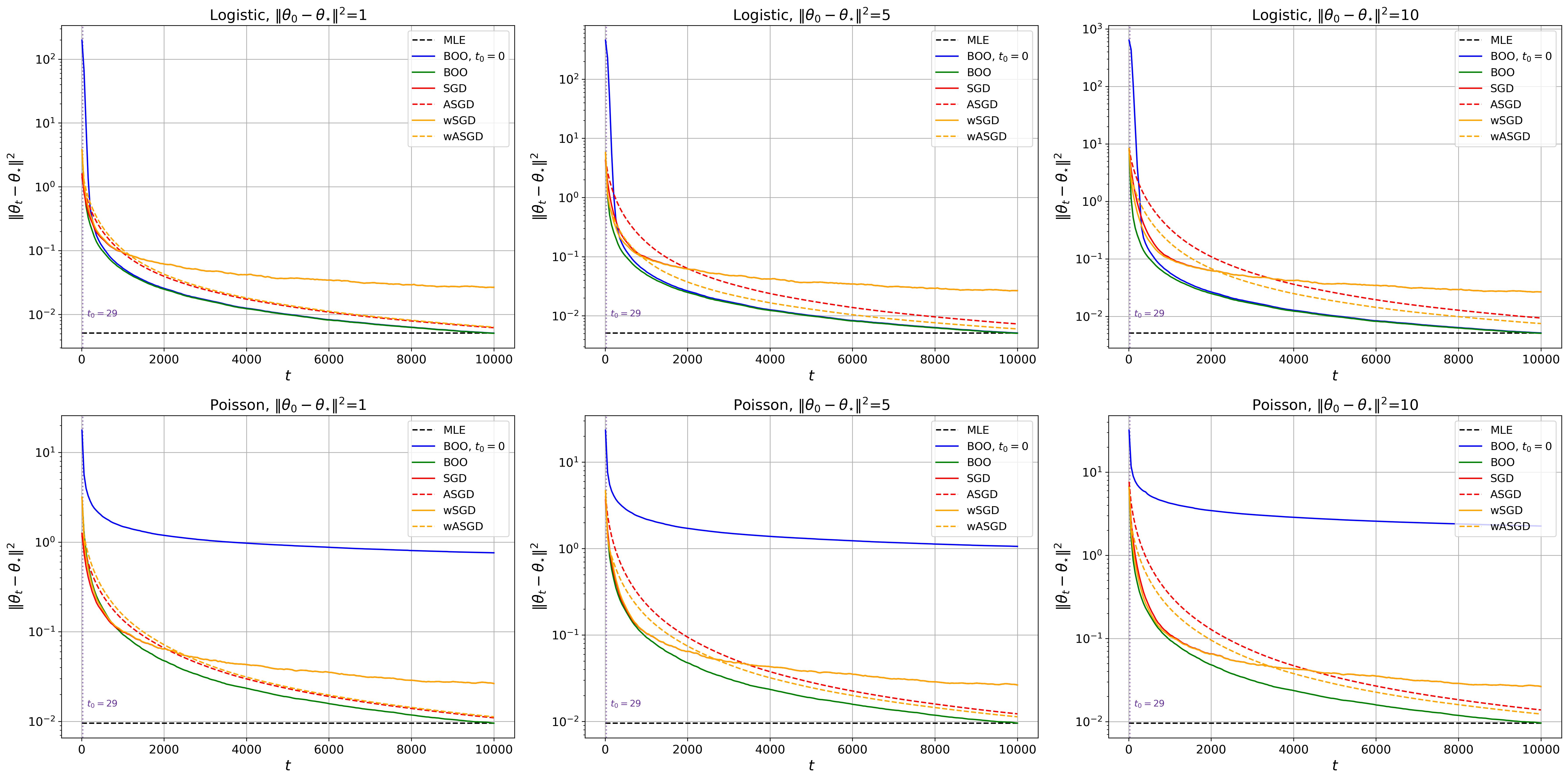}
    \caption{$\ell_2$-error across varying $\| \theta_0 - \theta_{\star} \|_2$.}
    \label{fig:sensitivity_init_dist}
\end{figure}


We next investigate the sensitivity to the initialization $\theta_0$, measured by the distance $\| \Delta_0 \|_2 = \| \theta_0 - \theta_{\star} \|_2$. Here, $\theta_0$ denotes the prior location parameter in $\Pi_0$ for BOO methods and the initial iterate for first-order methods.
In the previous experiments, we fixed $\| \Delta_0 \|_2 = \sqrt{5}$; 
here, we vary it over $\{1, \sqrt{5}, \sqrt{10}\}$ to assess the impact of initialization.

In the logistic regression model, BOO exhibits little sensitivity to the choice of $\theta_0$: 
both BOO with $t_0 = 0$ and $t_0 = 29$ show stable performance across all values of $\| \Delta_0 \|_2$. 
In contrast, the performance of ASGD deteriorates as $\| \Delta_0 \|_2$ increases. In the Poisson regression model, the behavior of BOO depends on the choice of $t_0$. BOO with $t_0 = 29$ remains stable across all values of $\| \Delta_0 \|_2$, whereas BOO with $t_0 = 0$ shows substantial degradation as $\| \Delta_0 \|_2$ increases. The performance of ASGD also deteriorates as in the logistic case.

\section{Conclusion} \label{sec:conclusion}

Despite the growing literature on Bayesian online learning, most existing methods have been guided primarily by heuristic ideas, and rigorous theoretical guarantees have remained scarce. 
To the best of our knowledge, \citet{lee2026online} provided the first theoretical justification for Bayesian online learning, showing that sequentially updated posteriors can deliver valid uncertainty quantification. However, their analysis requires the mini-batch sample size to diverge, which is typically a stringent condition in the online setting. Motivated by this limitation, we proposed in this paper a new Bayesian online learning algorithm that addresses this shortcoming, and established that it achieves the optimal convergence rate and valid uncertainty quantification under the one-pass online setting.

As a concluding remark, we offer a conjecture regarding the use of the quadratic approximation in our algorithm. Since Bayesian online learning inherently requires approximating the posterior at each step, we conjecture that, under the one-pass constraint, using the exact likelihood may in fact be theoretically \emph{less} favorable than employing an approximation. This view is empirically supported by the experiments in \citet[Section 9]{lee2026online}, and we believe it is closely tied to the geometry of posterior projection under consideration. A rigorous theoretical investigation of this phenomenon is of considerable interest, and we leave it for future work.


\addtocontents{toc}{\protect\setcounter{tocdepth}{2}}

\pagebreak
\bibliographystyle{apalike}
\bibliography{references}  
    \addtocontents{toc}{\protect\setcounter{tocdepth}{2}}
\pagebreak


\tableofcontents

\begin{appendix}

\section{Intuition and strategy of the proof} \label{sec:proof_strategy}

\subsection{Rates of convergence} \label{sec:rate_conv}

We outline the key idea of the analysis. The starting point is the following observation: the sequence $(V_t)_{t \geq t_0}$ admits a tractable recursive representation, which forms the basis of our non-asymptotic analysis.
Specifically, we establish a recursive inequality for $V_t$ of the form
\begin{align*}
    V_{t+1} \le V_t + \epsilon_t V_t^{3/2} + u_{t+1} \ \text{ on } \scrE_{1:4}(\rmx),
\end{align*}
where $\epsilon_t$ is sufficiently small to ensure stability of the recursion, while $(u_t)$ is controlled through a suitable bound on its cumulative sum. In view of the asymptotic rate of $\| \bF_{t, \theta_{\star}}^{1/2} ( \theta_{\star} - \hat \theta_t^{\ML} ) \|_2^2 = O_P(p + \log t)$ \citep{spokoiny2012parametric, lee2025advances}, we aim to show that $V_t$ increases at most logarithmically in $t$. To achieve this, we need to control the superlinear term $\epsilon_t V_t^{3/2}$ and the cumulative effect of $(u_t)$. To this end, we imposed the assumptions in Section \ref{sec:assumptions}, whose rationale we discuss throughout this subsection.

A key feature of Algorithm \eqref{Alg}, viewed as a second-order SGD method, lies in the randomness of its effective step-size. To clarify this point, consider the following standard first-order update:
\begin{align*}
    \widetilde \theta_t = \widetilde \theta_{t-1} - \eta_t g_t, \quad  t \in \bbN,
\end{align*}
where $\eta_t \in \bbR_+$ denotes the step-size. A careful choice of the sequence $(\eta_t)_{t \in \bbN}$ is crucial in online one-pass first order methods \citep[e.g.,][]{toulis2017SGD, chen2020SGD}. It is typically specified deterministically to ensure desirable statistical properties of the iterates. 
In Algorithm \eqref{Alg}, by contrast, the spectrum of $\bOmega_t^{-1}$ plays the role of an effective step-size. 
Since $\bOmega_t$ depends on the past iterates $(\theta_s)_{s=t_0}^{t-1}$, it is inherently random. 
Consequently, controlling the spectrum of $\bOmega_t$ requires uniform control over the past trajectory. This can be ensured as long as $(\theta_t)$ stays within a suitable local region such as $\Theta_r$, in which $\bOmega_t$ remains well behaved. Assumption \eqref{A1} ensures precisely this local regularity: the minimum eigenvalue of $\bOmega_t$ grows linearly in $t$, uniformly over trajectories confined to $\Theta_r$.

With the local regularity from \eqref{A1} in hand, we now return to the recursive analysis of $V_t$.
Based on Algorithm \eqref{Alg}, we can establish the following identity (see Lemma \ref{lemma:recursion}):
\begin{align} \label{eqn:basic_recursion}
    V_t = V_{t-1} 
    + \langle \bH_t, \Delta_{t-1}^{\otimes 2} \rangle
    - 2 \langle g_t, \Delta_{t-1} \rangle
    + \langle g_t, \bOmega_{t}^{-1} g_t \rangle, \quad t \in \bbN \text{ with } t > t_0.
\end{align}    
Although \eqref{eqn:basic_recursion} is exact, it is not directly amenable to analysis.
We therefore extract its essential structure, reducing \eqref{eqn:basic_recursion} to a relation that depends only on the current magnitude $V_t$, up to incremental terms whose cumulative effect can be controlled over time. The following lemma formalizes this reduction.

\begin{lemma} \label{lemma:recursion_processed}
    For given $\rmx, r > 0$, suppose that \eqref{A0}–\eqref{A2} and \eqref{A4}–\eqref{A5} hold.
    Then, on $\scrE_{1:4}(\rmx)$, we have
    \begin{align*}
        V_{t+1} \leq V_{t} + \epsilon_{t} V_{t}^{3/2} + u_{t+1}, \quad 
        \forall t \in \bbN \text{ with } t_0 \leq t < t_{\exit},
    \end{align*}
    for some non-negative sequence $(\epsilon_t)_{t \geq t_0}$ and real-valued sequence $(u_t)_{t > t_0}$ satisfying
    \begin{align*}
        \sum_{r = s}^{t} \epsilon_r &\leq K s^{-1/2}, &
        &\forall s, t \in \bbN \text{ with } t_0 \leq s < t \leq t_{\exit}, \\
        \sum_{r = t_0+1}^{t} u_r &\leq K' p \log t + K'' \big( \rmx + \log t \big), &
        &\forall t \in \bbN \text{ with } t_0 < t \leq t_{\exit},
    \end{align*}
    where $K = K(K_2, L_{1:2})$, $K' = K'(K_4)$ and $K'' = K''(K_3, K_5)$.
\end{lemma}

The refined recursion reveals two key sources of growth: the superlinear term $\epsilon_t V_t^{3/2}$, and the cumulative increment $U_t = \sum_{s=t_0+1}^t u_s$. A suitable control on the growth of $V_t$ requires balancing these two contributions.

\paragraph*{Controlling the superlinear term}
We first consider the superlinear term. To control the growth of $V_t$, the term $\epsilon_t V_t^{3/2}$ must remain negligible relative to $V_t$. Assumptions \eqref{A1}–\eqref{A2} ensure this by imposing linear growth of the eigenvalues of $\bOmega_t$ and local smoothness of $\ell_t$. These conditions yield $\epsilon_t \asymp t^{-3/2}$, so that the cumulative superlinear effect remains controlled.

The phase transition time $t_0$ plays an important role in this control as well. Beyond ensuring that the iterates enter the local region $\Theta_r$, $t_0$ must be large enough to make $\epsilon_t$ small at the onset of the recursion: when $\epsilon_t \asymp t^{-3/2}$, we have $\sum_{s=t_0}^{t} \epsilon_s \asymp t_0^{-1/2}$, so that a larger $t_0$ directly suppresses the cumulative superlinear effect. This motivates the requirement $t_0 \geq M(p \log p + \rmx)$ in \eqref{A5}, with $M > 0$ depending on the regularity constants.

\paragraph*{Controlling the cumulative increment}
We next consider the cumulative increment $U_t$. Even with $\epsilon_t$ small, $V_t$ may still diverge if $U_t$ grows too quickly: once $U_t$ pushes $V_t$ upward, the superlinear term amplifies this growth and can drive $V_t$ to grow rapidly. In the proof of Lemma \ref{lemma:recursion_processed}, $u_{t+1}$ is explicitly defined by
\begin{align*}
    u_{t+1} = 2 \big( \ell_{t+1}(\theta_{\star}) - \ell_{t+1}(\theta_{t}) \big) + \langle g_{t+1}, \bOmega_{t+1}^{-1} g_{t+1} \rangle.
\end{align*}
Accordingly, \eqref{A4} controls $U_t$ through a uniform bound on the partial sums of $(u_t)$.

\paragraph*{Induction over dyadic blocks}
Lemma \ref{lemma:recursion_processed} alone does not yield the desired bound on $V_t$.  To obtain such a bound, we employ a carefully constructed inductive argument in Proposition \ref{prop:induction}. The strategy is as follows. For each $k \in \bbN$, let $t_k = (2^k t_0) \wedge t_{\exit}$, and partition the time axis $\{ t \in \bbN : t \geq t_0 \}$ into dyadic blocks:
\begin{align*}
    \{ t_0, \ldots, t_1 \}, \quad \{ t_1, ..., t_2 \}, \quad \ldots \quad \{ t_{k_{\ast}-1}, ..., t_{k_{\ast}} \},
\end{align*}
where $k_{\ast} = \inf\{ k \in \bbN: t_{\exit} \leq 2^k t_0 \}$. We assume that $k_{\ast} > 2$ for simplicity; the general case is covered in Proposition \ref{prop:induction}.

The induction proceeds along these blocks. Suppose that $V_{t_k} \leq r(t_k, p, \rmx)$ for some $k < k_{\ast}$, where $r(t, p, \rmx)$ is of order $p \log t + \rmx$. Lemma \ref{lemma:induction_lemma} then ensures that, on $\scrE_{1:4}(\rmx)$, this bound propagates to the next block:
\begin{align*}
    V_{t} \leq r(t, p, \rmx), \quad \forall t \in \{ t_k, ..., t_{k+1} \}.
\end{align*}
Iterating this argument over successive dyadic blocks yields a uniform control of $V_t$ up to the exit time $t_{\exit}$. See Section \ref{sec:induction} for details. Although the above argument is confined to $t \leq t_{\exit}$, the resulting bound in fact implies $\theta_{t_{\exit}} \in \Theta_r$, which contradicts the definition of $t_{\exit}$. Consequently, $t_{\exit} = \infty$, and hence the argument extends to all $t \geq t_0$. This establishes Theorem \ref{thm:rate_conv}.

\subsection{Online Bernstein--von Mises theorem} \label{sec:online_BvM}

We now turn to the analysis of the online BvM theorem. Our goal is to establish the asymptotic equivalence between $\Pi_t$ and $Q_{\BvM, t}$. To this end, we introduce a useful technical bound. For $Q_1 = N(m_1, \bV_1^{-1})$ and $Q_2 = N(m_2, \bV_2^{-1})$ with sufficiently small $\| \bV_2^{-1/2} \bV_1 \bV_2^{-1/2} - \bI_p \|_{2}$, Lemma \ref{lemma:Gaussian_TV} yields
\begin{align} \label{TV_error_normals}
    d_{V}\left( Q_1, Q_2 \right) 
    \lesssim 
    \big\| \bV_1^{1/2} \big( m_1 - m_2 \big) \big\|_2 \vee \big\| \bV_2^{-1/2} \bV_1 \bV_2^{-1/2} - \bI_p \big\|_{\rm F}.
\end{align}
By \eqref{TV_error_normals}, bounding $d_V ( \Pi_t, Q_{\BvM, t} )$ reduces to proving that
\begin{align} \label{claim}
    \big\| \bF_{t, \theta_{\star}}^{1/2} \big( \theta_t - \hat \theta_t^{\ML}  \big) \big\|_{2} \vee
    \big\| \bOmega_{t}^{-1/2} \bF_{t, \theta_{\star}} \bOmega_{t}^{-1/2} - \bI_{p} \big\|_{\rm F}
    = o_P(1).
\end{align}
Propositions \ref{prop:online_location} and \ref{prop:online_precision} address the first and second term, respectively. 

We now outline the key insight underlying \eqref{claim}. Given the rates of convergence of $(\theta_t)$, the second term can be bounded using the smoothness of the map $\theta \mapsto \nabla^2 \ell_t(\theta)$. The main difficulty therefore lies in the first term. Since $\bF_{t, \theta_{\star}}$ is of order $t$, bounding the first term in \eqref{claim} is essentially equivalent to proving
\begin{align} \label{eqn:eff_rate}
    \big\| \theta_t - \hat \theta_t^{\ML} \big\|_{2} = o_{P}(t^{-1/2}).
\end{align}
Although Theorem \ref{thm:rate_conv} and \eqref{A5} imply that both estimators attain (nearly) optimal convergence rates, these rates alone yield at best $\|\theta_t - \hat \theta_t^{\ML}\|_2 = O_P(t^{-1/2})$ up to logarithmic factors, which falls short of \eqref{eqn:eff_rate}. A direct analysis of the error between $\theta_t$ and $\hat \theta_t^{\ML}$ is therefore required.

To obtain a sharp upper bound for the first term in \eqref{claim}, we show that $\theta_t$ and $\hat \theta_t^{\ML}$ are minimizers of closely related optimization problems. From this perspective, the difference between $\theta_t$ and $\hat \theta_t^{\ML}$ can be analyzed via the discrepancy between the objective functions. To formalize this idea, we introduce the following notation.
For $\theta \in \Theta$, $t \in \bbN$ with $t > t_0$, let
\begin{align} \label{def:batch_loss}
    L_{1:t}(\theta) = \sum_{r=1}^{t} \ell_r(\theta), \quad 
    \widetilde L_{t}(\theta) = \widetilde \ell_t(\theta) + \dfrac{1}{2} \big\| \bOmega_{t-1}^{1/2} (\theta - \theta_{t-1}) \big\|_2^2.
\end{align}
By construction, $\hat \theta_t^{\ML}$ and $\theta_t$ are the minimizers of $L_{1:t}$ and $\widetilde L_{t}$, respectively. 

The key step is to show that $\widetilde L_{t}(\theta)$ can be viewed as a perturbed version of $L_{1:t}(\theta)$, from which the closeness of $\hat \theta_t^{\ML}$ and $\theta_t$ follows. 
To make this precise, we define the approximation error between $\ell_t(\cdot)$ and $\widetilde \ell_t(\cdot)$ as
\begin{align} \label{def:error_loss}
    \eta_t(\theta) = \ell_t(\theta) - \widetilde \ell_t(\theta), \quad \theta \in \Theta.
\end{align}
Lemma \ref{lemma:loss_recursion} then gives
\begin{align*}
    L_{1:t}(\theta) - \widetilde L_{t}(\theta)
    =
    \sum_{s=t_0+1}^{t} \big\{ \eta_s(\theta) \big\} + L_{1:t_0}(\theta) - \dfrac{1}{2} \big\| \bOmega_{t_0}^{1/2} \big( \theta - \theta_{t_0} \big)  \big\|_2^2
    + 
    \operatorname{const.}
\end{align*}
This decomposition reveals two sources of discrepancy: the cumulative approximation error $\sum_{s=t_0+1}^{t} \eta_s(\theta)$ and the warm-start discrepancy $L_{1:t_0}(\theta) - \tfrac{1}{2} \| \bOmega_{t_0}^{1/2} (\theta - \theta_{t_0}) \|_2^2$.

Building on this decomposition, we now analyze the discrepancy between $\theta_t$ and $\hat \theta_t^{\ML}$.
Since $\hat \theta_t^{\ML}$ and $\theta_t$ are minimizers of $L_{1:t}$ and $\widetilde L_t$ respectively, we have
\begin{align*}
    \nabla L_{1:t}( \hat \theta_t^{\ML} ) = 0, \quad \nabla \widetilde L_{t}( \theta_t ) = 0,
\end{align*}
which implies that
\begin{align*}
    \nabla &L_{1:t}(\theta_t) - \nabla L_{1:t}(\hat \theta_t^{\ML} )
    = \nabla L_{1:t}(\theta_t) 
    = \nabla L_{1:t}(\theta_t) - \nabla \widetilde L_t(\theta_t) \\
    &= \sum_{s = t_0+1}^{t} \nabla \eta_s(\theta_t) + \nabla L_{1:t_0}(\theta_t) - \bOmega_{t_0} (\theta_t - \theta_{t_0}) \\
    &= \sum_{s = t_0+1}^{t} \nabla \eta_s(\theta_t)
    + \nabla L_{1:t_0}(\theta_t) - \nabla L_{1:t_0}(\hat \theta_{t_0}^{\ML}) 
    - \bOmega_{t_0} (\theta_t - \theta_{t_0}).
\end{align*}
If $\theta_t$, $\hat \theta_t^{\ML}$ and $\hat \theta_{t_0}^{\ML}$ are sufficiently close and the map $\theta \mapsto \nabla^2 L_{1:t}(\theta)$ is smooth, Taylor expansion gives
\begin{align*}
    \nabla L_{1:t}(\theta_t) - \nabla L_{1:t}(\hat \theta_t^{\ML} )
    &\approx 
    \bF_{t, \hat \theta_t^{\ML}} \big( \theta_t - \hat \theta_t^{\ML} \big), \\
    \nabla L_{1:t_0}(\theta_t) - \nabla L_{1:t_0}(\hat \theta_{t_0}^{\ML}) 
    &\approx 
    \bF_{t_0, \hat \theta_{t_0}^{\ML}} \big( \theta_t - \hat \theta_{t_0}^{\ML} \big),
\end{align*}
from which we obtain
\begin{align*}
    \Big\| \bF_{t, \hat \theta_t^{\ML}}^{1/2} (\theta_t - \hat \theta_t^{\ML}) \Big\|_2 
    &\lesssim 
    \Big\| \bF_{t, \hat \theta_t^{\ML}}^{-1/2} \sum_{s = t_0+1}^{t} \nabla \eta_s(\theta_t) \Big\|_2
    +
    \left\| \bF_{t, \hat \theta_t^{\ML}}^{-1/2} \bF_{t_0, \hat \theta_{t_0}^{\ML}} \big( \theta_t - \hat \theta_{t_0}^{\ML} \big) \right\|_2 \\
    &\qquad +
    \left\| \bF_{t, \hat \theta_t^{\ML}}^{-1/2} \bOmega_{t_0} (\theta_t - \theta_{t_0}) \right\|_2.
\end{align*}
Consequently, bounding the right-hand side of the above display is the main challenge in the analysis.

The dominant term in the above display is the first one, which captures the cumulative approximation errors. Recall from \eqref{def:tilde_loss} that $\widetilde \ell_t(\theta)$ is a quadratic approximation of $\ell_t(\theta)$. Accordingly, the gradient of the approximation error $\eta_s(\theta)$ can be written as
\begin{align*} 
    \nabla \eta_s(\theta) = (\overline \bH_s(\theta) -  \bH_s) (\theta - \theta_{s-1}),
\end{align*} 
where $\overline \bH_s(\theta) = \int_{0}^{1} \nabla^2 \ell_s ( \alpha \theta + (1-\alpha) \theta_{s-1} ) \rmd \alpha$. This representation shows that the approximation error is governed by the local variation of the hessian, and the distance between $\theta$ and $\theta_{s-1}$. Hence, if $\theta \mapsto \nabla^2 \ell_s(\theta)$ is locally smooth and $\theta_t$ remains sufficiently close to $(\theta_s)_{s=t_0}^{t-1}$, then the cumulative error term can be controlled. 
By the estimation rates in Theorem \ref{thm:rate_conv}, we have for all $s \in \bbN$ with $t_0 \leq s \leq t-1$
\begin{align*}
    \big\| \theta_t - \theta_{s} \big\|_2 \lesssim \sqrt{\dfrac{p \log s + \rmx}{s}} \quad \text{ on } \scrE_{1:4}(\rmx).
\end{align*}
Combining this with the local Lipschitz condition in \eqref{A3} yields
\begin{align*}
    \big\| \bF_{t, \hat \theta_t^{\ML}}^{1/2} (\theta_t - \hat \theta_t^{\ML}) \big\|_2
    \lesssim
    \epsilon_{\app, t}.
\end{align*}
The bound in \eqref{claim}, however, is stated in terms of $\bF_{t, \theta_{\star}}$ rather than $\bF_{t, \hat \theta_t^{\ML}}$. Under the local Lipschitz condition in \eqref{A3}, $\bF_{t, \theta_{\star}}$ and
$\bF_{t, \hat \theta_t^{\ML}}$ are close whenever $\| \hat \theta_t^{\ML} - \theta_{\star} \|_2$ is sufficiently small. Consequently, we have
\begin{align*}
    \big\| \bF_{t, \theta_{\star}}^{1/2} ( \theta_t - \hat \theta_t^{\ML}) \big\|_2
    \approx
    \big\| \bF_{t, \hat \theta_t^{\ML}}^{1/2} (\theta_t - \hat \theta_t^{\ML}) \big\|_2
    \lesssim
    \epsilon_{\app, t}.
\end{align*}
The following proposition formalizes this result.

\begin{proposition} \label{prop:online_location}
    For given $\rmx, r > 0$, suppose that \eqref{A0}–\eqref{A5} hold.
    Then, there exists a constant $M_1 = M_1(K_6) > 0$ such that, on $\scrE_{1:4}(\rmx)$, for any $t \in \bbN$ with $t > M_1 t_0$,
    \begin{align*}
    \begin{aligned} 
        \big\| \bF_{t, \theta_{\star}}^{1/2} ( \theta_t - \hat \theta_t^{\ML}) \big\|_2
        \leq 
        K \epsilon_{\app, t}.
    \end{aligned}        
    \end{align*}   
    where $K = K(K_{1:5}, L_{1:3}, \overline \lambda_0, r) > 0$.
\end{proposition} 

We next consider the second term in \eqref{claim}. Combining \eqref{A3} with the estimation rates of
$(\theta_s)_{s=t_0}^{t-1}$, we obtain on $\scrE_{1:4}(\rmx)$
\begin{align*}
    &\big\| \bOmega_t^{-1/2} \bF_{t, \theta_{\star}} \bOmega_t^{-1/2} - \bI_p  \big\|_2
    \leq 
    \lambda_{\min}^{-1}(\bOmega_t) \big\| \bF_{t, \theta_{\star}}  - \bOmega_t  \big\|_2 \\
    &\quad \leq 
    \lambda_{\min}^{-1}(\bOmega_t)
    \left[ \big\| \bOmega_0 \big\|_2 
    + \sum_{s=1}^{t_0} \left\| \big( \bH_s(\theta_{t_0}) - \bH_s(\theta_{\star}) \big) \right\|_2
    + \sum_{s=t_0+1}^{t} \left\| \big( \bH_s(\theta_{s-1}) - \bH_s(\theta_{\star}) \big) \right\|_2
    \right] \\
    &\qquad \lesssim
    t^{-1}
    \left[
    1
    + \sum_{s=1}^{t_0} \| \theta_{t_0} - \theta_{\star} \|_2
    + \sum_{s=t_0}^{t-1} \| \theta_{s} - \theta_{\star} \|_2
    \right]
    \lesssim
    \sqrt{ \dfrac{p\log t + \rmx}{t} }.
\end{align*}
The following proposition formalizes this result.

\begin{proposition} \label{prop:online_precision}
    For given $\rmx, r > 0$, suppose that \eqref{A0}–\eqref{A5} hold.
    Then, on $\scrE_{1:4}(\rmx)$, we have
    \begin{align*} 
        \big\| \bOmega_t^{-1/2} \bF_{t, \theta_{\star}} \bOmega_t^{-1/2} - \bI_p  \big\|_2 \vee
        \big\| \bF_{t, \theta_{\star}}^{-1/2} \bOmega_t \bF_{t, \theta_{\star}}^{-1/2} - \bI_p  \big\|_2
        \leq 
        K \tilde \epsilon_{\app, t}, \quad \forall t \in \bbN \text{ with } t > t_0,
    \end{align*}   
    where $K = K(K_{2:5}, L_3, \overline \lambda_0) > 0$ and
    \begin{align*}
        \tilde \epsilon_{\app, t} = \sqrt{ \dfrac{p\log t + \rmx}{t} }.
    \end{align*}
\end{proposition}

We now establish the online BvM theorem.
By Propositions \ref{prop:online_location} and \ref{prop:online_precision}, on $\scrE_{1:4}(\rmx)$, we have 
\begin{align*}
    \big\| \bF_{t, \theta_{\star}}^{1/2} \big( \theta_t - \hat \theta_t^{\ML}  \big) \big\|_{2} 
    = O(\epsilon_{\app, t}), \quad 
    \big\| \bOmega_{t}^{-1/2} \bF_{t, \theta_{\star}} \bOmega_{t}^{-1/2} - \bI_{p} \big\|_{2}
    = O(\tilde \epsilon_{\app, t}).
\end{align*}
Moreover, we have
\begin{align*}
    \big\| \bOmega_{t}^{-1/2} \bF_{t, \theta_{\star}} \bOmega_{t}^{-1/2} - \bI_{p} \big\|_{\rm F}
    \leq 
    \sqrt{p} \big\| \bOmega_{t}^{-1/2} \bF_{t, \theta_{\star}} \bOmega_{t}^{-1/2} - \bI_{p} \big\|_{2}
    \lesssim
    \sqrt{ \dfrac{p(p\log t + \rmx)}{t} }
    \lesssim
    \epsilon_{\app, t}.
\end{align*}
Combining these bounds with \eqref{TV_error_normals}, we obtain on $\scrE_{1:4}(\rmx)$
\begin{align*}
    d_{V} \big( Q_{\BvM, t}, \Pi_t \big)
    \lesssim
    \big\| \bF_{t, \theta_{\star}}^{1/2} \big( \theta_t - \hat \theta_t^{\ML}  \big) \big\|_{2} 
    +
    \big\| \bOmega_{t}^{-1/2} \bF_{t, \theta_{\star}} \bOmega_{t}^{-1/2} - \bI_{p} \big\|_{\rm F}
    \lesssim
    \epsilon_{\app, t}.
\end{align*}
Together with the classical bound on $d_{V} ( Q_{\BvM, t}, \Pi(\cdot \mid \bD_{1:t}) )$, this leads to the online BvM theorem.

\section{Proofs for Section \ref{sec:theory}}

\subsection{Rates of convergence}

\begin{lemma}[Recursion] \label{lemma:recursion}
    Suppose that \eqref{A0} holds. Then, we have
    \begin{align*}
        V_t 
        = V_{t-1} 
            + \langle \bH_t, \Delta_{t-1}^{\otimes 2} \rangle
            -2 \langle g_t, \Delta_{t-1} \rangle
            + \langle g_t, \bOmega_{t}^{-1} g_t \rangle
    \end{align*}    
    for all $t \in \bbN$ with $t > t_{0}$.
\end{lemma}
\begin{proof}
    Recall that $V_t = \| \bOmega_t^{1/2} ( \theta_t - \theta_{\star} ) \|_2^2$. 
    By the update formula $\theta_t = \theta_{t-1} - \bOmega_{t}^{-1} g_t$, we have
    \begin{align*}
        V_t 
        &= \big\| \bOmega_t^{1/2} ( \theta_t - \theta_{\star} ) \big\|_2^2 
        = \big\| \bOmega_t^{1/2} ( \theta_{t-1} - \bOmega_{t}^{-1} g_t - \theta_{\star} ) \big\|_2^2 \\
        &= 
        \big\| \bOmega_t^{1/2} ( \theta_{t-1} - \theta_{\star} ) \big\|_2^2
        +
        \big\| \bOmega_t^{1/2} \bOmega_t^{-1} g_t \big\|_2^2
        -
        2\langle \bOmega_t^{1/2} ( \theta_{t-1} - \theta_{\star} ) , \bOmega_t^{1/2} \bOmega_t^{-1} g_t \rangle \\
        &=
        \big\| \bOmega_t^{1/2} ( \theta_{t-1} - \theta_{\star} ) \big\|_2^2
        +
        \big\| \bOmega_t^{-1/2} g_t \big\|_2^2
        -
        2\langle \theta_{t-1} - \theta_{\star} , g_t \rangle \\
        &=
        V_{t-1}
        +
        \langle \bH_t, \Delta_{t-1}^{\otimes 2} \rangle
        +
        \langle g_t, \bOmega_t^{-1} g_t \rangle
        -
        2\langle g_t, \Delta_{t-1} \rangle,
    \end{align*}
    where the last equality holds by $\bOmega_t = \bOmega_{t-1} + \bH_t$.
    This completes the proof.
\end{proof}

\begin{proof}[Proof of Lemma \ref{lemma:recursion_processed}]
    Let $\rmx > 0$.
    In this proof, we will work on the event $\scrE_{1:4}(\rmx)$ without explicitly mentioning it. 
    We also assume that $M = M( K_{1:5}, L_{1:2}, r)$ specified in \eqref{A5} is sufficiently large.
    For simplicity in notations, we denote in this proof, for any $t \in \bbN$ with $t_0 \leq t \leq t_{\exit}$,
    \begin{align*}
        \overline \bH_t = \nabla^2 \ell_t( \overline \theta_{t-1} ),
    \end{align*}
    where $\overline \theta_{t-1} \in \Theta_r$ satisfies 
    \begin{align} \label{eqn:4_1}
        \ell_t(\theta_{\star})
        =
        \ell_t(\theta_{t-1})
        + \langle \nabla \ell_t(\theta_{t-1}), \theta_{\star}- \theta_{t-1} \rangle
        + \dfrac{1}{2} \langle \nabla^2 \ell_t ( \overline \theta_{t-1} ), (\theta_{\star}- \theta_{t-1})^{\otimes 2} \rangle.
    \end{align}    
    By Taylor's theorem, such $\overline \theta_{t-1}$ exists because $\Theta_r$ is convex and $\theta_{t-1}, \theta_{\star} \in \Theta_r$ for any $t \leq t_{\exit}$.
    Recall that $\bOmega_{t_0} = \bOmega_0 + \sum_{s=1}^{t_0} \nabla^2 \ell_s (\hat \theta_{t_0}^{\MAP})$.
    Since $t_0 \geq K_1 (p + \rmx)$ by \eqref{A5}, we have for any $t \in \bbN$ with $t_0 \leq t \leq t_{\exit}$
    \begin{align}
    \begin{aligned} \label{eqn:4_2}
        \lambda_{\min} \big( \bOmega_t \big)
        &\geq 
        \lambda_{\min} \big( \bOmega_0 \big)
        +
        \lambda_{\min} \left( \sum_{s=1}^{t_0} \nabla^2 \ell_s (\hat \theta_{t_0}^{\MAP}) + \sum_{s= t_0 + 1}^{t} \nabla^2 \ell_s (\theta_{s-1}) \right) \\
        &\geq 
        \lambda_{\min} \big( \bOmega_0 \big)
        +
        K_2 t
        \geq 
        K_2 t,
    \end{aligned}        
    \end{align}
    where the second inequality holds by \eqref{A1}, and $\hat \theta_{t_0}^{\MAP} \in \Theta_r$ by \eqref{A5} and $\theta_s \in \Theta_r$ for any $t_0 \leq s < t_{\exit}$.
    Recall that for any $t \in \bbN$ with $t_0 \leq t < t_{\exit}$
    \begin{align}
    \begin{aligned} \label{eqn:4_3}
        \langle \bH_{t+1}, \Delta_{t}^{\otimes 2} \rangle
        \overset{\eqref{A2}}&{\leq}
        \big( 1 + \phi_{t+1}(\theta_{\star}, \theta_{t})  \big) \langle \overline \bH_{t+1}, \Delta_{t}^{\otimes 2} \rangle \\
        \langle \overline \bH_{t+1}, \Delta_{t}^{\otimes 2} \rangle 
        \overset{\eqref{eqn:4_1}}&{=}
        \Big\{ 2\big( \ell_{t+1}(\theta_{\star}) - \ell_{t+1}(\theta_{t}) \big) 
        + 2\langle g_{t+1}, \Delta_t \rangle \Big\}, \\
        \phi_{t+1}(\theta_{\star}, \theta_{t})
        \overset{\eqref{A2}}&{\leq}
        L_1 \| \Delta_t \|_2 
        \leq L_1 \lambda_{\min}^{-1/2}(\bOmega_{t}) V_{t}^{1/2}
        \overset{\eqref{eqn:4_2}}{\leq} (L_1 K_2^{-1/2} t^{-1/2}) V_{t}^{1/2}.
    \end{aligned}        
    \end{align}    
    Let $t \in \bbN$ with $t_0 \leq t < t_{\exit}$.
    Let $D_{t+1} = \ell_{t+1}(\theta_{\star}) - \ell_{t+1}(\theta_{t})$ in this proof.
    By Lemma \ref{lemma:recursion}, we have
    \begin{align*}
        V_{t+1}
        &=
        V_{t}
        +
        \langle \bH_{t+1}, \Delta_{t}^{\otimes 2} \rangle
        +
        \langle g_{t+1}, \bOmega_{t+1}^{-1} g_{t+1} \rangle
        -
        2\langle g_{t+1}, \Delta_{t} \rangle \\
        \overset{ \eqref{eqn:4_3} }&{\leq}
        V_{t}
        +
        \big( 1 + \phi_{t+1}(\theta_{\star}, \theta_{t}) \big) \langle \overline \bH_{t+1}, \Delta_{t}^{\otimes 2} \rangle
        +
        \langle g_{t+1}, \bOmega_{t+1}^{-1} g_{t+1} \rangle
        -
        2\langle g_{t+1}, \Delta_{t} \rangle \\
        \overset{ \eqref{eqn:4_3} }&{=}
        V_{t}
        +
        \phi_{t+1}(\theta_{\star}, \theta_{t}) \langle \overline \bH_{t+1}, \Delta_{t}^{\otimes 2} \rangle
        +
        2D_{t+1} 
        +
        \langle g_{t+1}, \bOmega_{t+1}^{-1} g_{t+1} \rangle \\
        &=
        V_{t}
        +
        \phi_{t+1}(\theta_{\star}, \theta_{t}) \langle \bOmega_t^{-1/2} \overline \bH_{t+1} \bOmega_t^{-1/2}, (\bOmega_t^{1/2} \Delta_{t})^{\otimes 2} \rangle
        +
        2D_{t+1} 
        +
        \langle g_{t+1}, \bOmega_{t+1}^{-1} g_{t+1} \rangle \\
        &\leq 
        V_{t}
        +
        \phi_{t+1}(\theta_{\star}, \theta_{t}) \big\| \bOmega_t^{-1/2} \overline \bH_{t+1} \bOmega_t^{-1/2} \big\|_2
        V_t
        +
        2D_{t+1} 
        +
        \langle g_{t+1}, \bOmega_{t+1}^{-1} g_{t+1} \rangle \\
        &\leq 
        V_{t}
        +
        \big\{ (L_1 K_2^{-1/2} t^{-1/2}) V_{t}^{1/2} \big\} \big\{ L_2 (K_2 t)^{-1} \big\}
        V_t
        +
        2D_{t+1} 
        +
        \langle g_{t+1}, \bOmega_{t+1}^{-1} g_{t+1} \rangle \\
        &= 
        V_{t}
        +
        \big( L_1 L_2 K_2^{-3/2} \big) t^{-3/2} V_t^{3/2}
        +
        2D_{t+1} 
        +
        \langle g_{t+1}, \bOmega_{t+1}^{-1} g_{t+1} \rangle \\
        &= 
        V_{t}
        +
        \epsilon_{t} V_t^{3/2}
        +
        u_{t+1},
    \end{align*}
    where the last inequality holds by \eqref{A2}, \eqref{eqn:4_2} and \eqref{eqn:4_3}, and
    \begin{align*}
        \epsilon_{t} = \big( L_1 L_2 K_2^{-3/2} \big) t^{-3/2}, \quad 
        u_{t+1} = 2D_{t+1} + \langle g_{t+1}, \bOmega_{t+1}^{-1} g_{t+1} \rangle.
    \end{align*}
    For any $s, t \in \bbN$ with $t_0 \leq s < t < t_{\exit}$, we have
    \begin{align*}
        \sum_{r = s}^{t} \epsilon_r 
        &= \big( L_1 L_2 K_2^{-3/2} \big) \sum_{r = s}^{t} r^{-3/2}
        \leq \big( L_1 L_2 K_2^{-3/2} \big) \sum_{r = s}^{\infty} r^{-3/2}
        \leq \big( c_1 L_1 L_2 K_2^{-3/2} \big) s^{-1/2} \\
        &= c_2 s^{-1/2}
    \end{align*}
    where $c_1 > 0$ is a universal constant and $c_2 = c_2(L_1, L_2, K_2) > 0$.
    Also, for any $t \in \bbN$ with $t_0 < t \leq t_{\exit}$, we have
    \begin{align*}
        \sum_{r = t_0 + 1}^{t} u_r 
        &\leq \sum_{r = t_0+1}^{t} 2D_{r} + \sum_{r = t_0+1}^{t} \langle g_{r}, \bOmega_{r}^{-1} g_{r} \rangle \\
        \overset{\eqref{A4}}&{\leq}
        2K_3 \rmx + K_4 p \log t + K_5 ( \rmx + \log t ) 
        =
        K_4 p \log t + (2K_3 + K_5) ( \rmx + \log t ).
    \end{align*}
    By taking $K = c_2$, $K' = K_4$ and $K'' = 2K_3 + K_5$, we complete the proof.
\end{proof}

\begin{proof}[Proof of Theorem \ref{thm:rate_conv}]
    Let $\rmx > 0$.
    In this proof, we will work on the event $\scrE_{1:4}(\rmx)$ without explicitly mentioning it. 
    We also assume that $M = M( K_{1:5}, L_{1:2}, r)$ specified in \eqref{A5} is sufficiently large.
    
    First, we will show that $t_{\exit} = \infty$ by contradiction, where $t_{\exit}$ is specified in \eqref{def:exit_time}.
    Suppose that $t_{\exit} < \infty$.
    By the assumptions, we can apply Lemma \ref{lemma:recursion_processed}, which implies that on $\scrE_{1:4}(\rmx)$ 
    \begin{align*}
        V_{t+1} \leq V_{t} + \epsilon_{t} V_{t}^{3/2} + u_{t+1}, \quad 
        \forall t \in \bbN \text{ with } t_0 \leq t < t_{\exit},
    \end{align*}
    for some non-negative sequence $(\epsilon_t)_{t \geq t_0}$ and real-valued sequence $(u_t)_{t > t_0}$ satisfying
    \begin{align*} 
        \sum_{r = s}^{t} \epsilon_r \leq c_1 s^{-1/2}, \quad
        \sum_{r = t_0+1}^{t} u_r \leq c_2 p \log t + \widetilde c_2 \big( \rmx + \log t \big), \quad
        \forall s, t \in \bbN \text{ with } t_0 \leq s < t \leq t_{\exit},
    \end{align*}
    where $c_1 = c_1(K_2, L_{1:2}) > 0$, $c_2 = c_2(K_{4}) > 0$ and $\widetilde c_2 = \widetilde c_2(K_{3}, K_5) > 0$.
    Also, by \eqref{A5}, we have
    \begin{align*}
        V_{t_0} \leq c_3(p + \rmx),
    \end{align*}
    where $c_3 > 0$ is a universal constant. By the last three displays, all conditions in Proposition \ref{prop:induction} are satisfied with
    \begin{align*}
        X_t = V_t, \quad a_t = \epsilon_t, \quad b_t = u_t, \quad \tau_0 = t_0, \quad \tau_{\ast} = t_{\exit}, \quad (D_0, D_1, D_2, D_3) = (c_3, c_1, c_2, \widetilde c_2).
    \end{align*}    
    Hence, Proposition \ref{prop:induction} implies that
    \begin{align*}
        V_t \leq c_4 \left( p \log t + \rmx  \right), \quad 
        \forall t \in \bbN \text{ with } t_0 \leq t \leq t_{\exit},
    \end{align*}
    where $c_4 = c_4(c_2, \widetilde c_2, c_3) = c_4(K_{3:5}) > 0$.
    Also,
    \begin{align*}
        \lambda_{\min} (\bOmega_t) \geq K_2 t, \quad \forall t \in \bbN \text{ with } t_0 \leq t \leq t_{\exit},
    \end{align*}    
    where the inequality holds by \eqref{A1} and \eqref{A5}; see \eqref{eqn:4_2} for details. It follows that
    \begin{align*}
        \| \theta_t - \theta_{\star} \|_2^2 
        \leq (K_2 t)^{-1}  c_4 \left( p \log t + \rmx  \right)
        = (c_4 K_2^{-1}) \dfrac{p \log t + \rmx}{t}, \quad 
        \forall t \in \bbN \text{ with } t_0 \leq t \leq t_{\exit}.
    \end{align*}
    Since $t_0 \geq 3$ by large enough constant $M$, the map $t \mapsto t^{-1}\log t$ is non-increasing function for any $t \geq t_0$.
    Then, for $t \in \bbN$ with $t_0 \leq t \leq t_{\exit}$, we have
    \begin{align*}
        \| \theta_t - \theta_{\star} \|_2^2 
        \leq 
        (c_4 K_2^{-1}) \dfrac{p \log t + \rmx}{t}
        \leq 
        (c_4 K_2^{-1}) \dfrac{p \log t_0 + \rmx}{t_0}
        \leq 
        r^2,
    \end{align*}
    where the last inequality holds by a large enough constant $M = M(c_4, K_2) = M(K_{2:5}) > 0$. This contradicts $\theta_{t_\exit} \notin \Theta_r$. Hence, we have $t_{\exit} = \infty$ and $\theta_t \in \Theta_r$ for any $t \geq t_0$.
\end{proof}

\subsection{Online BvM}

\begin{lemma} \label{lemma:loss_recursion}
For a given $\rmx > 0$, suppose that \eqref{A0}–\eqref{A5} hold.
Then, on $\scrE_{1:4}(\rmx)$, we have
\begin{align} \label{eqn:5_first_assertion}
    L_{1:t}(\theta) - \widetilde L_t(\theta) 
    = \ell_t(\theta) - \widetilde \ell_t(\theta) + L_{1:t-1}(\theta) - \widetilde L_{t-1}(\theta) + C,
    \quad \forall t \in \bbN \text{ with } t > t_0
\end{align}
where $C$ does not depend on $\theta$.
Consequently, on $\scrE_{1:4}(\rmx)$, we have, for any $t \in \bbN$ with $t > t_0$,
\begin{align*}
    L_{1:t}(\theta) - \widetilde L_t(\theta) 
    = \sum_{s = t_0+1}^{t} \big\{ \ell_s(\theta) - \widetilde \ell_s(\theta) \big\} + L_{1:t_0}(\theta) - \dfrac{1}{2} \big\| \bOmega_{t_0}^{1/2} (\theta - \theta_{t_0}) \big\|_2^2 + K,
\end{align*}
where $K$ does not depend on $\theta$.
\end{lemma}
\begin{proof}
    Let $\rmx > 0$.
    In this proof, we will work on the event $\scrE_{1:4}(\rmx)$ without explicitly mentioning it. 
    Let $t \in \bbN$ with $t > t_0$.
    First, we will prove that
    \begin{align} \label{eqn:5_1} 
        \dfrac{1}{2} \big\| \bOmega_{t}^{1/2} (\theta - \theta_{t}) \big\|_2^2
        = 
        \widetilde L_t (\theta) + c_0,
    \end{align}
    where $c_0 = c_0(t)$ does not depend on $\theta$ (but may depend on $t$). 
    Since both sides in \eqref{eqn:5_1} are differentiable on the connected set $\Theta$, we only need to check that
    \begin{align*}
        \bOmega_{t} (\theta - \theta_{t})
        =
        \nabla \widetilde L_t (\theta), \quad \forall \theta \in \Theta.
    \end{align*}    
    Let $\theta \in \Theta$. 
    Note that
    \begin{align*}
        \nabla \widetilde L_t (\theta)
        &=
        \nabla \widetilde \ell_t (\theta) + \bOmega_{t-1} (\theta - \theta_{t-1})
        \overset{\eqref{def:tilde_loss}}{=}
        g_t + \bH_{t} (\theta - \theta_{t-1}) + \bOmega_{t-1} (\theta - \theta_{t-1}) \\
        &= 
        g_t + \bOmega_{t} (\theta - \theta_{t-1}) 
        =
        g_t + \bOmega_{t} (\theta - \theta_t + \theta_t - \theta_{t-1}) \\
        &=
        g_t + \bOmega_{t} (\theta - \theta_t) - \bOmega_{t} \bOmega_{t}^{-1} g_t
        =
        g_t + \bOmega_{t} (\theta - \theta_t) - g_t
        = 
        \bOmega_{t} (\theta - \theta_t),
    \end{align*}
    which completes the proof of \eqref{eqn:5_1}.
    Since it is easy to see that
    \begin{align*}
        L_{1:t}(\theta) = \ell_t(\theta) + L_{1:t-1}(\theta),
    \end{align*}
    we complete the proof of the first assertion \eqref{eqn:5_first_assertion}.
    By the recursive form, we have
    \begin{align*}
        L_{1:t}(\theta) - \widetilde L_t(\theta)
        &= \ell_t(\theta) - \widetilde \ell_t(\theta) + L_{1:t-1}(\theta) - \widetilde L_{t-1}(\theta) + c_1 \\
        &= \sum_{s = t_0+1}^{t} \big\{ \ell_s(\theta) - \widetilde \ell_s(\theta) \big\} + L_{1:t_0}(\theta) - \widetilde L_{t_0}(\theta) + c_2 \\
        \overset{\eqref{eqn:5_1}}&{=} \sum_{s = t_0+1}^{t} \big\{ \ell_s(\theta) - \widetilde \ell_s(\theta) \big\} + L_{1:t_0}(\theta) - \dfrac{1}{2} \big\| \bOmega_{t_0}^{1/2} (\theta - \theta_{t_0}) \big\|_2^2 + c_3,
    \end{align*}
    where $c_1, c_2, c_3$ are independent of $\theta$. This completes the proof.
\end{proof}

\begin{theorem} \label{thm:online_cumulative_error}
    For given $\rmx, r > 0$, suppose that \eqref{A0}–\eqref{A5} hold.
    Then, there exists a constant $M_1 = M_1(K_6) > 0$ such that, on $\scrE_{1:4}(\rmx)$,
    \begin{align*}
        \big\| \bF_{t, \hat \theta_t^{\ML}}^{1/2} (\theta_t - \hat \theta_t^{\ML}) \big\|_2
        \leq 
        K \epsilon_{\app, t}, \quad \forall t \in \bbN \text{ with } t > M_1 t_0
    \end{align*}   
    where $K = K(K_{1:5}, L_{1:3}, \overline \lambda_0, r) > 0$.
\end{theorem}
\begin{proof}
    Let $\rmx > 0$.
    In this proof, we will work on the event $\scrE_{1:4}(\rmx)$ without explicitly mentioning it. 
    Also, we may assume that $M = M( K_{1:5}, L_{1:2}, r)$ and $M_1 = M_1(K_6)$ are large enough constants.
    Let $t \in \bbN$ with $t > M_1 t_0$.  
    For $\theta \in \Theta$, let 
    \begin{align*}
        \eta_t(\theta) 
        = \ell_t(\theta) - \widetilde \ell_t(\theta)
        = \ell_t(\theta) 
        - \Big[ \ell_t(\theta_{t-1}) + \langle g_t, \theta - \theta_{t-1}  \rangle + \dfrac{1}{2} \langle \bH_t, (\theta - \theta_{t-1})^{\otimes 2}  \rangle \Big].
    \end{align*}
    By Lemma \ref{lemma:loss_recursion}, we have
    \begin{align*} 
        L_{1:t}(\theta) - \widetilde L_t(\theta) 
        = \sum_{s = t_0+1}^{t} \eta_s(\theta) + L_{1:t_0}(\theta) - \dfrac{1}{2} \big\| \bOmega_{t_0}^{1/2} (\theta - \theta_{t_0}) \big\|_2^2 + {\rm const}.
    \end{align*}
    It follows that
    \begin{align} \label{eqn:6_1} 
        \nabla L_{1:t}(\theta) - \nabla \widetilde L_t(\theta) 
        =
        \sum_{s = t_0+1}^{t} \nabla \eta_s(\theta) + \nabla L_{1:t_0}(\theta) - \bOmega_{t_0} (\theta - \theta_{t_0}).
    \end{align}
    By the first-order optimality, we have
    \begin{align} \label{eqn:6_2} 
        \nabla \widetilde L_t(\theta_t) = 0, \quad 
        \nabla L_{1:t}(\hat \theta_{t}^{\ML}) = 0, \quad 
        \nabla L_{1:t_0}(\hat \theta_{t_0}^{\ML}) = 0.
    \end{align}
    Also, by Taylor's theorem, we have
    \begin{align} \label{eqn:6_3} 
        \nabla \eta_t(\theta) = (\overline \bH_t(\theta) -  \bH_t) (\theta - \theta_{t-1}),
    \end{align} 
    where $\overline \bH_t(\theta) = \int_{0}^{1} \nabla^2 \ell_t ( s \theta + (1-s) \theta_{t-1} ) \rmd s$.
    Combining \eqref{eqn:6_1}, \eqref{eqn:6_2} and \eqref{eqn:6_3}, we have
    \begin{align*}
        \nabla &L_{1:t}(\theta_t) - \nabla L_{1:t}(\hat \theta_t^{\ML} )
        = \nabla L_{1:t}(\theta_t) 
        = \nabla L_{1:t}(\theta_t) - \nabla \widetilde L_t(\theta_t) \\
        &= \sum_{s = t_0+1}^{t} \nabla \eta_s(\theta_t) + \nabla L_{1:t_0}(\theta_t) - \bOmega_{t_0} (\theta_t - \theta_{t_0}) \\
        &= \sum_{s = t_0+1}^{t} (\overline \bH_s(\theta_t) -  \bH_s) (\theta_t - \theta_{s-1}) 
        + \nabla L_{1:t_0}(\theta_t) - \nabla L_{1:t_0}(\hat \theta_{t_0}^{\ML}) 
        - \bOmega_{t_0} (\theta_t - \theta_{t_0}).
    \end{align*}
    It follows that
    \begin{align*}
        &\bF_{t, \hat \theta_t^{\ML}}^{-1/2} \left( \nabla L_{1:t}(\theta_t) - \nabla L_{1:t}( \hat \theta_t^{\ML}) \right) \\
        &\quad = \bF_{t, \hat \theta_t^{\ML}}^{-1/2} \sum_{s = t_0+1}^{t} \big( \overline \bH_s(\theta_t) -  \bH_s \big) (\theta_t - \theta_{s-1}) 
        + \bF_{t, \hat \theta_t^{\ML}}^{-1/2} \big(\nabla L_{1:t_0}(\theta_t) - \nabla L_{1:t_0}( \hat \theta_{t_0}^{\ML} ) \big) \\
        &\qquad - \bF_{t, \hat \theta_t^{\ML}}^{-1/2} \bOmega_{t_0} (\theta_t - \theta_{t_0})
    \end{align*}
    where $\bF_{t, \hat \theta_t^{\ML}} = \nabla^2 L_{1:t}\big( \hat \theta_t^{\ML} \big)$.
    Then, we have
    \begin{align*}
        ({\rm i})
        \leq 
        ({\rm ii}) + ({\rm iii})+ ({\rm iv}), 
    \end{align*}
    where
    \begin{align*}
        ({\rm i}) &= \left\| \bF_{t, \hat \theta_t^{\ML}}^{-1/2} \big( \nabla L_{1:t}(\theta_t) - \nabla L_{1:t}(\hat \theta_t^{\ML}) \big) \right\|_2, \\
        ({\rm ii}) &= \left\| \bF_{t, \hat \theta_t^{\ML}}^{-1/2}\sum_{s = t_0+1}^{t} \big( \overline \bH_s(\theta_t) -  \bH_s \big) (\theta_t - \theta_{s-1}) \right\|_2, \\
        ({\rm iii}) &= \Big\| \bF_{t, \hat \theta_t^{\ML}}^{-1/2} \big(\nabla L_{1:t_0}(\theta_t) - \nabla L_{1:t_0}( \hat \theta_{t_0}^{\ML} ) \big) \Big\|_2, \\
        ({\rm iv}) &= \Big\| \bF_{t, \hat \theta_t^{\ML}}^{-1/2}\bOmega_{t_0} (\theta_t - \theta_{t_0}) \Big\|_2. 
    \end{align*}
    We will obtain a lower bound of $({\rm i})$ and upper bounds of $({\rm ii})$, $({\rm iii})$ and $({\rm iv})$.
    
    \noindent \textbf{Step 1:} $({\rm i})$ \\
    Note that
    \begin{align} \label{eqn:6_4}
        \nabla L_{1:t}(\theta_t) - \nabla L_{1:t}( \hat \theta_t^{\ML} )
        = 
        \left( \int_{0}^{1} \nabla^2 L_{1:t}\big( s \theta_t + (1-s) \hat \theta_t^{\ML} \big) \rmd s \right)
        (\theta_t - \hat \theta_t^{\ML} ).
    \end{align}
    By Theorem \ref{thm:rate_conv} and \eqref{A5}, we have
    \begin{align} \label{eqn:6_4_2}
        \| \theta_t - \theta_{\star} \|_2 \leq c_1 \sqrt{\dfrac{p \log t + \rmx}{t}}, \quad 
        \| \hat \theta_t^{\ML} - \theta_{\star} \|_2 \leq c_2 \sqrt{\dfrac{p + \log t + \rmx}{t}},
    \end{align}
    where $c_1 = c_1(K_{2:5}) > 0$ and $c_2$ is a universal constant. It follows that
    \begin{align} \label{eqn:6_4_3}
        \| \theta_t - \hat \theta_t^{\ML} \|_2 \leq c_3 \sqrt{\dfrac{p \log t + \rmx}{t}} \leq K_6^{-1} \log 2,
    \end{align}
    where $c_3 = c_3(K_{2:5}) > 0$, and the last inequality holds by
    \begin{align*}
        t > M_1 t_0 \overset{\eqref{A5}}{\geq} M_1 M (p\log (p \vee 3) + \rmx)
    \end{align*}
    with large enough $M_1 = M_1(K_6)$ and $M = M(c_3) =  M(K_{2:5})$. Hence, by \eqref{A3}, we have
    \begin{align*}
        \bF_{t, \hat \theta_t^{\ML}}^{-1/2}
        \left( \int_{0}^{1} \nabla^2 L_{1:t}\big( s \theta_t + (1-s) \hat \theta_t^{\ML} \big) \rmd s \right)
        \bF_{t, \hat \theta_t^{\ML}}^{-1/2}
        &\succeq
        \bF_{t, \hat \theta_t^{\ML}}^{-1/2}
        \left( \dfrac{1}{3} \nabla^2 L_{1:t}\big( \hat \theta_t^{\ML} \big) \right)
        \bF_{t, \hat \theta_t^{\ML}}^{-1/2} \\
        &= \dfrac{1}{3} \bI_p,
    \end{align*}
    which, combining with \eqref{eqn:6_4}, implies that
    \begin{align} \label{eqn:6_5}
        \Big\| \bF_{t, \hat \theta_t^{\ML}}^{-1/2} \big(\nabla L_{1:t}(\theta_t) - \nabla L_{1:t}(\hat \theta_t^{\ML}) \big) \Big\|_2
        \geq 
        \dfrac{1}{3} \big\| \bF_{t, \hat \theta_t^{\ML}}^{1/2} (\theta_t - \hat \theta_t^{\ML}) \big\|_2.
    \end{align}
    
    \noindent \textbf{Step 2:} $({\rm ii})$ \\
    Recall that $\overline \bH_s(\theta) = \int_{0}^{1} \nabla^2 \ell_s ( \alpha \theta + (1-\alpha) \theta_{s-1} ) \rmd \alpha$.
    Note that
    \begin{align*}
        &\left\| \sum_{s = t_0+1}^{t} (\overline \bH_s(\theta_t) -  \bH_s) (\theta_t - \theta_{s-1}) \right\|_2
        \leq 
        \sum_{s = t_0+1}^{t} \left\| (\overline \bH_s(\theta_t) -  \bH_s) (\theta_t - \theta_{s-1}) \right\|_2 \\
        &\qquad \leq  
        \sum_{s = t_0+1}^{t} \big\| \overline \bH_s(\theta_t) -  \bH_s \big\|_2 \left\| \theta_t - \theta_{s-1} \right\|_2 \\
        &\qquad \leq  
        \sum_{s = t_0+1}^{t}
        \int_{0}^{1} \left\| \nabla^2 \ell_s \big( \overline \theta_{s-1}(t, \alpha) \big) - \nabla^2 \ell_s \big( \theta_{s-1} \big) \right\|_2 \rmd \alpha \left\| \theta_t - \theta_{s-1} \right\|_2 \\
        &\qquad \leq  
        \sum_{s = t_0+1}^{t}
        \int_{0}^{1} 
        L_3 \big\| \overline \theta_{s-1}(t, \alpha) - \theta_{s-1} \big\|_2 \rmd \alpha \left\| \theta_t - \theta_{s-1} \right\|_2 \\
        &\qquad \leq  
        \sum_{s = t_0+1}^{t}
        \int_{0}^{1} 
        L_3 \left\| \theta_t - \theta_{s-1} \right\|_2 \rmd \alpha \left\| \theta_t - \theta_{s-1} \right\|_2 \\
        &\qquad = 
        \sum_{s = t_0+1}^{t} L_3 \left\| \theta_t - \theta_{s-1} \right\|_2^2,
    \end{align*}
    where the last two inequalities hold by \eqref{A3} and $\overline \theta_{s-1}(t, \alpha) = \alpha \theta_t + (1- \alpha) \theta_{s-1}$ for $\alpha \in [0, 1]$. Also,
    \begin{align*}
        \left\| \theta_t - \theta_{s-1} \right\|_2^2
        &\leq
        2\big( \left\| \theta_t - \theta_{\star} \right\|_2^2 + \left\| \theta_{s-1} - \theta_{\star} \right\|_2^2 \big) \\
        &\leq
        2 \bigg( c_1^2 \cdot \dfrac{p \log t + \rmx}{t} + c_1^2 \cdot \dfrac{p\log (s-1) + \rmx}{s-1} \bigg) \\
        &=
        2c_1^2 \bigg( \dfrac{p \log t + \rmx}{t} + \dfrac{p\log (s-1) + \rmx}{s-1} \bigg).
    \end{align*}
    It follows that
    \begin{align*}
        &\left\| \sum_{s = t_0+1}^{t} (\overline \bH_s(\theta_t) -  \bH_s) (\theta_t - \theta_{s-1}) \right\|_2
        \leq 
        2 L_3 c_1^2 \sum_{s = t_0+1}^{t} \bigg( \dfrac{p \log t + \rmx}{t} + \dfrac{p\log (s-1) + \rmx}{s-1} \bigg) \\
        &\qquad \leq 
        2 L_3 c_1^2 \bigg[ \left( \dfrac{t - t_0}{t} (p \log t + \rmx)  \right)
        + \sum_{s = t_0}^{t-1} \bigg( \dfrac{p\log s + \rmx}{s} \bigg) \bigg] \\
        &\qquad \leq 
        c_4 \bigg[ \left( \dfrac{t - t_0}{t} (p \log t + \rmx)  \right)
        + \log(t/t_0) (p\log t + \rmx) \bigg] \\
        &\qquad \leq 
        c_4 \bigg[ (p \log t + \rmx) + \log(t/t_0) (p\log t + \rmx) \bigg] 
        \leq 
        c_5 \bigg[ \log(t/t_0)(p \log t + \rmx) \bigg],
    \end{align*}
    where $c_4 = c_4(c_1, L_3) = c_4(K_{2:5}, L_3) > 0$ and $c_5 = c_5(c_4) = c_5(K_{2:5}, L_3) > 0$.
    Also, by applying \eqref{A1} with the choice $\overline \theta_s = \hat \theta_t^{\ML}$ which lies in $\Theta_r$ by \eqref{A5}, we have $\lambda_{\min} (\bF_{t, \hat \theta_t^{\ML}}) \geq K_2 t$.
    Consequently, we have
    \begin{align} \label{eqn:6_6}
    \begin{aligned}
        &\left\| \bF_{t, \hat \theta_t^{\ML}}^{-1/2} \sum_{s = t_0+1}^{t} (\overline \bH_s(\theta_t) -  \bH_s) (\theta_t - \theta_{s-1}) \right\|_2 \\
        &\qquad \leq 
        \lambda_{\min}^{-1/2} \big( \bF_{t, \hat \theta_t^{\ML}} \big)
        \left\| \sum_{s = t_0+1}^{t} (\overline \bH_s(\theta_t) -  \bH_s) (\theta_t - \theta_{s-1}) \right\|_2 \\
        &\qquad \leq 
        (K_2 t)^{-1/2} c_5 \bigg[ \log(t/t_0)(p \log t + \rmx) \bigg]
        =
        c_6 \bigg[ t^{-1/2} \log(t/t_0)(p \log t + \rmx) \bigg],
    \end{aligned}        
    \end{align}
    where $c_6 = c_6(c_5, K_2) = c_6(K_{2:5}, L_3) = K_2^{-1/2} c_5$.

    \noindent \textbf{Step 3:} $({\rm iii})$ \\
    Note that
    \begin{align*} 
        \nabla L_{1:t_0}(\theta_t) - \nabla L_{1:t_0}(\hat \theta_{t_0}^{\ML} )
        = 
        \left( \int_{0}^{1} \nabla^2 L_{1:t_0}\big( s \theta_t + (1-s) \hat \theta_{t_0}^{\ML} \big) \rmd s \right)
        (\theta_t - \hat \theta_{t_0}^{\ML}).
    \end{align*}
    As in \eqref{eqn:6_4_3}, we have
    \begin{align*}
        \| \theta_t - \hat \theta_{t_0}^{\ML} \|_2 \leq c_3 \sqrt{\dfrac{p \log t_0 + \rmx}{t_0}}.
    \end{align*}
    Also, by $\theta_t, \hat \theta_{t_0}^{\ML} \in \Theta_r$ and \eqref{A2}, we have
    \begin{align*}
        \left\| \int_{0}^{1} \nabla^2 L_{1:t_0}\big( s \theta_t + (1-s) \hat \theta_{t_0}^{\ML} \big) \rmd s \right\|_2
        &\leq 
        \sup_{s \in [0, 1]} \left\| \nabla^2 L_{1:t_0}\big( s \theta_t + (1-s) \hat \theta_{t_0}^{\ML} \big) \right\|_2 \\
        &\leq 
        L_2 t_0.
    \end{align*}
    It follows that
    \begin{align} 
    \begin{aligned} \label{eqn:6_7}
        &\Big\| \bF_{t, \hat \theta_t^{\ML}}^{-1/2} \big(\nabla L_{1:t_0}(\theta_t) - \nabla L_{1:t_0}(\hat \theta_{t_0}^{\ML}) \big) \Big\|_2 \\
        &\qquad \leq 
        \lambda_{\min}^{-1/2} \big( \bF_{t, \hat \theta_t^{\ML}} \big)
        \Big\| \nabla L_{1:t_0}(\theta_t) - \nabla L_{1:t_0}(\hat \theta_{t_0}^{\ML}) \Big\|_2 \\
        &\qquad \leq 
        \lambda_{\min}^{-1/2} \big( \bF_{t, \hat \theta_t^{\ML}} \big)
        \left\| \int_{0}^{1} \nabla^2 L_{1:t_0}\big( s \theta_t + (1-s) \hat \theta_{t_0}^{\ML} \big) \rmd s \right\|_2
        \| \theta_t - \hat \theta_{t_0}^{\ML} \|_2 \\
        &\qquad \leq 
        \big( K_2 t \big)^{-1/2}
        \big( L_2 t_0 \big)
        \left( c_3 \sqrt{\dfrac{p \log t_0 + \rmx}{t_0}} \right)
        =
        c_7 \sqrt{ \dfrac{t_0 (p \log t_0 + \rmx)}{t} },
    \end{aligned}        
    \end{align}
    where $c_7 = c_7(c_3, K_2, L_2) = c_7(K_{2:5}, L_2) = K_2^{-1/2} L_2 c_3$.

    \noindent \textbf{Step 4:} $({\rm iv})$ \\
    Note that
    \begin{align*}
        \| \bOmega_{t_0} \|_2 \overset{\eqref{A2}}{\leq} L_2 t_0 + \overline \lambda_0, \quad 
        \lambda_{\min} \big( \bOmega_t \big) \overset{\eqref{A1}}{\geq} K_2 t + \underline \lambda_0,
    \end{align*}    
    where the inequalities holds by $\theta_t, \theta_{t_0} \in \Theta_r$, \eqref{A1} and \eqref{A2}.
    Also,
    \begin{align*}
        &\big\| \bOmega_{t_0} (\theta_t - \theta_{t_0}) \big\|_2 
        \leq 
        \big\| \bOmega_{t_0} (\theta_t - \theta_{\star}) \big\|_2 
        +
        \big\| \bOmega_{t_0} (\theta_{t_0} - \theta_{\star}) \big\|_2  \\
        &\qquad =
        \big\| \bOmega_{t_0} \bOmega_{t}^{-1/2} \bOmega_{t}^{1/2} (\theta_t - \theta_{\star}) \big\|_2 
        +
        \big\| \bOmega_{t_0}^{1/2} \bOmega_{t_0}^{1/2} (\theta_{t_0} - \theta_{\star}) \big\|_2 \\
        &\qquad \leq 
        \big\| \bOmega_{t_0} \bOmega_{t}^{-1/2} \big\|_2
        \big\| \bOmega_{t}^{1/2} (\theta_t - \theta_{\star}) \big\|_2 
        +
        \big\| \bOmega_{t_0}^{1/2} \big\|_2
        \big\| \bOmega_{t_0}^{1/2} (\theta_{t_0} - \theta_{\star}) \big\|_2 \\
        &\qquad \leq
        \big( L_2 t_0 + \overline \lambda_0 \big)
        \big( K_2 t + \underline \lambda_0  \big)^{-1/2}
        c_8 \big( p \log t + \rmx \big)^{1/2}
        +
        \big( L_2 t_0 + \overline \lambda_0 \big)^{1/2}
        c_8 \big( p + \rmx \big)^{1/2} \\
        &\qquad \leq 
        \big( L_2 t_0 + \overline \lambda_0 \big)
        \big( K_2 t \big)^{-1/2}
        c_8 \big( p \log t + \rmx \big)^{1/2}
        +
        \big( L_2 t_0 + \overline \lambda_0 \big)^{1/2}
        c_8 \big( p + \rmx \big)^{1/2}
    \end{align*}
    where $c_8 = c_8(K_{3:5}) > 0$, and the third inequality holds by Theorem \ref{thm:rate_conv}, \eqref{A5} and the previous display. It follows that
    \begin{align}
    \begin{aligned} \label{eqn:6_8}
        &\Big\| \bF_{t, \hat \theta_t^{\ML}}^{-1/2}\bOmega_{t_0} (\theta_t - \theta_{t_0}) \Big\|_2 
        \leq 
        \lambda_{\min}^{-1/2} \big( \bF_{t, \hat \theta_t^{\ML}} \big)
        \big\| \bOmega_{t_0} (\theta_t - \theta_{t_0}) \big\|_2 \\
        &\quad \leq 
        \big( K_2 t \big)^{-1/2} 
        \big\| \bOmega_{t_0} (\theta_t - \theta_{t_0}) \big\|_2 \\
        &\quad \leq
        \big( L_2 t_0 + \overline \lambda_0 \big)
        \big( K_2 t \big)^{-1}
        c_8 \big( p \log t + \rmx \big)^{1/2} \\
        &\qquad +
        \big( L_2 t_0 + \overline \lambda_0 \big)^{1/2}
        \big( K_2 t \big)^{-1/2}
        c_8 \big( p + \rmx \big)^{1/2} \\
        &\quad \leq
        \big( L_2 t_0 + \overline \lambda_0 \big)
        \big( K_2 t \big)^{-1}
        c_8 \big( p \log t + \rmx \big)^{1/2} \\
        &\qquad +
        \big( L_2 t_0 + \overline \lambda_0 \big)^{1/2}
        \big( K_2 t \big)^{-1/2}
        c_8 \big( p\log t + \rmx \big)^{1/2} \\
        &\quad \leq
        c_8 \bigg( \big( L_2 t_0 + \overline \lambda_0 \big) K_2^{-1} t^{-1/2} 
        + \big( L_2 t_0 + \overline \lambda_0 \big)^{1/2} K_{2}^{-1/2} \bigg)
        \sqrt{\dfrac{p \log t + \rmx}{t}} \\
        &\quad \leq
        c_9 \big( \sqrt{t_0/t} + 1 \big)
        \sqrt{\dfrac{t_0(p \log t + \rmx)}{t}}
        \leq
        2c_9 \sqrt{\dfrac{t_0(p \log t + \rmx)}{t}},
    \end{aligned}        
    \end{align}
    where $c_9 = c_9(c_8, L_2, \overline \lambda_0, K_2) = c_9(K_{2:5}, L_2, \overline \lambda_0) > 0$.

    \noindent \textbf{Step 5:} 
    Combining \eqref{eqn:6_5}-\eqref{eqn:6_8}, we have
    \begin{align*}
        &\dfrac{1}{3} \big\| \bF_{t, \hat \theta_t^{\ML}}^{1/2} (\theta_t - \hat \theta_t^{\ML} ) \big\|_2
        \leq ({\rm i})
        \leq 
        ({\rm ii}) + ({\rm iii})+ ({\rm iv}) \\
        &\qquad \leq 
        c_6
        \sqrt{ \dfrac{\log^2(t/t_0)(p \log t + \rmx)^2}{t} }
        +
        c_7 \sqrt{ \dfrac{t_0 (p \log t_0 + \rmx)}{t} }
        +
        2c_9 \sqrt{\dfrac{t_0(p \log t + \rmx)}{t}} \\
        &\qquad \leq 
        c_6
        \sqrt{ \dfrac{\log^2(t/t_0)(p \log t + \rmx)^2}{t} }
        +
        (c_7 + 2c_9) \sqrt{\dfrac{t_0(p \log t + \rmx)}{t}} \\
        &\qquad \leq 
        c_{10}
        \sqrt{ \dfrac{\log^2(t/t_0)(p \log t + \rmx)^2}{t} }
    \end{align*}
    where $c_{10} = c_{10}(c_6, c_7, c_9, M) = c_{10}(K_{1:5}, L_{1:3}, \overline \lambda_0, r) > 0$.
    It follows that
    \begin{align*}
        \big\| \bF_{t, \hat \theta_t^{\ML}}^{1/2} (\theta_t - \hat \theta_t^{\ML} ) \big\|_2
        \leq 
        3c_{10}
        \sqrt{ \dfrac{\log^2(t/t_0)(p \log t + \rmx)^2}{t} },
    \end{align*}
    which completes the proof.
\end{proof}

\begin{remark}[Rate improvement]
    One can replace the conclusion in Theorem \ref{thm:online_cumulative_error} with 
    \begin{align*}
        \big\| \bF_{t, \hat \theta_t^{\ML}}^{1/2} (\theta_t - \hat \theta_t^{\ML}) \big\|_2
        \leq 
        K \epsilon_{\app, t}', \quad \forall t \in \bbN \text{ with } t > M_1 t_0
    \end{align*}   
    where $K = K(K_{1:5}, L_{1:3}, \overline \lambda_0) > 0$ and
    \begin{align*}
        \epsilon_{\app, t}' = \sqrt{ \dfrac{\log^2(t/t_0)(p \log t + \rmx) \big\{ t_0 \vee (p \log t + \rmx) \big\}  }{t} }.
    \end{align*}
    For simplicity, we retain the original statement.
\end{remark}

\begin{proof}[Proof of Proposition \ref{prop:online_location}]
    Let $\rmx > 0$.
    In this proof, we will work on the event $\scrE_{1:4}(\rmx)$ without explicitly mentioning it. 
    Also, we may assume that $M = M( K_{1:5}, L_{1:2}, r)$ and $M_1 = M_1(K_6)$ are large enough constants.
    Let $t \in \bbN$ with $t > M_1 t_0$. By Theorem \ref{thm:online_cumulative_error}, we have
    \begin{align*}
        \big\| \bF_{t, \hat \theta_t^{\ML}}^{1/2} (\theta_t - \hat \theta_t^{\ML}) \big\|_2
        &\leq 
        c_1 \epsilon_{\app, t},
    \end{align*}
    where $c_1 = c_1(K_{1:5}, L_{1:3}, \overline \lambda_0, r) > 0$ and $\epsilon_{\app, t}$ is specified in Theorem \ref{thm:online_cumulative_error}.
    By \eqref{A5}, we have
    \begin{align} \label{eqn:7_2}  
        \| \hat \theta_t^{\ML} - \theta_{\star} \|_2 
        \leq c_2 \sqrt{\dfrac{p + \log t + \rmx}{t}}
        \leq K_6^{-1} \log 2
    \end{align}
    where $c_2$ is a universal constant and the last inequality holds by
    \begin{align*}
        t > M_1 t_0 \overset{\eqref{A5}}{\geq} M_1 M (p\log (p \vee 3) + \rmx)
    \end{align*}
    with large enough $M_1 = M_1(K_6)$ and $M$.
    Let $h_t = \hat \theta_t^{\ML} - \theta_{\star}$.
    By \eqref{eqn:7_2} and \eqref{A3}, we have
    \begin{align*}
        \dfrac{1}{3} \bF_{t, \theta_{\star}}
        \preceq
        \dfrac{1}{1 + e^{K_6 \| h_t \|_2 }} \bF_{t, \theta_{\star}}
        \preceq
        \bF_{t, \hat \theta_{t}^{\ML}},
    \end{align*}
    which implies that
    \begin{align*}
        \big\| \bF_{t, \hat \theta_{t}^{\ML}}^{-1/2} \bF_{t, \theta_{\star}}  \bF_{t, \hat \theta_{t}^{\ML}}^{-1/2} \big\|_2
        \leq 3.
    \end{align*}
    Therefore, we have
    \begin{align*}
        \big\| \bF_{t, \theta_{\star}}^{1/2} (\hat \theta_t^{\ML} - \theta_t) \big\|_2
        \leq
        \big\| \bF_{t, \hat \theta_{t}^{\ML}}^{-1/2} \bF_{t, \theta_{\star}}  \bF_{t, \hat \theta_{t}^{\ML}}^{-1/2} \big\|_2^{1/2}
        \big\| \bF_{t, \hat \theta_t^{\ML}}^{1/2} (\hat \theta_t^{\ML} - \theta_t) \big\|_2
        \leq 
        \sqrt{3} c_1 \epsilon_{\app, t}.
    \end{align*}
    This completes the proof.
\end{proof}

\begin{proof}[Proof of Proposition \ref{prop:online_precision}]
    Let $\rmx > 0$. In this proof, we will work on the event $\scrE_{1:4}(\rmx)$ without explicitly mentioning it. 
    For notational simplicity, we denote in this proof, for $t \in \bbN$ and $\theta \in \Theta$,
    \begin{align*}
        \bH_t(\theta) = \nabla^2 \ell_t (\theta).
    \end{align*}    
    By Theorem \ref{thm:rate_conv} and \eqref{A5}, we have, for any $t \in \bbN$ with $t \geq t_0$,
    \begin{align} \label{eqn:8_1} 
        \| \theta_t - \theta_{\star} \|_2 \leq c_1 \sqrt{\dfrac{p \log t + \rmx}{t}}, \quad 
        \| \hat \theta_t^{\MAP} - \theta_{\star} \|_2 \leq c_2 \sqrt{\dfrac{p + \log t + \rmx}{t}},
    \end{align}
    where $c_1 = c_1(K_{2:5}) > 0$ and $c_2$ is a universal constant. Also, by \eqref{eqn:8_1}, $t \geq t_0$ and \eqref{A1}, we have
    \begin{align} \label{eqn:8_2} 
        \lambda_{\min} ( \bF_{t, \theta_{\star}} ) \geq K_2 t, \quad
        \lambda_{\min} ( \bOmega_t ) \overset{\eqref{eqn:4_2}}{\geq} K_2 t.
    \end{align}       
    Let $t \in \bbN$ with $t > t_0$. Note that
    \begin{align*}
        &\big\| \bF_{t, \theta_{\star}}^{-1/2} \bOmega_t \bF_{t, \theta_{\star}}^{-1/2} - \bI_p  \big\|_2
        =
        \big\| \bF_{t, \theta_{\star}}^{-1/2} (\bOmega_t - \bF_{t, \theta_{\star}}) \bF_{t, \theta_{\star}}^{-1/2} \big\|_2
        \overset{\eqref{eqn:8_2}}{\leq}
        (K_2 t)^{-1} \big\| \bOmega_t - \bF_{t, \theta_{\star}} \big\|_2 \\
        &\leq 
        (K_2 t)^{-1} 
        \left[ \big\| \bOmega_0 \big\|_2
        + \left\| \sum_{s=1}^{t_0} \big( \bH_s(\theta_{t_0}) - \bH_s(\theta_{\star}) \big) \right\|_2
        + \left\| \sum_{s=t_0+1}^{t} \big( \bH_s(\theta_{s-1}) - \bH_s(\theta_{\star}) \big) \right\|_2
        \right] \\
        &\leq 
        (K_2 t)^{-1} 
        \left[ \big\| \bOmega_0 \big\|_2 
        + \sum_{s=1}^{t_0} \left\| \big( \bH_s(\theta_{t_0}) - \bH_s(\theta_{\star}) \big) \right\|_2
        + \sum_{s=t_0+1}^{t} \left\| \big( \bH_s(\theta_{s-1}) - \bH_s(\theta_{\star}) \big) \right\|_2
        \right] \\
        \overset{\eqref{A3}}&{\leq}
        (K_2 t)^{-1} 
        \left[ \overline \lambda_0
        + \sum_{s=1}^{t_0} L_3 \| \theta_{t_0} - \theta_{\star} \|_2
        + \sum_{s=t_0}^{t-1} L_3 \| \theta_{s} - \theta_{\star} \|_2
        \right] \\
        \overset{\eqref{eqn:8_1}}&{\leq}
        (K_2 t)^{-1} 
        \left[ \overline \lambda_0
        + c_2 L_3 t_0 \sqrt{\dfrac{p + \log t_0 + \rmx}{t_0}}
        + c_1 L_3 \sum_{s=t_0}^{t-1} \sqrt{\dfrac{p \log s + \rmx}{s}}
        \right] \\
        &\leq 
        c_3 
        \left[ t^{-1}
        + \dfrac{  \sqrt{t_0 (p + \log t_0 + \rmx) }  }{t}
        + \dfrac{  \sqrt{t (p\log t + \rmx) }  }{t}
        \right] \\
        &= 
        c_3 
        \left[ t^{-1}
        + \sqrt{ \dfrac{ t_0 }{t} } \sqrt{ \dfrac{ (p + \log t_0 + \rmx)  }{t} }
        + \sqrt{ \dfrac{ p\log t + \rmx }{t} }
        \right] 
        \leq 
        3c_3 \sqrt{ \dfrac{ p\log t + \rmx }{t} },
    \end{align*}
    where $c_3 = c_3(c_1, c_2, K_2, L_3) = c_3(K_{2:5}, L_3, \overline \lambda_0) > 0$.
    By applying similar technique, we have
    \begin{align*}
        \big\| \bOmega_t^{-1/2} \bF_{t, \theta_{\star}} \bOmega_t^{-1/2} - \bI_p  \big\|_2
        \leq 
        3c_3 \sqrt{ \dfrac{ p\log t + \rmx }{t} },
    \end{align*}
    which completes the proof.
\end{proof}

\begin{proof}[Proof of Theorem \ref{thm:online_BvM}]
    Let $\rmx > 0$.
    In this proof, we will work on the event $\scrE_{1:4}(\rmx)$ without explicitly mentioning it. 
    Also, we may assume that $M_2 = M_2(K_6, L_3, \overline \lambda_0)$ is large enough constant.
    Let $t \in \bbN$ with $t > M_2 t_0$. By Proposition \ref{prop:online_precision} we have
    \begin{align} \label{eqn:9_1}
        \big\| \bOmega_t^{-1/2} \bF_{t, \theta_{\star}} \bOmega_t^{-1/2} - \bI_p  \big\|_2
        \leq c_1 \tilde \epsilon_{\app, t}
        = c_1 \sqrt{ \dfrac{ p\log t + \rmx }{t} } \leq 0.684,
    \end{align}
    where $c_1 = c_1(K_{2:5}, L_3, \overline \lambda_0) > 0$ and the last inequality holds by 
    \begin{align*}
        t > M_2 t_0 \overset{\eqref{A5}}{\geq} M_2 M (p\log (p \vee 3) + \rmx)
    \end{align*}
    with large enough $M_2 = M_2(L_3, \overline \lambda_0)$ and $M = M(K_{2:5})$.
    Note that \eqref{eqn:9_1} allows us to utilize Lemma \ref{lemma:Gaussian_TV}. Also, by Proposition \ref{prop:online_location}, we have
     \begin{align*}
        \big\| \bF_{t, \theta_{\star}}^{1/2} ( \hat \theta_t^{\ML} - \theta_t) \big\|_2
        \leq 
        c_2 \epsilon_{\app, t} 
        = c_2 \sqrt{ \dfrac{\log^2(t/t_0)(p \log t + \rmx)^2}{t} },
    \end{align*}   
    where $c_2 = c_2(K_{1:5}, L_{1:3}, \overline \lambda_0, r) > 0$.
    It follows that
    \begin{align*}
        \big\| \bOmega_t^{-1/2} \bF_{t, \theta_{\star}} \bOmega_t^{-1/2} - \bI_p  \big\|_{\rm F}
        &\leq 
        \sqrt{p} \big\| \bOmega_t^{-1/2} \bF_{t, \theta_{\star}} \bOmega_t^{-1/2} - \bI_p  \big\|_{2}
        \leq 
        c_1 \sqrt{\dfrac{ p(p\log t + \rmx) }{t}} \\
        &\leq 
        c_1 \epsilon_{\app, t}.
    \end{align*}
    By applying Lemma \ref{lemma:Gaussian_TV}, we have
    \begin{align*}
        d_{V} \big( Q_{\BvM, t}, \Pi_t \big)
        &\leq 
        \dfrac{1}{2} 
        \left( 
        \big\| \bF_{t, \theta_{\star}}^{1/2} ( \hat \theta_{t}^{\ML} - \theta_t )  \big\|_2^2
        +
        \big\| \bOmega_t^{-1/2} \bF_{t, \theta_{\star}} \bOmega_t^{-1/2} - \bI_p  \big\|_{\rm F}^2
        \right)^{1/2} \\
        &\leq 
        \dfrac{1}{2} 
        \left( 
        \big\| \bF_{t, \theta_{\star}}^{1/2} ( \hat \theta_{t}^{\ML} - \theta_t )  \big\|_2
        +
        \big\| \bOmega_t^{-1/2} \bF_{t, \theta_{\star}} \bOmega_t^{-1/2} - \bI_p  \big\|_{\rm F}
        \right) \\
        &\leq 
        \dfrac{1}{2} \big( c_2 \epsilon_{\app, t} + c_1 \epsilon_{\app, t} \big)
        = c_3 \epsilon_{\app, t},
    \end{align*}
    where $c_3 = c_3(c_1, c_2) = c_3(K_{1:5}, L_{1:3}, \overline \lambda_0, r) = (c_1 + c_2)/2$. Also,
    \begin{align*}
        d_{V} \big( \Pi_t, \Pi(\cdot \mid \bD_{1:t}) \big)
        &\leq 
        d_{V} \big( \Pi_t, Q_{\BvM, t} \big)
        +
        d_{V} \big( Q_{\BvM, t}, \Pi(\cdot \mid \bD_{1:t}) \big) \\
        &\leq 
        c_3 \epsilon_{\app, t}
        +
        d_{V} \big( Q_{\BvM, t}, \Pi(\cdot \mid \bD_{1:t}) \big),
    \end{align*}
    which completes the proof.
\end{proof}


\section{Induction devices} \label{sec:induction}

\begin{lemma}[Induction lemma] \label{lemma:induction_lemma}
    Let $(\tau_k)_{k \in \bbN_0}$ be a sequence in $\bbN$ such that $\tau_k < \tau_{k+1}$ for all $k \in \bbN_0$. Let $(X_t)_{t \in \bbN_0}$ be a sequence in $\bbR_{+}$ satisfying
    \begin{align*}
        X_t \leq X_{t-1} + a_{t-1} X_{t-1}^{3/2} + b_{t}, \quad \forall t \in \bbN,
    \end{align*}
    where $a_t \in \bbR_+$ and $b_t \in \bbR$ for all $t \in \bbN_0$.
    For $t \in \bbN$ with $t \geq \tau_0$, let
    \begin{align*}
        B_t = 0 \vee \left( \max_{\tau_0 \leq s \leq t} \sum_{r = \tau_0 + 1}^{s} b_r \right), \quad
        H_t = X_{\tau_0} + B_t.
    \end{align*}
    For $k \in \bbN_0$, let
    \begin{align*}
        A_k = \sum_{t = \tau_k}^{\tau_{k+1}} a_t, \quad 
        S_k = H_{\tau_{k+1}} + \sum_{t= \tau_0 + 1}^{\tau_k} a_{t-1} X_{t-1}^{3/2}.
    \end{align*}
    Fix $k \in \bbN_0$. Assume that 
    \begin{align} \label{assume:induction_1}
        A_k \leq \dfrac{C_k - 1}{ C_k^{3/2} \sqrt{(S_k \vee 1)} }.
    \end{align}
    for some $C_k > 1$.
    Then, 
    \begin{align*}
        X_t \leq C_k S_k, \quad \forall t \in \{ \tau_k, \tau_k + 1, ..., \tau_{k+1} \}.
    \end{align*}
\end{lemma}
\begin{proof}
    Let $C_k > 1$ satisfying \eqref{assume:induction_1} be given.
    For $k \in \bbN_0$, let $I_k = \{ \tau_k, \tau_k + 1, \tau_k + 2, ..., \tau_{k+1} \}$ and $t_{\ast} = t_{\ast}(k) = \inf\{ t \in I_k :  X_t > C_k S_k \}$. 
    Fix $k \in \bbN_0$. Let $t_{\ast} = t_{\ast}(k)$ for simplicity in notation.
    By contradiction, we will show that $t_{\ast} = \infty$.
    Suppose that $t_{\ast} \in I_k$. By the definition of $t_{\ast}$, we have $X_{t_{\ast}} > C_k S_k$ and
    \begin{align} \label{eqn:2_1}
        X_{t} \leq C_k S_k, \quad \forall t \in \{ \tau_k, ..., t_{\ast} - 1 \}.
    \end{align}    
    First, we cosider the case $t_{\ast} = \tau_k$. 
    If $k = 0$, it is easy to see that $X_{\tau_0} \leq S_0$ because $B_t$ and $a_{t} X_t^{3/2}$ are non-negative for any $t$. Hence, we now assume that $t_{\ast} = \tau_k$ and $k \in \bbN$.
    By the recursive formula, we have
    \begin{align*}
        X_{t_{\ast}} 
        &\leq 
        X_{\tau_0} 
        + \sum_{t = \tau_0 + 1}^{t_{\ast}} a_{t-1} X_{t-1}^{3/2} 
        + \sum_{t = \tau_0 + 1}^{t_{\ast}} b_t 
        =
        X_{\tau_0} 
        + \sum_{t = \tau_0 + 1}^{\tau_k} b_t
        + \sum_{t = \tau_0 + 1}^{\tau_k} a_{t-1} X_{t-1}^{3/2} \\
        &\leq
        X_{\tau_0} + B_{\tau_k}
        + \sum_{t = \tau_0 + 1}^{\tau_k} a_{t-1} X_{t-1}^{3/2} 
        \leq
        X_{\tau_0} + B_{\tau_{k+1}}
        + \sum_{t = \tau_0 + 1}^{\tau_k} a_{t-1} X_{t-1}^{3/2} \\
        &=
        H_{\tau_{k+1}}
        + \sum_{t = \tau_0 + 1}^{\tau_k} a_{t-1} X_{t-1}^{3/2} 
        =
        S_k \leq C_k S_k,
    \end{align*}
    where the second and third inequalities hold by the definition of $B_t$ and the monotonicity of $t \mapsto B_t$, respectively. 
    The last display contradicts $X_{t_{\ast}} > C_k S_k$. Hence, we only need to consider case $t_{\ast} > \tau_k$.
    Suppose that $k \in \bbN$ and $t_{\ast} > \tau_k$.
    As in the last display, we have
    \begin{align*}
        X_{t_{\ast}} 
        &\leq 
        X_{\tau_0} 
        + \sum_{t = \tau_0 + 1}^{t_{\ast}} a_{t-1} X_{t-1}^{3/2} 
        + \sum_{t = \tau_0 + 1}^{t_{\ast}} b_t \\
        &=
        X_{\tau_0} 
        + \sum_{t = \tau_0 + 1}^{t_{\ast}} b_t
        + \sum_{t = \tau_0 + 1}^{\tau_k} a_{t-1} X_{t-1}^{3/2} 
        + \sum_{t = \tau_k + 1}^{t_{\ast}} a_{t-1} X_{t-1}^{3/2} \\
        &\leq
        X_{\tau_0} + B_{t_{\ast}}
        + \sum_{t = \tau_0 + 1}^{\tau_k} a_{t-1} X_{t-1}^{3/2} 
        + \sum_{t = \tau_k + 1}^{t_{\ast}} a_{t-1} X_{t-1}^{3/2} \\
        &\leq
        X_{\tau_0} + B_{\tau_{k+1}}
        + \sum_{t = \tau_0 + 1}^{\tau_k} a_{t-1} X_{t-1}^{3/2} 
        + \sum_{t = \tau_k + 1}^{t_{\ast}} a_{t-1} X_{t-1}^{3/2} \\
        &=
        H_{\tau_{k+1}}
        + \sum_{t = \tau_0 + 1}^{\tau_k} a_{t-1} X_{t-1}^{3/2} 
        + \sum_{t = \tau_k + 1}^{t_{\ast}} a_{t-1} X_{t-1}^{3/2}
        =
        S_k
        + \sum_{t = \tau_k + 1}^{t_{\ast}} a_{t-1} X_{t-1}^{3/2},
    \end{align*}
    Note that
    \begin{align*}
        \sum_{t = \tau_k + 1}^{t_{\ast}} a_{t-1} X_{t-1}^{3/2}
        \overset{\eqref{eqn:2_1}}{\leq}
        \sum_{t = \tau_k + 1}^{t_{\ast}} a_{t-1} \big( C_k S_k \big)^{3/2}
        \leq   
        \big( C_k S_k \big)^{3/2} \sum_{t = \tau_k}^{\tau_{k+1}} a_{t} 
        =
        \big( C_k S_k \big)^{3/2} A_k.
    \end{align*}
    By the last display, we have
    \begin{align*}
        X_{t_{\ast}} 
        \leq S_k + \big( C_k S_k \big)^{3/2} A_k.
    \end{align*}
    By the construction, we have $S_k \geq 0$.
    Also, by the assumption \eqref{assume:induction_1}, we have
    \begin{align*}
        \big( C_k S_k \big)^{3/2} A_k 
        \leq 
        \big( C_k S_k \big)^{3/2}
        \dfrac{C_k - 1}{ C_k^{3/2} \sqrt{(S_k \vee 1)} }
        =
        (C_k - 1)\dfrac{ S_k^{3/2} }{ \sqrt{(S_k \vee 1)} }
        \leq 
        (C_k - 1)S_k.
    \end{align*}
    It follows that
    \begin{align*}
        X_{t_{\ast}} 
        \leq 
        S_k + \big( C_k S_k \big)^{3/2} A_k
        \leq 
        S_k + (C_k - 1)S_k
        =
        C_k S_k,
    \end{align*}
    which contradicts $X_{t_{\ast}} > C_k S_k$. This completes the proof.
\end{proof}

\begin{proposition}[Dyadic recursion bound] \label{prop:induction}
    Let $\tau_{\ast} \in \bbN \cup \{ \infty \}$ with $\tau_{\ast} > \tau_0$, where
    \begin{align*}
        \tau_0 \geq \lceil M(p \log (p \vee 3) + \rmx) \rceil
    \end{align*}
    for some $M > 0$ and $\rmx > 0$.
    Let $(\tau_k)_{k \in \bbN_0}$ be a sequence in $\bbN$ such that $\tau_{k} = (2^{k} \tau_0) \wedge \tau_{\ast}$ for all $k \in \bbN_0$.
    Let $(X_t)_{t \in \bbN_0}$ be a non-negative sequence satisfying
    \begin{align*}
        X_t \leq X_{t-1} + a_{t-1} X_{t-1}^{3/2} + b_{t}, \quad \forall t \in \bbN \text{ with } \tau_0 < t \leq \tau_{\ast},
    \end{align*}
    where $a_t \in \bbR_+$ and $b_t \in \bbR \ $ for all $\tau_0 \leq t \leq \tau_{\ast}$.
    For $s, t \in \bbN$ with $\tau_0 \leq s < t \leq \tau_{\ast}$, suppose that
    \begin{align*}
        X_{\tau_0} \leq D_0 \left( p + \rmx  \right), \quad 
        \sum_{r = s}^{t} a_r \leq D_1 s^{-1/2}, \quad 
        \sum_{r = \tau_0+1}^{t} b_r \leq D_2  p \log t + D_3 ( \rmx + \log t ),
    \end{align*}
    where $D_0, D_1, D_2, D_3 > 0$. Then, there exists $M_0 = M_0(D_0, D_1, D_2, D_3) > 0$ such that if $M \geq M_0$,
    \begin{align*}
        X_t \leq K \left( p \log t + \rmx  \right), \quad \forall t \in \bbN \text{ with } \tau_0 \leq t \leq \tau_{\ast},
    \end{align*}
    where $K = K(D_0, D_2, D_3)> 0$.
\end{proposition}
\begin{proof}
    In this proof, we assume that $M_0 = M_0(D_0, D_1, D_2, D_3) > 0$ is large enough.    
    For $t \in \bbN$ with $t > \tau_0$, let
    \begin{align*}
        B_t = 0 \vee \left( \max_{\tau_0 \leq s \leq t} \sum_{r = \tau_0 + 1}^{s} b_r \right), \quad
        H_t = X_{\tau_0} + B_t, \quad 
        \overline B(t) = D_2  p \log t + D_3 ( \rmx + \log t ).       
    \end{align*}
    One can easily check that $B_t \leq \overline B(t)$ for any $t \in \bbN$ with $t > \tau_0$.
    For $k \in \bbN_0$, let
    \begin{align*}
        A_k = \sum_{t = \tau_k}^{\tau_{k+1}} a_t, \quad 
        S_k = H_{\tau_{k+1}} + \sum_{t= \tau_0 + 1}^{\tau_k} a_{t-1} X_{t-1}^{3/2}, \quad 
        C_k = 1 + \frac{2^{-k/4}}{7}, \quad 
        P_k = \prod_{j=0}^{k-1} C_j, 
    \end{align*}
    Here, we set $P_0 = 1$ and $S_0 = H_{\tau_1}$.
    For any $k \in \bbN_0$, note that
    \begin{align}
    \begin{aligned} \label{eqn:3_0}
        A_k &\leq D_1 \tau_k^{-1/2}, \\
        P_k &\leq P_{\infty} \overset{\rm def}{=} \prod_{j=0}^{\infty} C_j \leq 2.5,
    \end{aligned}        
    \end{align}
    where the first inequality holds by the assumption.

\noindent \textbf{Step 1:} 
We will show that if  
\begin{align} \label{eqn:3_1}
    A_{k} \leq \dfrac{C_{k} - 1}{ C_{k}^{3/2} \sqrt{S_{k} \vee 1} }, \quad 
    S_{k} \leq P_k H_{\tau_{k+1}}
\end{align}
for a given $k \in \bbN_0$ with $\tau_k < \tau_{k+1}$, then
\begin{align} 
\begin{aligned} \label{eqn:3_2}
    X_{t} \leq C_k S_k \leq c_0 \big( p \log t + \rmx \big), \quad \forall t \in I_k,
\end{aligned}    
\end{align}
where $I_k = \{ \tau_k, ..., \tau_{k+1} \}$ and $c_0 = c_0(D_0, D_2, D_3) > 0$.
Suppose that \eqref{eqn:3_1} holds for a given $k \in \bbN_0$ with $\tau_k < \tau_{k+1}$.
Then, Lemma \ref{lemma:induction_lemma} implies that for any $t \in \{ \tau_{k}, ..., \tau_{k+1} \}$
\begin{align*}
    X_t \leq C_k S_k,
\end{align*}
which completes the first inequality in \eqref{eqn:3_2}. 
By the assumption, for any $t > \tau_0$,
\begin{align} \label{eqn:3_2_0}
    \overline B(t) \leq c_1 \big( p \log t + \rmx \big), \quad
    X_{\tau_0} \leq D_0 \big( p + \rmx \big),
\end{align}
where $c_1 = c_1(D_2, D_3) > 0$.
Also,
\begin{align}
\begin{aligned} \label{eqn:3_2_1}
    C_k S_k
    \overset{\eqref{eqn:3_1}}&{\leq} C_{k} P_k H_{\tau_{k+1}}
    = P_{k+1} H_{\tau_{k+1}}
    \leq P_{\infty} H_{\tau_{k+1}}
    = P_{\infty} \big( X_{\tau_0} + B_{\tau_{k+1}} \big) \\
    &\leq P_{\infty} \big( X_{\tau_0} + \overline B(\tau_{k+1}) \big)  
    \leq P_{\infty} \big( X_{\tau_0} + \overline B(2t) \big)    
    \overset{\eqref{eqn:3_0}}{\leq} 2.5 \big( X_{\tau_0} + \overline B(2t) \big) \\
    &\leq 
    c_0 \big( p \log t + \rmx \big),
\end{aligned}    
\end{align}
where $c_0 = c_0(c_1, D_0) = c_0(D_0, D_2, D_3) > 0$, and the fourth inequality holds by the monotonicity of $t \mapsto \overline B(t)$ and $\tau_{k+1} \leq 2t$.
This completes the proof of the two assertions in \eqref{eqn:3_2}.
\bigskip

\noindent \textbf{Step 2:}
Let $k_{\ast}  = \inf\{k \in \bbN : \tau_{\ast} \leq 2^k \tau_0 \}$ so that $\tau_{\ast} = \tau_{k_{\ast}}$ if $k_{\ast} < \infty$.
By induction, we will show that 
\begin{align} \label{eqn:3_3}
    A_{k} \leq \dfrac{C_{k} - 1}{ C_{k}^{3/2} \sqrt{S_{k} \vee 1} }, \quad 
    S_{k} \leq P_k H_{\tau_{k+1}}, \quad 
    \forall k \in \bbN_0 \text{ with } k \leq k_{\ast}-1.
\end{align}
Since $\tau_{k} < \tau_{k+1}$ for all $k \leq k_{\ast}-1$, we can apply Lemma \ref{lemma:induction_lemma}.
Even if $k_{\ast} = \infty$, the assertion in \eqref{eqn:3_3} is understood to hold for all $k \in \bbN_0$.
Let $k = 0$. 
Note that
\begin{align} \label{eqn:3_4}
    S_0 
    = H_{\tau_1}
    = P_0 H_{\tau_1}.
\end{align}
Also,
\begin{align}
\begin{aligned} \label{eqn:3_4_1}
    S_0 \vee 1
    &\leq H_{\tau_1} + 1
    \leq X_{\tau_0} + B_{\tau_1} + 1
    \leq X_{\tau_0} + \overline B(\tau_1) + 1 \\
    &\leq D_0 (p + \rmx) + c_1(p \log \tau_1 + \rmx) + 1 
    \leq  c_2^2 \big( p \log \tau_{1} + \rmx \big) \\
    &\leq c_2^2 \big( p \log 2 + p \log \tau_{0} + \rmx \big),
\end{aligned}    
\end{align}
where $c_2 = c_2(D_0, c_1) = c_2(D_0, D_2, D_3) > 0$.
Note that
\begin{align*}
    &A_0 \sqrt{S_0 \vee 1}
    \leq 
    D_1 \tau_0^{-1/2}
    \big\{ c_2^2 \big( p \log 2 + p \log \tau_{0} + \rmx \big) \big\}^{1/2} \\
    &\qquad \leq 
    c_2 D_1 
    \left( \sqrt{\dfrac{p \log 2}{\tau_0}} + \sqrt{\dfrac{p \log \tau_0}{\tau_0}} + \sqrt{\dfrac{\rmx}{\tau_0}}  \right) 
    \leq 
    c_2 D_1 
    \left( M^{-1/2} + \sqrt{\dfrac{p \log \tau_0}{\tau_0}} + M^{-1/2} \right) \\
    &\qquad \leq 
    c_3 M^{-1/4}
    \leq 
    c_3 M_0^{-1/4}
\end{align*}
where $c_3 = c_3(c_2, D_1) = c_3(D_0, D_1, D_2, D_3) > 0$. Also,
\begin{align*}
    \dfrac{C_0 - 1}{C_0^{3/2}}
    =
    7^{-1} (1 + 7^{-1})^{-3/2}
    \geq 0.1.
\end{align*}
It follows that
\begin{align} \label{eqn:3_5}
    A_0 \sqrt{S_0 \vee 1}
    \leq 
    c_3 M_0^{-1/4}
    \leq 
    0.1
    \leq 
    \dfrac{C_0 - 1}{C_0^{3/2}},
\end{align}
where the second inequality holds by large enough $M_0 = M_0(c_3) = M_0(D_0, D_1, D_2, D_3)$. By \eqref{eqn:3_4} and \eqref{eqn:3_5}, \eqref{eqn:3_3} holds for $k = 0$.

\noindent \textbf{Step 3:} Let $k \in \bbN_0$ with $k \leq k_{\ast} - 2$. 
By \textbf{Step 1}, it is sufficient to consider the case where $k \leq k_{\ast} - 2$.
Suppose that 
\begin{align} \label{eqn:3_6_0}
    A_{k} \leq \dfrac{C_{k} - 1}{ C_{k}^{3/2} \sqrt{S_{k} \vee 1} }, \quad 
    S_{k} \leq P_{k} H_{\tau_{k+1}}.
\end{align}
By \eqref{eqn:3_1}, \eqref{eqn:3_2} and \eqref{eqn:3_6_0}, we have
\begin{align} \label{eqn:3_6}
    X_{t} \leq C_{k} S_{k} \leq c_0 \big( p \log t + \rmx \big), \quad \forall t \in I_{k}.
\end{align}
Also, by \eqref{eqn:3_0}, we have
\begin{align} \label{eqn:3_6_1}
    A_{k+1} \leq D_1 \tau_{k+1}^{-1/2} \leq D_1 2^{-(k+1)/2} \tau_0^{-1/2}.
\end{align}
First, we will show that $S_{k+1} \leq P_{k+1} H_{\tau_{k+2}}$.
Note that
\begin{align*}
    S_{k} - H_{\tau_{k+1}} &= \sum_{t = \tau_0 + 1}^{\tau_{k}} a_{t-1} X_{t-1}^{3/2}, \\
    \sum_{t = \tau_{k} + 1}^{\tau_{k+1}} a_{t-1} X_{t-1}^{3/2} 
    \overset{\eqref{eqn:3_6}}&{\leq}
    \sum_{t = \tau_{k} + 1}^{\tau_{k+1}} a_{t-1} \big( C_k S_k \big)^{3/2} 
    \leq A_{k} \big( C_k S_k \big)^{3/2}
    \overset{\eqref{eqn:3_6_0}}{\leq}
    (C_k - 1)S_k,
\end{align*}
where the first equality holds by the definition of $S_{k}$.
Then, we have
\begin{align*}
    S_{k+1} 
    &= H_{\tau_{k+2}} + \sum_{t = \tau_0 + 1}^{\tau_{k+1}} a_{t-1} X_{t-1}^{3/2}
    = H_{\tau_{k+2}} + \big( S_{k} - H_{\tau_{k+1}} \big) + \sum_{t = \tau_{k} + 1}^{\tau_{k+1}} a_{t-1} X_{t-1}^{3/2} \\
    &\leq 
    H_{\tau_{k+2}} + \big( S_{k} - H_{\tau_{k+1}} \big) + (C_k - 1)S_k
    = 
    C_k S_k + H_{\tau_{k+2}} - H_{\tau_{k+1}}. 
\end{align*}
Combining the last display with \eqref{eqn:3_6_0}, we have
\begin{align}
\begin{aligned} \label{eqn:3_6_2}
    S_{k+1} 
    &\leq C_k \big( P_{k} H_{\tau_{k+1}} \big) + H_{\tau_{k+2}} - H_{\tau_{k+1}}
    = P_{k+1} H_{\tau_{k+1}} + H_{\tau_{k+2}} - H_{\tau_{k+1}} \\
    &= \big( P_{k+1} - 1 \big)  H_{\tau_{k+1}} + H_{\tau_{k+2}}
    \leq \big( P_{k+1} - 1 \big) H_{\tau_{k+2}} + H_{\tau_{k+2}}
    = P_{k+1} H_{\tau_{k+2}}.
\end{aligned}
\end{align}
Next, we will show that 
\begin{align*}
    A_{k+1} \leq \dfrac{C_{k+1} - 1}{ C_{k+1}^{3/2} \sqrt{S_{k+1} \vee 1} }.
\end{align*}
As in \eqref{eqn:3_4_1}, we have
\begin{align*}
    S_{k+1} \vee 1
    \overset{\eqref{eqn:3_6_2}}&{\leq}
    P_{k+1} H_{\tau_{k+2}} + 1
    \leq 
    P_{\infty} H_{\tau_{k+2}} +1
    \leq 
    2.5 \big( X_{\tau_0} + B_{\tau_{k+2}} \big) + 1 \\
    &\leq 
    2.5 \big( X_{\tau_0} + \overline B(\tau_{k+2}) \big) + 1 
    \leq
    c_4^2 \big( p \log \tau_{k+2} + \rmx \big),
\end{align*}
where $c_4 = c_4(D_0, c_1) = c_4(D_0, D_2, D_3) > 0$.
For any $k \le k_{\ast}-2$, it holds that
\begin{align*}
    \log \tau_{k+2} \leq (k+2)\log 2 + \log \tau_0.
\end{align*}
By combining the last two displays with \eqref{eqn:3_6_1}, the following bound holds:
\begin{align*}
    &A_{k+1} \sqrt{S_{k+1} \vee 1}
    \leq 
    D_1 2^{-(k+1)/2} \tau_0^{-1/2}
    \big\{ c_4 \big( p \log \tau_{k+2} + \rmx \big)^{1/2} \big\} \\
    &\qquad \leq
    D_1 2^{-(k+1)/2} \tau_0^{-1/2}
    \Big\{ c_4 \sqrt{ p(k+2)\log2 + p\log \tau_0 + \rmx} \Big\} \\
    &\qquad \leq
    c_4 D_1 2^{-(k+1)/2} \tau_0^{-1/2}
    \big( \sqrt{p(k+2)\log2} + \sqrt{p\log \tau_0} + \sqrt{\rmx} \big) \\
    &\qquad \leq
    c_4 D_1 2^{-(k+1)/4} \left( 2^{-(k+1)/4} \sqrt{k+2}  \right) \tau_0^{-1/2}
    \big( \sqrt{ p \log 2} + \sqrt{p \log \tau_0} + \sqrt{\rmx} \big) \\
    &\qquad \leq
    c_5 D_1 2^{-(k+1)/4} \left( \sqrt{\dfrac{p\log \tau_0}{\tau_0}} + \sqrt{\dfrac{\rmx}{\tau_0}} \right)
    \leq 
    c_6 2^{-(k+1)/4} M^{-1/4}
    \leq 
    c_6 2^{-(k+1)/4} M_0^{-1/4},
\end{align*}
where $c_5 = c_5(c_4) = c_5(D_0, D_2, D_3) > 0$ and $c_6 = c_6(D_1, c_5) = c_6(D_0, D_1, D_2, D_3) > 0$.
Also,
\begin{align*}
    \dfrac{C_{k+1} - 1}{C_{k+1}^{3/2}}
    =
    \dfrac{7^{-1} 2^{-(k+1)/4}}{(1 + 7^{-1}2^{-(k+1)/4})^{3/2}}
    \geq 
    c_7 2^{-(k+1)/4},
\end{align*}
where $c_7 > 0$ is a universal constant. It follows that
\begin{align*}
    A_{k+1} \sqrt{S_{k+1} \vee 1}
    \leq c_6 2^{-(k+1)/4} M_0^{-1/4}
    \leq c_7 2^{-(k+1)/4}
    \leq \dfrac{C_{k+1} - 1}{C_{k+1}^{3/2}},
\end{align*}
where the second inequality holds by large enough $M_0 = M_0(c_6, c_7) = M_0(D_0, D_1, D_2, D_3)$. 
Combining the last display and \eqref{eqn:3_6_2}, we have 
\begin{align*}
    A_{k+1} \leq \dfrac{C_{k+1} - 1}{ C_{k+1}^{3/2} \sqrt{S_{k+1} \vee 1} }, \quad 
    S_{k+1} \leq P_{k+1} H_{\tau_{k+2}},
\end{align*}
which completes the proof of \eqref{eqn:3_3}.
Combining \eqref{eqn:3_1}, \eqref{eqn:3_2} and \eqref{eqn:3_3}, therefore, we have
\begin{align*}
    X_{t} \leq c_0 \big( p \log t + \rmx \big), \quad \forall t \in \bbN \text{ with } \tau_0 \leq t \leq \tau_{\ast},
\end{align*}
which completes the proof.
\end{proof}


\section{Theoretical justifications for assumptions} \label{sec:assume_example}

In this subsection, we follows settings and notations in \ref{sec:verification}

\subsection{Example: logistic regression model}

Now, we will demonstrate that \eqref{A1} is mild assumption under the logistic regression model. 
For $t \in \bbN$ and $\theta \in \bbR^p$, note that
\begin{align*}
    \nabla^2 L_{1:t}(\theta) 
    = \sum_{s=1}^{t} \nabla^2 \ell_{s}(\theta)
    = \sum_{s=1}^{t} b''(x_s^{\top} \theta) x_s x_s^{\top}.
\end{align*}

\begin{proposition}[Verification of \eqref{A1}] \label{prop:A1}
    For $\tau > 0$ and $t \in \bbN$, let $I_{t}(\tau) = \{ s \in [t] : |x_s^{\top} \theta_{\star}| \leq \tau \}$.
    Assume that there exist some constants $D_1, D_2, D_3 > 0$ such that for any $t \in \bbN$ with $t \geq D_1 p$
    \begin{align*}
        \lambda_{\min} \left( \sum_{s \in I_{t}(D_2) } x_s x_s^{\top} \right) \geq D_3 t.
    \end{align*}
    Assume further that $\max_{t \in \bbN} \| x_t \|_2 \cdot r \leq D_4$ for some constant $D_4 > 0$.
    Then, we have, for any $t \in \bbN$ with $t \geq D_1 p$,
    \begin{align*}
        \inf_{ \{ \overline \theta_s \}_{s=1}^{t} \subset \Theta_{r} } 
        \lambda_{\min} \left( \sum_{s=1}^{t} \nabla^2 \ell_{s}( \overline \theta_s ) \right)
        \geq 
        K t,
    \end{align*}
    where $K = K(D_2, D_3, D_4) > 0$.
\end{proposition}
\begin{proof}
    Let $\overline \theta \in \Theta_r$. 
    By the definition of $\Theta_r$ specified in \eqref{def:local_set}, we have $\| \overline \theta - \theta_{\star} \|_2 \leq r$.
    By Lemmas \ref{lemma:logit_self_concordance} and \ref{lemma:smooth_hessian_logit}, we have for any $t \in \bbN$
    \begin{align*}
        \nabla^2 \ell_t( \overline \theta )
        \succeq
        \exp \big( - \sqrt{6} \| x_t \|_2 \| \overline \theta - \theta \|_2 \big)
        \nabla^2 \ell_t( \theta_{\star} )
        \succeq
        e^{- \sqrt{6} D_4}
        \nabla^2 \ell_t( \theta_{\star} )
        =
        c_1 \nabla^2 \ell_t( \theta_{\star} ),
    \end{align*}
    where $c_1 = c_1(D_4) > 0$.     
    Let $t \in \bbN$ with $t \geq D_1 p$ and $\{ \overline \theta_s \}_{s=1}^{t} \subset \Theta_{r}$.
    By the last display, we have
    \begin{align*}
        \sum_{s=1}^{t} \nabla^2 \ell_{s}( \overline \theta_s )
        \succeq
        c_1 \sum_{s=1}^{t} \nabla^2 \ell_{s}( \theta_{\star} ).
    \end{align*}
    Recall that $I_{t}(D_2) = \{ s \in [t] : |x_s^{\top} \theta_{\star}| \leq D_2 \}$.
    Then, we have for any $s \in I_{t}(D_2)$
    \begin{align*}
        b''(x_{s}^{\top} \theta_{\star})
        \geq 
        b''(D_2) 
        = e^{D_2}/(1 + e^{D_2})^2
        \overset{\rm def}{=} c_2,
    \end{align*}
    where the inequality holds by the symmetry and monotonicity of $b''(\cdot)$ in the logistic regression model.    
    \begin{align*}
        \lambda_{\min} \left(  \sum_{s=1}^{t} \nabla^2 \ell_{s}( \theta_{\star} ) \right)
        &=
        \lambda_{\min} \left(  \sum_{s=1}^{t} b''( x_s^{\top} \theta_{\star}) x_s x_s^{\top}  \right)
        \geq 
        \lambda_{\min} \left(  \sum_{s \in I_t(D_2) } b''(x_s^{\top} \theta_{\star}) x_s x_s^{\top}  \right) \\
        &\geq c_2 \lambda_{\min} \left(  \sum_{s \in I_t(D_2) } x_s x_s^{\top} \right)
        \geq c_2 D_3 t.
    \end{align*}
    It follows that
    \begin{align*}
        \lambda_{\min} \left( \sum_{s=1}^{t} \nabla^2 \ell_{s}( \overline \theta_s ) \right)
        \geq 
        c_1 c_2 D_3 t = c_3 t,
    \end{align*}
    where $c_3 = c_3(c_1, c_2, D_3) = c_3(D_2, D_3, D_4) > 0$. Since $\{ \overline \theta_s \}_{s=1}^{t} \subset \Theta_{r}$ is arbitrary, we complete the proof.
\end{proof}

\begin{proposition}[Verifiaction of \eqref{A2}] \label{prop:A2}
    Let $t \in \bbN$ and $\theta_1, \theta_2 \in \Theta_r$. 
    Assume that there exist some constants $C_7, C_8 > 0$ such that
    \begin{align*}
        \max_{t \in \bbN} \| x_t \|_2 \cdot r \leq C_7, \quad 
        \max_{t \in \bbN} \| x_t \|_2 \leq C_8.
    \end{align*}
    Then, there exists $\tilde \theta = \tilde \theta(\theta_1, \theta_2) \in \Theta_r$ such that
    \begin{align*}
        \ell_t(\theta_1)
        &=
        \ell_t(\theta_2)
        + \langle \nabla \ell_t(\theta_2), \theta_1 - \theta_2 \rangle
        + \dfrac{1}{2} \langle \nabla^2 \ell_t ( \tilde \theta ), (\theta_1- \theta_2)^{\otimes 2} \rangle, \\
        \nabla^2 \ell_t ( \theta_1 )
        &\preceq
        (1 + K \| \theta_1 - \theta_2 \|_2 ) \nabla^2 \ell_t ( \tilde \theta ),
    \end{align*}
    where $K = K(C_7, C_8) > 0$.
    Furthermore, we have
    \begin{align*}
        \max_{t \in \bbN} \sup_{\theta \in \Theta_r} \| \nabla^2 \ell_t ( \theta ) \|_2 \leq \widetilde K,
    \end{align*}
    where $\widetilde K = \widetilde K(C_8) > 0$.
\end{proposition}
\begin{proof}
    Let $t \in \bbN$ and $\theta_1, \theta_2 \in \Theta_r$.
    Note that $\ell_t(\cdot)$ is twice differentiable function.
    Then, by Taylor's theorem, there exists $\tilde \theta = \tilde \theta(\theta_1, \theta_2) \in \Theta_r$ such that
    \begin{align*}
        \ell_t(\theta_1)
        =
        \ell_t(\theta_2)
        + \langle \nabla \ell_t(\theta_2), \theta_1- \theta_2 \rangle
        + \dfrac{1}{2} \langle \nabla^2 \ell_t ( \tilde \theta ), (\theta_1- \theta_2)^{\otimes 2} \rangle.
    \end{align*}
    By Lemma \ref{lemma:logit_self_concordance}, $\ell_t(\cdot)$ is $\sqrt{6} \|x_t \|_2$-self-concordance like function for any $t \in \bbN$. Let $h = \tilde \theta - \theta_1$. By applying Lemma \ref{lemma:smooth_int_hessian_logit} with $f = \ell_t$, $M = \sqrt{6} \|x_t \|_2$ and $u = \theta_1$, we have
    \begin{align*}
        \dfrac{ 1 - e^{-\sqrt{6} \|x_t\|_2 \| h \|_2 }  }{ \sqrt{6} \|x_t\|_2 \| h \|_2 }
        \nabla^2 \ell_t(\theta_1)
        \preceq
        \int_{0}^{1} \nabla^2 \ell_t \big( \theta_1 + \alpha h \big) \rmd \alpha.
    \end{align*}
    Similarly, by applying Lemma \ref{lemma:smooth_int_hessian_logit} with $u = \tilde \theta$, 
    \begin{align*}
        \int_{0}^{1} \nabla^2 \ell_t \big( \tilde \theta - \alpha h \big) \rmd \alpha
        \preceq
        \dfrac{ e^{\sqrt{6} \|x_t\|_2 \| h \|_2 } -1 }{ \sqrt{6} \|x_t\|_2 \| h \|_2 }
        \nabla^2 \ell_t(\tilde \theta).
    \end{align*}
    Since
    \begin{align*}
        \int_{0}^{1} \nabla^2 \ell_t \big( \theta_1 + \alpha h \big) \rmd \alpha
        =
        \int_{0}^{1} \nabla^2 \ell_t \big( \tilde \theta - \alpha h \big) \rmd \alpha,
    \end{align*}
    we have
    \begin{align*}
        \dfrac{ 1 - e^{-\sqrt{6} \|x_t\|_2 \| h \|_2 }  }{ \sqrt{6} \|x_t\|_2 \| h \|_2 }
        \nabla^2 \ell_t(\theta_1)
        \preceq
        \dfrac{ e^{\sqrt{6} \|x_t\|_2 \| h \|_2 } -1 }{ \sqrt{6} \|x_t\|_2 \| h \|_2 }
        \nabla^2 \ell_t(\tilde \theta),
    \end{align*}
    which further simplified to
    \begin{align*}
        \nabla^2 \ell_t(\theta_1)
        &\preceq
        \dfrac{ e^{\sqrt{6} \|x_t\|_2 \| h \|_2 } -1 }{ 1 - e^{-\sqrt{6} \|x_t\|_2 \| h \|_2 } }
        \nabla^2 \ell_t(\tilde \theta)
        =
        e^{\sqrt{6} \|x_t\|_2 \| h \|_2 }
        \nabla^2 \ell_t(\tilde \theta).
    \end{align*}
    Since $\tilde \theta$ lies on the line segment between $\theta_1$ and $\theta_2$, we have
    \begin{align*}
        e^{\sqrt{6} \|x_t\|_2 \| h \|_2 }
        &\leq 
        e^{\sqrt{6} \|x_t\|_2 \| \theta_1 - \theta_2 \|_2 }
        \leq 
        1 + e^{\sqrt{6} \|x_t\|_2 \| \theta_1 - \theta_2 \|_2 } \big( \sqrt{6} \|x_t\|_2 \| \theta_1 - \theta_2 \|_2 \big) \\
        &\leq 
        1 + e^{2 \sqrt{6} \|x_t\|_2 r } \big( \sqrt{6} C_8 \| \theta_1 - \theta_2 \|_2 \big)
        \leq 
        1 + e^{2 \sqrt{6} C_7 } \big( \sqrt{6} C_8 \| \theta_1 - \theta_2 \|_2 \big) \\
        &=
        1 + \big( e^{2 \sqrt{6} C_7 } \sqrt{6} C_8 \big) \| \theta_1 - \theta_2 \|_2,
    \end{align*}
    where the second inequality holds by the following inequality:
    \begin{align*}
        e^{z} \leq 1 + e^{z}z, \quad \forall z \in \bbR_{+}.
    \end{align*}
    Next, we will prove the last assertion. 
    Let $t \in \bbN$ and $\theta \in \Theta_r$. 
    Recall that
    \begin{align*}
        \nabla^2 \ell_t ( \theta ) = b''(x_{t}^{\top} \theta) x_t x_t^{\top}, \quad 
        \sup_{\eta \in \bbR} b''(\eta) \leq 1/4.
    \end{align*}
    Let $\cU = \{ u \in \bbR^p : \| u \|_2 = 1 \}$. Note that
    \begin{align*}
        \| \nabla^2 \ell_t ( \theta ) \|_2
        &= \sup_{u \in \cU} \langle u, \nabla^2 \ell_t ( \theta ) u \rangle
        = \sup_{u \in \cU} b''(x_{t}^{\top} \theta) \langle u, (x_t x_t^{\top}) u \rangle
        \leq \dfrac{1}{4} \sup_{u \in \cU} \langle x_t, u \rangle^2 \\
        &= \dfrac{1}{4} \| x_t \|_2^2
        \leq \dfrac{C_8^2}{4}.
    \end{align*}
    This completes the proof.
\end{proof}

For $3$-order symmmetric tensor $\cT = (\cT_{ijk})_{i,j,k \in [p]} \in \bbR^{p \times p \times p}$ and $\bA \in \bbR^{p \times p}$, let
\begin{align*}
    \langle \cT, \bA \rangle =  \big( \langle \cT_i, \bA \rangle \big)_{i \in [p]} \in \bbR^p,
\end{align*}
where $\cT_i = (\cT_{ijk})_{j, k \in [p]} \in \bbR^{p \times p}$. For a three times differentiable function $f : \bbR^{p} \rightarrow \bbR$, let
\begin{align*}
    \nabla^3 f(\theta) = \left( \dfrac{\partial^3}{\partial \theta_{i_1} \partial \theta_{i_2} \partial \theta_{i_3} } f(\theta) \right)_{i_1, i_2, i_3 \in [p]} \in \bbR^{p \times p \times p}.
\end{align*}

\begin{proposition}[Verifiaction of \eqref{A3}] \label{prop:A5}
    Assume that there exists a constant $C_8 > 0$ such that
    \begin{align*}
        \max_{t \in \bbN} \| x_t \|_2 \leq C_8.
    \end{align*}
    Then, there exists $K = K(C_8) > 0$ such that
    \begin{align*} 
        \max_{t \in \bbN} \| \nabla^2 \ell_t ( \theta ) - \nabla^2 \ell_t ( \theta' ) \|_2 
        \leq K \| \theta - \theta' \|_2, \quad
        \forall \theta, \theta' \in \Theta_r.
    \end{align*}
    Furthermore, there exists $K' = K'(C_8) > 0$ such that for any $\theta, \theta + h \in \Theta_r$
    \begin{align*}
        e^{-K' \| h \|_2} \bF_{t, \theta}
        \preceq \bF_{t, \theta+h}
        \preceq e^{K' \| h \|_2} \bF_{t, \theta}, \quad \forall t \in \bbN.
    \end{align*}
\end{proposition}
\begin{proof}
    Let $t \in \bbN$ and $\theta, \theta' \in \Theta_r$.
    Let $\cU = \{ u \in \bbR^p : \| u \|_2 = 1 \}$.
    By Taylor's theorem, we have
    \begin{align*}
        &\big\| \nabla^2 \ell_t ( \theta ) - \nabla^2 \ell_t ( \theta' ) \big\|_2
        =
        \sup_{u \in \cU}
        \left|
        \left\langle \int_{0}^{1} \nabla^3 \ell_t \big( \theta + s (\theta' - \theta) \big) \rmd s, \ (\theta - \theta') \otimes u \right\rangle \right| \\
        &\qquad \leq 
        \sup_{s \in [0, 1]} \left\| \nabla^3 \ell_t \big( \theta + s (\theta' - \theta) \big)  \right\|_2
        \| \theta - \theta' \|_2
        \leq 
        \sup_{\overline \theta \in \Theta_r} \big\| \nabla^3 \ell_t ( \overline \theta ) \big\|_2
        \| \theta - \theta' \|_2.
    \end{align*}
    Let $\overline \theta \in \Theta_r$.
    Note that 
    \begin{align*}
        \nabla^3 \ell_t ( \overline \theta ) = b'''(x_{t}^{\top} \overline \theta) x_t^{\otimes 3}, \quad 
        \sup_{\eta \in \bbR} | b'''(\eta) | \leq 1.
    \end{align*}
    It follows that
    \begin{align*}
        \big\| \nabla^3 \ell_t ( \overline \theta ) \big\|_2
        = 
        \sup_{u \in \cU} \left| \langle b'''(x_{t}^{\top} \overline \theta ) x_t^{\otimes 3}, u^{\otimes 3} \rangle \right|
        \leq 
        \sup_{u \in \cU} \left| \langle x_t^{\otimes 3}, u^{\otimes 3} \rangle \right|
        =
        \sup_{u \in \cU} \left| \langle x_t, u \rangle \right|^3
        \leq 
        \| x_t \|_2^3
        \leq C_8^3.
    \end{align*}
    Therefore, we have
    \begin{align*}
        \max_{t \in \bbN} \| \nabla^2 \ell_t ( \theta ) - \nabla^2 \ell_t ( \theta' ) \|_2 
        \leq C_8^3 \| \theta - \theta' \|_2.
    \end{align*}
    Next, we will prove the second assertion. Let $\theta, \theta + h \in \Theta_r$. By Lemmas \ref{lemma:logit_self_concordance} and \ref{lemma:smooth_hessian_logit}, we have
    \begin{align*}
        e^{-\sqrt{6} \|x_t\|_2 \| h \|_2 } \nabla^2 \ell_t (\theta)
        \preceq
        \nabla^2 \ell_t (\theta + h)
        \preceq
        e^{\sqrt{6} \|x_t\|_2 \| h \|_2 } \nabla^2 \ell_t (\theta),
    \end{align*}
    which, combining with the assumption, can be simplified to
    \begin{align*}
        e^{-\sqrt{6} C_8 \| h \|_2 } \nabla^2 \ell_t (\theta)
        \preceq
        \nabla^2 \ell_t (\theta + h)
        \preceq
        e^{\sqrt{6} C_8 \| h \|_2 } \nabla^2 \ell_t (\theta).
    \end{align*}
    Recall that $\bF_{t, \theta} = \sum_{s=1}^{t} \nabla^2 \ell_s (\theta)$ for $\theta \in \Theta$.
    It follows that
    \begin{align*}
        e^{-\sqrt{6} C_8 \| h \|_2 } \bF_{t, \theta}
        \preceq
        \nabla^2 \bF_{t, \theta+h}
        \preceq
        e^{\sqrt{6} C_8 \| h \|_2 } \bF_{t, \theta},
    \end{align*}
    which completes the proof.
\end{proof}

The following proposition demonstrate that under some conditions
\begin{align*} 
        \sum_{s = \tau+1}^{t} \langle \nabla \ell_{s}( \theta_{s-1}), \bOmega_{s}^{-1}  \nabla \ell_{s}( \theta_{s-1}) \rangle
        &\lesssim p \big( \log (t/t_{0}) + 1 \big) + \rmx + \log t \\
        &\lesssim p \log t + \rmx,
\end{align*}
which provides a justification of \eqref{A4-b}.

\begin{proposition}[Verification of \eqref{A4-b} in \eqref{A4}] \label{prop:A3-b}
    Let $\rmx > 0$. For $\tau > 0$ and $t \in \bbN$, let $I_{t}(\tau) = \{ s \in [t] : |x_s^{\top} \theta_{\star}| \leq \tau \}$.
    We assume the following:
    \ben
    \item There exist some constants $C_1, C_2, C_3, C_4 > 0$ such that for any $t \in \bbN$ with $t \geq C_1 p$
    \begin{align*} 
        \lambda_{\min} \bigg( \sum_{s \in I_{t}(C_2) } x_s x_s^{\top} \bigg) \geq C_3 t, \quad
        \lambda_{\max} \bigg( \sum_{s = 1}^{t} x_s x_s^{\top} \bigg) \leq C_4 t.
    \end{align*}
    \item There exist some constants $C_5, C_6 > 0$ such that 
    \begin{align*}
        \max_{t \in \bbN} \| x_t \|_2 \cdot r \leq C_5, \quad 
        \max_{t \in \bbN} \| x_t \|_2^2 \leq C_6.
    \end{align*}
    \item For the phase transition time $t_0$ satisfying $t_0 \geq C_1 (p + \rmx)$, there exists an event $\scrE_0(\rmx)$ with $\bbP_{\theta_{\star}}(\scrE_0(\rmx)) > 1 - e^{-\rmx}$ such that, on $\scrE_0(\rmx)$, $\hat \theta_{t_0}^{\MAP} \in \Theta_r$.
    \item Let $\epsilon_t = Y_t - b'( x_t^{\top} \theta_{\star} )$ and $\sigma_t^2 = b''( x_t^{\top} \theta_{\star} )$. There exist some constants $\nu, \alpha > 0$ such that for any $t \in \bbN$
    \begin{align*}
        \bbE \big[ e^{\lambda (\epsilon_t^2 - \sigma_t^2)} \mid \cF_{t-1} \big] \leq \exp \left( \nu^2 \dfrac{\lambda^2}{2} \right), \quad
        \forall |\lambda| \leq 1/\alpha.
    \end{align*}
    \een
    Then, there exists an event $\scrE(\rmx)$ with $\bbP_{\theta_{\star}}(\scrE(\rmx)) \geq 1 - e^{-\rmx}$ such that on $\scrE(\rmx) \cap \scrE_0$, for any $t \in \bbN$ with $t_{0} < t$ and $\{ \theta_{s} \}_{s = t_0+1}^{t-1} \subset \Theta_r$,
    \begin{align*} 
        \sum_{s = t_0+1}^{t} \langle \nabla \ell_{s}( \theta_{s-1}), \bOmega_{s}^{-1}  \nabla \ell_{s}( \theta_{s-1}) \rangle
        \leq 
        K \Big( p \big( \log (t/t_{0}) + 1 \big) + \rmx + \log t \Big),
    \end{align*}
    where $K = K(\nu, \alpha, C_2, C_3, C_4, C_5, C_6) > 0$.
\end{proposition}

\begin{proof}
    For simplicity of notation, we introduce the following notations:
    \begin{align*}
        \hat \mu_t = b'( x_{t}^{\top} \theta_{t-1} ), \quad
        \mu_t = b'( x_{t}^{\top} \theta_{\star} ), \quad
        \hat \sigma_t^2 = b''( x_{t}^{\top} \theta_{t-1} ), \\
        \delta_t = \hat \mu_t - \mu_t, \quad 
        g_t = \nabla \ell_t (\theta_{t-1}), \quad 
        \bH_t = \nabla^2 \ell_t (\theta_{t-1}).
    \end{align*}
    Fix $t_{0} \in \bbN$ with $t_{0} \geq C_1 (p + \rmx)$ and let $t \in \bbN$ with $t_{0} < t$. In this proof, we assume that $\{ \theta_s \}_{s=t_{0}+1}^{t-1} \subset \Theta_r$ and $\hat \theta_{t_0}^{\MAP} \in \Theta_r$ on $\scrE_0$.
    Note that
    \begin{align*}
        &\sum_{s=t_{0}+1}^{t} \langle g_s, \bOmega_{s}^{-1} g_s \rangle
        = \sum_{s=t_{0}+1}^{t} \big( Y_s - \hat \mu_s \big)^2 \langle x_s, \bOmega_{s}^{-1} x_s \rangle
        = \sum_{s=t_{0}+1}^{t} \big( \epsilon_s - \delta_s \big)^2 \langle x_s, \bOmega_{s}^{-1} x_s \rangle \\
        &\qquad \leq 
        2\sum_{s=t_{0}+1}^{t} \epsilon_s^2 \langle x_s, \bOmega_{s}^{-1} x_s \rangle
        +
        2\sum_{s=t_{0}+1}^{t} \delta_s^2 \langle x_s, \bOmega_{s}^{-1} x_s \rangle \\
        &\qquad =
        2\sum_{s=t_{0}+1}^{t} (\epsilon_s^2 - \sigma_s^2) \langle x_s, \bOmega_{s}^{-1} x_s \rangle
        +
        2\sum_{s=t_{0}+1}^{t} \sigma_s^2 \langle x_s, \bOmega_{s}^{-1} x_s \rangle
        +
        2\sum_{s=t_{0}+1}^{t} \delta_s^2 \langle x_s, \bOmega_{s}^{-1} x_s \rangle \\
        &\qquad =
        2\sum_{s=t_{0}+1}^{t} \xi_s \langle x_s, \bOmega_{s}^{-1} x_s \rangle
        +
        2\sum_{s=t_{0}+1}^{t} \sigma_s^2 \langle x_s, \bOmega_{s}^{-1} x_s \rangle
        +
        2\sum_{s=t_{0}+1}^{t} \delta_s^2 \langle x_s, \bOmega_{s}^{-1} x_s \rangle,
    \end{align*}
    where $\xi_s = \epsilon_s^2 - \sigma_s^2$.
    By the last display, we only need to bound the following quantities:
    \begin{align*}
        ({\rm i})_t = \sum_{s=t_{0}+1}^{t} \xi_s \langle x_s, \bOmega_{s}^{-1} x_s \rangle, \quad
        ({\rm ii})_t = \sum_{s=t_{0}+1}^{t} \sigma_s^2 \langle x_s, \bOmega_{s}^{-1} x_s \rangle, \quad
        ({\rm iii})_t = \sum_{s=t_{0}+1}^{t} \delta_s^2 \langle x_s, \bOmega_{s}^{-1} x_s \rangle.
    \end{align*}
    \textbf{Step 1: $({\rm i})_t$} \\
    First, we will obtain an upper bound of $({\rm i})_t$. 
    Let $S_t = \sum_{s=t_{0}+1}^{t} \xi_s \langle x_s, \bOmega_{s}^{-1} x_s \rangle$.
    Let $q_t = \langle x_t, \bOmega_t^{-1} x_t \rangle$ and $\widetilde \xi_t = \xi_t q_t$.
    Note that $q_t$ is $\cF_{t-1}$-measurable.
    Combining with the assumption, we have
    \begin{align*}
        \bbE \big[ e^{\lambda \widetilde \xi_s} \mid \cF_{s-1} \big]
        =        
        \bbE \big[ e^{\lambda \xi_s q_s} \mid \cF_{s-1} \big]
        \leq 
        \exp \left(  (\nu q_s)^2 \dfrac{\lambda^2}{2} \right), \quad
        \forall |\lambda| \leq 1/(\alpha q_s).
    \end{align*}    
    Let $\rmx > 0$ and
    \begin{align*}
        V_t = \sum_{s=t_{0} + 1}^{t} \nu^2 q_s^2, \quad 
        A_t = \alpha \max_{s \in \{ t_{0}+1, ..., t \}} q_s, \quad 
        \rmx_t = \rmx + \log t
    \end{align*}
    By Lemma \ref{lemma:MDS_subexp}, there exists an event $\scrE(\rmx)$ with $\bbP_{\theta_{\star}}(\scrE(\rmx)) \geq 1 - e^{-\rmx}$ such that, on $\scrE(\rmx)$, for any $t > t_{0}$,
    \begin{align*}
        \left| S_t \right| 
        \leq c_1 \left( \sqrt{V_t(\rmx_t + 1)} + A_t (\rmx_t + 1)  \right)
        \leq c_2 \left( \sqrt{V_t \rmx_t} + A_t \rmx_t \right),
    \end{align*}
    where $c_1, c_2 > 0$ are universal constants, and the second inequality holds by $t > 1$.
    To bound the right-hand side in the last display, we will obtain upper bounds of $V_t$ and $A_t$.
    Note that
    \begin{align}
    \begin{aligned} \label{eqn:ex_1_1} 
        \lambda_{\max} \big( \bOmega_t \big)
        &\leq 
        \| \bOmega_0 \|_2 + \lambda_{\max} \bigg( \sum_{s=1}^{t_{0}} \nabla^2 \ell_s ( \hat \theta_{t_0}^{\MAP} )
        +
        \sum_{s=t_{0}+1}^{t} \nabla^2 \ell_s (\theta_{s-1}) \bigg) \\
        &\leq 
        \| \bOmega_0 \|_2 + \dfrac{1}{4}\lambda_{\max} \bigg( \sum_{s = 1}^{t} x_s x_s^{\top} \bigg) 
        \leq
        \overline \lambda_0 + \dfrac{C_4}{4} t \leq c_3 t,
    \end{aligned}        
    \end{align}
    where the third inequality holds by the definition of $\overline \lambda_0$ and the assumption, and $c_3 = c_3(C_4) > 0$. 
    Recall that $\theta_{t_0} \in \Theta_r$ and $\{ \theta_s \}_{s=t_{0}}^{t-1} \subset \Theta_r$. Then, we have on $\scrE_0$
    \begin{align}
    \begin{aligned} \label{eqn:ex_1_2} 
        \lambda_{\min} \big( \bOmega_t \big)
        &\geq 
        \lambda_{\min} \big( \bOmega_0 \big) 
        + \lambda_{\min} \left( 
        \sum_{s=1}^{t_{0}} \nabla^2 \ell_s ( \hat \theta_{t_0}^{\MAP} )
        +
        \sum_{s=t_{0}+1}^{t} \nabla^2 \ell_s (\theta_{s-1}) 
        \right) \\
        &\geq 
        \underline \lambda_0 + c_4 t
        \geq 
        c_4 t,
    \end{aligned}        
    \end{align}
    where $c_4 = c_4(C_2, C_3, C_4, C_5) > 0$ and the second inequality holds by Proposition \ref{prop:A1}.
    Then, we have on $\scrE_0$
    \begin{align*}
        V_t 
        &= \sum_{s=t_{0} + 1}^{t} \nu^2 \langle x_s, \bOmega_{s}^{-1} x_s \rangle^2
        \leq 
        \nu^2 \sum_{s=t_{0} + 1}^{t} \| x_s \|_2^4 \lambda_{\min}^{-2} \big( \bOmega_{s} \big)
        \leq 
        \nu^2 C_6^2 \sum_{s=t_{0} + 1}^{t} (c_4 s)^{-2} \\
        &\leq 
        c_5 t_{0}^{-1} \leq c_5,
    \end{align*}
    where $c_5 = c_5(\nu, C_6, c_4) = c_5(\nu, C_2, C_3, C_4, C_5, C_6) > 0$.
    Also, on $\scrE_0$
    \begin{align*}
        A_t 
        &= \alpha \max_{s \in \{ t_{0}+1, ..., t \}} \langle x_s, \bOmega_{s}^{-1} x_s \rangle
        \leq \alpha \lambda_{\min}^{-1} \big( \bOmega_{t_{0} +1} \big) \max_{s \in \{ t_{0}+1, ..., t \}} \| x_s \|_2^2  \\
        &\leq \alpha C_6 (c_4 t_{0})^{-1}
        \leq c_6,
    \end{align*}
    where $c_6 = c_6(\alpha, C_6, c_4) = c_6(\alpha, C_2, C_3, C_4, C_5, C_6) > 0$.
    Therefore, on $\scrE(\rmx) \cap \scrE_0$, the following inequalities hold uniformly for all $t > t_{0}$:
    \begin{align} \label{eqn:ex_1_4}
        ({\rm i})_t \leq
        \left| S_t \right| 
        \leq c_2 \left( \sqrt{V_t \rmx_t} + A_t \rmx_t \right)
        \leq c_7 \rmx_t,
    \end{align}
    where $c_7 = c_7(c_2, c_5, c_6) = c_7(\nu, \alpha, C_2, C_3, C_4, C_5, C_6) > 0$.

    \noindent \textbf{Step 2:} $({\rm ii})_t$ and $({\rm iii})_t$ \\
    Next, we will obtain upper bounds of $({\rm ii})_t$ and $({\rm iii})_t$.
    Let $s \in \bbN$ with $t_0 < s \leq t$.
    By the definitions, we have
    \begin{align*}
        |\delta_s| 
        &= | b'( x_{s}^{\top} \theta_{s-1} ) - b'( x_{s}^{\top} \theta_{\star} ) |
        = b''( x_{s}^{\top} \widetilde \theta_{s-1} ) | x_s^{\top} ( \theta_{s-1} - \theta_{\star} ) | \\
        &= \dfrac{ b''( x_{s}^{\top} \widetilde \theta_{s-1} ) }{ b''( x_{s}^{\top} \theta_{s-1} ) }
        | x_s^{\top} ( \theta_{s-1} - \theta_{\star} ) |
        b''( x_{s}^{\top} \theta_{s-1} ) \\
        \overset{ \text{Lemma \ref{lemma:b_ratio}} }&{\leq}
        \exp\big( 3 |x_s^{\top} ( \widetilde \theta_{s-1} - \theta_{s-1} ) | \big)
        | x_s^{\top} ( \theta_{s-1} - \theta_{\star} ) |
        b''( x_{s}^{\top} \theta_{s-1} ) \\
        &\leq
        \exp\big( 6 \max_{t \in \bbN} \| x_t \|_2 \cdot r \big)
        \big( \max_{t \in \bbN} \| x_t \|_2 \cdot r \big)
        b''( x_{s}^{\top} \theta_{s-1} ) \\
        &\leq 
        \exp\big( 6 C_5 \big) C_5 b''( x_{s}^{\top} \theta_{s-1} )
    \end{align*}
    where the second equality holds for some $\widetilde \theta_{s-1} \in \Theta_r$ by the Taylor's theorem. 
    By the last display, we have
    \begin{align*}
        \delta_s^2
        \leq 
        \exp\big( 12 C_5 \big) C_5^2 \big[ b''( x_{s}^{\top} \theta_{s-1} ) \big]^2
        \leq 
        \dfrac{e^{12C_5} C_5^2}{4} b''( x_{s}^{\top} \theta_{s-1} )
        =
        c_8 \hat \sigma_s^2,
    \end{align*}
    where the last inequality holds by $b''(\cdot) \leq 1/4$, and $c_8 = c_8(C_5) > 0$.
    Also,
    \begin{align*}
        \sigma_s^2 
        &= b''( x_{s}^{\top} \theta_{\star} )
        = \dfrac{ b''( x_{s}^{\top} \theta_{\star} ) }{ b''( x_{s}^{\top} \theta_{s-1} ) } 
        b''( x_{s}^{\top} \theta_{s-1} ) \\
        \overset{ \text{Lemma \ref{lemma:b_ratio}} }&{\leq}
        \exp\big( 3 |x_s^{\top} ( \theta_{s-1} - \theta_{\star} ) | \big)
        b''( x_{s}^{\top} \theta_{s-1} )
        \leq 
        \exp\big( 3 \max_{t \in \bbN} \| x_t \|_2 \cdot r \big)
        b''( x_{s}^{\top} \theta_{s-1} ) \\
        &\leq 
        \exp\big( 3 C_5 \big) b''( x_{s}^{\top} \theta_{s-1} )
        =
        c_9 b''( x_{s}^{\top} \theta_{s-1} ) = c_9 \hat \sigma_s^2,
    \end{align*}
    where $c_9 = c_9(C_5) > 0$.
    It follows that
    \begin{align*}
        ({\rm ii})_t 
        &= 
        \sum_{s=t_{0}+1}^{t} \sigma_s^2 \langle x_s, \bOmega_{s}^{-1} x_s \rangle
        \leq 
        c_9 \sum_{s=t_{0}+1}^{t} \hat \sigma_s^2 \langle x_s, \bOmega_{s}^{-1} x_s \rangle
        =
        c_9 \sum_{s=t_{0}+1}^{t} \langle \bOmega_{s}^{-1}, \bH_s \rangle \\
        &\leq
        c_9 \big( \logdet(\bOmega_t) - \logdet(\bOmega_{t_{0}}) \big),
    \end{align*}
    where the second equality and inequality hold by $\bH_s = \hat \sigma_s^2 x_s^{\otimes 2}$ and Lemma \ref{lemma:elliptic_potential}, respectively.
    Also,
    \begin{align*}
        ({\rm iii})_t 
        &= 
        \sum_{s=t_{0}+1}^{t} \delta_s^2 \langle x_s, \bOmega_{s}^{-1} x_s \rangle
        \leq 
        c_8 \sum_{s=t_{0}+1}^{t} \hat \sigma_s^2 \langle x_s, \bOmega_{s}^{-1} x_s \rangle
        =
        c_8 \sum_{s=t_{0}+1}^{t} \langle \bOmega_{s}^{-1}, \bH_s \rangle \\
        &\leq
        c_8 \big( \logdet(\bOmega_t) - \logdet(\bOmega_{t_{0}}) \big).
    \end{align*}
    To bound $({\rm ii})_t$ and $({\rm iii})_t$, we will obtain an upper bound of $\logdet(\bOmega_t) - \logdet(\bOmega_{t_{0}})$.
    By \eqref{eqn:ex_1_1} and \eqref{eqn:ex_1_2}, we have on $\scrE_0$
    \begin{align*}
        \logdet(\bOmega_t) - \logdet(\bOmega_{t_{0}})
        &\leq 
        p \log \Big( \lambda_{\max} \big( \bOmega_t \big) \Big)
        -
        p \log \Big( \lambda_{\min} \big( \bOmega_{t_{0}} \big) \Big) \\
        &\leq
        p \log (c_3 t) - p \log (c_4 t_{0})
        \leq
        p \log (t/t_{0}) + p |\log(c_3/c_4)|.
    \end{align*}
    Therefore, we have on $\scrE_0$
    \begin{align}
    \begin{aligned} \label{eqn:ex_1_3}
        ({\rm ii})_t &\leq c_9 p \log (t/t_{0}) +  p \big( c_9 |\log(c_3/c_4)| \big) 
        \leq c_{10} \cdot p \big( \log (t/t_{0}) + 1 \big), \\
        ({\rm iii})_t &\leq c_8 p \log (t/t_{0}) +  p \big( c_8 |\log(c_3/c_4)| \big)
        \leq c_{10} \cdot p \big( \log (t/t_{0}) + 1 \big),
    \end{aligned}        
    \end{align}
    where $c_{10} = c_{10}(c_3, c_4, c_8, c_9) = c_{10}(C_2, C_3, C_4, C_5) > 0$.

    \noindent \textbf{Step 3:}
    Combining \eqref{eqn:ex_1_4} and \eqref{eqn:ex_1_3}, we have, on $\scrE(\rmx) \cap \scrE_0$, 
    \begin{align*}
        2 \big( ({\rm i})_t + ({\rm ii})_t + ({\rm iii})_t \big)
        &\leq 
        2c_7 (\rmx + \log t) + 
        2c_{10} \cdot p \big( \log (t/t_{0}) + 1 \big) +
        2c_{10} \cdot p \big( \log (t/t_{0}) + 1 \big) \\
        &\leq 
        c_{11} \Big( p \big( \log (t/t_{0}) + 1 \big) + \rmx + \log t \Big),
    \end{align*}    
    where $c_{11} = c_{11}(c_7, c_{10}) = c_{11}(\nu, \alpha, C_2, C_3, C_4, C_5, C_6) > 0$.
    This completes the proof.
\end{proof}

\begin{remark}
    If we consider $C_6 = \widetilde C_6 p$ for some constant $\widetilde C_6 > 0$, we can replace the bound in Proposition \ref{prop:A3-b} with
    \begin{align*} 
            \sum_{s = \tau+1}^{t} \langle \nabla \ell_{s}( \theta_{s-1}), \bOmega_{s}^{-1}  \nabla \ell_{s}( \theta_{s-1}) \rangle
            \leq 
            \widetilde K \Big( p \big( \log (t/t_{0}) + 1 \big) + \sqrt{p ( \rmx + \log t) } + \rmx + \log t \Big),
    \end{align*}
    where $\widetilde K = \widetilde K(\nu, \alpha, C_2, C_3, C_4, C_5, \widetilde C_6) > 0$.    
\end{remark}

\begin{remark}
    Since we focus on logistic regression with bounded $\epsilon_t$, the fourth assumption in Proposition \ref{prop:A3-b} can be satisfied with $\nu = 1/2$ and $\alpha = 0$ (with the convention $1/0 = \infty$), which leads to a simpler and tighter analysis. However, we aim to provide a general proof strategy for \eqref{A4-b}. In particular, the same proof framework can be readily extended to various statistical models with general error distributions (e.g., $\alpha$-sub-exponential distributions; see \citet{Gotze2021subE}). Therefore, we believe that retaining the current proof is beneficial for readers.
\end{remark}

\subsection{Example: general setup}
\begin{proposition}[General version of \eqref{A4-a} in \eqref{A4}] \label{prop:A3-a}
    Let $(\Omega, \cA)$ be a measurable space equipped with a filtration $\bbF = (\cF_t)_{t \in \bbN_0}$.
    Let $\Theta = \bbR^p$ be the parameter space, and $\{ \bbP_{\theta} : \theta \in \Theta \}$ be a family of probability measures on $(\Omega, \cA)$.
    Let $\bbF$-adapted real-valued data stream $(Y_t)_{t \in \bbN}$ and $\bbF$-adapted $\Theta$-valued sequence $(\widetilde \theta_t)_{t \in \bbN_0}$ be given.
    We assume the following:
    \ben
    \item For $t \in \bbN$, $Y_t$ admits a jointly measurable family of conditional densities
    \begin{align*}
        (\omega, \theta, y) \mapsto p_{t, \theta}(y \mid \cF_{t-1})(\omega), \quad (\omega, \theta, y) \in \Omega \times \Theta \times \bbR
    \end{align*}
    such that for every $\theta \in \Theta$ and all measurable $A \subset \bbR$
    \begin{align*}
        \bbP_{\theta} \big( Y_t \in A \mid \cF_{t-1} \big) = \int_{A} p_{t, \theta}(y \mid \cF_{t-1}) \rmd y, \quad {\rm a.s.}
    \end{align*}
    \item For $t \in \bbN$, the support $S_t = \{ y \in \bbR : p_{t, \theta}(y \mid \cF_{t-1}) > 0 \}$ is independent of $\theta \in \Theta$ almost surely.
    \een
    Let $\ell_t(\theta) = -\log p_{t, \theta}(Y_t \mid \cF_{t-1})$.
    Then, for any fixed true parameter $\theta_{\star}$ and $\rmx > 0$, we have
    \begin{align*}
        \bbP_{\theta_{\star}} \bigg( \sup_{t \geq 1} \sum_{s=1}^{t} \big( \ell_s(\theta_{\star}) - \ell_s(\widetilde \theta_{s-1}) \big)  \leq \rmx \bigg) \geq 1 - e^{-\rmx}.
    \end{align*}
\end{proposition}
\begin{proof}
    Let $M_0 = 1$. For $t \in \bbN$, define the likelihood ratio process $(M_t)_{t \in \bbN}$ as follows:
    \begin{align*}
        M_t = \prod_{s=1}^{t} \dfrac{ p_{s, \widetilde \theta_{s-1}} (Y_s \mid \cF_{s-1} ) }{ p_{s, \theta_{\star}} (Y_s \mid \cF_{s-1} ) }.
    \end{align*}
    By the definition of $\ell_t(\theta)$, we have for any $t \in \bbN$
    \begin{align*}
        \ell_t(\theta_{\star}) - \ell_t(\widetilde \theta_{t-1})
        &=
        - \log p_{t, \theta_{\star}} (Y_t \mid \cF_{t-1} )
        + \log p_{t, \widetilde \theta_{t-1}} (Y_t \mid \cF_{t-1} ) \\
        &=
        \log \dfrac{ p_{t, \widetilde \theta_{t-1}} (Y_t \mid \cF_{t-1} ) }{ p_{t, \theta_{\star}} (Y_t \mid \cF_{t-1} ) }.
    \end{align*}
    Hence, we have
    \begin{align*}
        \log M_t 
        = \sum_{s=1}^{t} \log \dfrac{ p_{s, \widetilde \theta_{s-1}} (Y_s \mid \cF_{s-1} ) }{ p_{s, \theta_{\star}} (Y_s \mid \cF_{s-1} ) }
        = \sum_{s=1}^{t} \ell_s(\theta_{\star}) - \ell_s(\widetilde \theta_{s-1}).
    \end{align*}
    Note that
    \begin{align*}
        \bbE_{\theta_{\star}} \big[ M_t \mid \cF_{t-1} \big]
        &=
        M_{t-1}
        \int
        \dfrac{ p_{t, \widetilde \theta_{t-1}} (y \mid \cF_{t-1} ) }{ p_{t, \theta_{\star}} (y \mid \cF_{t-1} ) }
        p_{t, \theta_{\star}} (y \mid \cF_{t-1} ) \rmd y \\
        &=
        M_{t-1}
        \int
        p_{t, \widetilde \theta_{t-1}} (y \mid \cF_{t-1} ) \rmd y 
        =
        M_{t-1}, \quad {\rm a.s.},
    \end{align*}
    where $\bbE_{\theta_{\star}}$ denotes the expectation under $\bbP_{\theta_{\star}}$. By applying Ville's maximal inequality for non-negative (super)martingale, for any $\lambda > 0$, we have
    \begin{align*}
        \bbP_{\theta_{\star}} \Big( \sup_{t \geq 1} M_t \geq \lambda \Big) \leq \dfrac{ \bbE_{\theta_{\star}} M_0 }{\lambda} \leq \lambda^{-1}.
    \end{align*}
    For any $\rmx > 0$ with $\lambda = e^{\rmx} > 0$, therefore, we have
    \begin{align*}
        \bbP_{\theta_{\star}} \Big( \sup_{t \geq 1} \log M_t \geq \rmx \Big) \leq e^{-\rmx},
    \end{align*}
    which completes the proof.
\end{proof}

\section{Technical lemmas}

\subsection{Lemmas for main results}

\begin{lemma} \label{lemma:Gaussian_TV}
    For $\mu_{1}, \mu_{2} \in \bbR^{p}$ and $\bOmega_{1}, \bOmega_{2} \in \symmPD$, let $Q_1 = \cN(\mu_{1}, \bOmega_{1}^{-1})$ and $Q_2 = \cN(\mu_{2}, \bOmega_{2}^{-1})$.
    Suppose that 
    \begin{align*}
        \left\| \bOmega_{2}^{-1/2} \bOmega_{1} \bOmega_{2}^{-1/2} - \bI_{p} \right\|_{2} \leq 0.684.
    \end{align*}
    Then,
    \begin{align*}
        d_{V} \left( Q_1, Q_2 \right) 
        \leq 
        \dfrac{1}{2} \Bigg( \left\| \bOmega_{1}^{1/2} \left( \mu_1 - \mu_2  \right)  \right\|_{2}^{2}  
        +  
        \left\| \bOmega_{2}^{-1/2} \bOmega_{1} \bOmega_{2}^{-1/2} - \bI_{p} \right\|_{\rm F}^{2}
        \Bigg)^{1/2}.
    \end{align*}
\end{lemma}
\begin{proof}
    By Pinsker's inequality, we have
    \begin{align*}
    d_{V} \left( Q_1, Q_2 \right) \leq \bigg( \dfrac{1}{2} K \left( Q_1, Q_2 \right) \bigg)^{1/2}.
    \end{align*}
    By the definition of KL divergence, $K \left( Q_1, Q_2 \right)$ is equal to
    \begin{align*}
    \dfrac{1}{2}  \bigg[  \left\| \bOmega_{1}^{1/2} \left( \mu_1 - \mu_2  \right) \right\|_{2}^{2} + \operatorname{tr} \left( \bOmega_{2}^{-1/2} \bOmega_{1} \bOmega_{2}^{-1/2} - \bI_{p}  \right) - \operatorname{logdet} \left( \bOmega_{2}^{-1/2} \bOmega_{1} \bOmega_{2}^{-1/2} \right) \bigg].
    \end{align*}
    Let $(\lambda_{j})_{j \in [p]}$ be eigenvalues of $\mathbf{B} = \bOmega_{2}^{-1/2} \bOmega_{1} \bOmega_{2}^{-1/2} - \bI_{p}$. Then, the last display is represented by
    \begin{align*}
    &\dfrac{1}{2}  \bigg[  \left\| \bOmega_{1}^{1/2} \left( \mu_1 - \mu_2  \right) \right\|_{2}^{2} 
    + \sum_{j=1}^{p} \lambda_{j} - \sum_{j=1}^{p} \log (1 + \lambda_{j}) \bigg] \\   
    &\leq  
    \dfrac{1}{2}  \bigg[  \left\| \bOmega_{1}^{1/2} \left( \mu_1 - \mu_2  \right) \right\|_{2}^{2} 
    + \sum_{j=1}^{p} \lambda_{j} - \sum_{j=1}^{p} (\lambda_{j} - \lambda_{j}^{2}) \bigg] 
    =
    \dfrac{1}{2}  \bigg[  \left\| \bOmega_{1}^{1/2} \left( \mu_1 - \mu_2  \right) \right\|_{2}^{2} 
    + \sum_{j=1}^{p} \lambda_{j}^{2} \bigg] \\   
    &=
    \dfrac{1}{2}  \bigg[  \left\| \bOmega_{1}^{1/2} \left( \mu_1 - \mu_2  \right) \right\|_{2}^{2} 
    + \operatorname{tr}(\mathbf{B}^{2}) \bigg] 
    =
    \dfrac{1}{2}  \bigg[  \left\| \bOmega_{1}^{1/2} \left( \mu_1 - \mu_2  \right) \right\|_{2}^{2} 
    + \left\| \mathbf{B} \right\|_{\rm F}^{2} \bigg]. 
    \end{align*}
    where the first inequality holds by $\max_{j \in [p]}|\lambda_{j}| \leq 0.684$.
    It follows that
    \begin{align*}
    d_{V} \left( Q_1, Q_2 \right) 
    \leq \bigg( \dfrac{1}{2} K \left( Q_1, Q_2 \right) \bigg)^{1/2}    
    \leq \dfrac{1}{2}  \bigg[  \left\| \bOmega_{1}^{1/2} \left( \mu_1 - \mu_2  \right) \right\|_{2}^{2} 
    + \left\| \mathbf{B} \right\|_{\rm F}^{2} \bigg]^{1/2}, 
    \end{align*}
    which completes the proof.
\end{proof}

\subsection{Lemmas for Section \ref{sec:assume_example}} 

We introduce the notion of smoothness for a given function. 
For $M > 0$, we say that a three times differentiable convex function $f : \bbR^{p} \to \bbR$ is $M$-self-concordant-like function if for any $x, z_1, z_2, z_3 \in \bbR^p$, $f$ satisfies 
\begin{align*}
    \left| \langle \nabla^3 f(x), z_1 \otimes z_2 \otimes z_3 \rangle \right|
    \leq 
    M \| z_1 \|_2 \| z_2 \|_{x} \| z_3 \|_{x},
\end{align*}
where $\| z \|_x = \sqrt{\langle z, \nabla^2 f(x) z \rangle}$ for $x, z \in \bbR^p$.
In the following lemma, we demenstrate that the logistic regression model loss $\ell_t$ enjoys the self-concordance-type smoothness.

\begin{lemma}[Self-concordance] \label{lemma:logit_self_concordance}
    For $t \in \bbN$, the loss function $\ell_t$ specified in \eqref{def:logit_loss} is $\sqrt{6} \| x_t \|_2$-self-concordant-like function.
\end{lemma}
\begin{proof}
    See Lemma 4 in \citet{tran2015composite}.
\end{proof}

\begin{lemma}[Smoothness of hessian] \label{lemma:smooth_hessian_logit}
    For $M > 0$, let $f: \bbR^p \to \bbR$ be a $M$-self-concordant-like function. Then, for any $u, h \in \bbR^p$, we have
    \begin{align*}
        \exp( - M \| h \|_2 ) \nabla^2 f(u) 
        \preceq
        \nabla^2 f(u + h) 
        \preceq            
        \exp( M \| h \|_2 ) \nabla^2 f(u) 
    \end{align*}
\end{lemma}
\begin{proof}
    See Theorem 4.(b) in \citet{tran2015composite}.
\end{proof}

\begin{lemma}[Smoothness of integrated hessian] \label{lemma:smooth_int_hessian_logit}
    For $M > 0$, let $f: \bbR^p \to \bbR$ be a $M$-self-concordant-like function. Then, for any $u, h \in \bbR^p$, we have
    \begin{align*}
        \dfrac{1 - e^{ -M \| h\|_2 }}{ M \| h\|_2 } \nabla^2 f(u) 
        \preceq
        \int_{0}^{1} \nabla^2 f(u + \alpha h) \rmd \alpha
        \preceq            
        \dfrac{e^{ M \| h\|_2 } - 1}{ M \| h\|_2 } \nabla^2 f(u) 
    \end{align*}
\end{lemma}
\begin{proof}
    See Corollary 2 in \citet{sun2019generalized}.
\end{proof}

\begin{lemma} \label{lemma:b_ratio}
    Let $b(\cdot) = \log( 1 + \exp(\cdot))$. Then, for any $\eta_1, \eta_2 \in \bbR$, we have
    \begin{align*}
        \dfrac{ b''(\eta_1) }{ b''(\eta_2) } \leq e^{3 |\eta_1 - \eta_2| }.
    \end{align*}
\end{lemma}
\begin{proof}
    See Lemma 34 in \citet{lee2026online}.
\end{proof}

Let $\bbF = (\cF_{t})_{t \in \bbN}$ be a filtration.
We say $(Z_t, \cF_t)_{t \in \bbN}$ is martingale difference sequence if $(Z_t)_{t \in \bbN}$ is adapted to $\bbF$ and 
\begin{align*}
    \bbE |Z_t| < + \infty, \quad \bbE \big[ Z_{t+1} \mid \cF_t \big] = 0.
\end{align*}

\begin{lemma} \label{lemma:MDS_subexp}
    Let $(\xi_t, \cF_t)_{t \in \bbN}$ be a martingale difference sequence. 
    For $t \in \bbN$, suppose that 
    \begin{align*}
        \bbE \big[ e^{\lambda \xi_t} \mid \cF_{t-1} \big] \leq e^{\lambda^2 \nu_t^2/2} \quad {\rm a.s.} \quad \forall |\lambda| < 1/\alpha_t,
    \end{align*}
    where $\nu_t, \alpha_t > 0$.
    Let $n \in \bbN$.
    Then, we have, for any $\rmx > 0$,
    \begin{align*}
        &\bbP \left( \bigg| \sum_{t=1}^{n} \xi_t \bigg| \geq K \left( \sqrt{ V_n (\rmx + 1) } + \alpha_{\max} (\rmx + 1) \right) \right) \leq 2e^{-\rmx}, \\
        &\bbP \left( \bigg| \sum_{s=1}^{t} \xi_s \bigg| \geq K \left( \sqrt{ V_n (\rmx_t + 1) } + \alpha_{\max} (\rmx_t + 1) \right) \text{ for some } t \in \bbN \right) \leq e^{-\rmx}, 
    \end{align*}
    where $V_n = \sum_{t=1}^{n} v_t^2$, $\alpha_{\max} = \max_{t \in [n]} \alpha_t$, $\rmx_t = \rmx + \log t$, and $K > 0$ is a universal constant.
\end{lemma}
\begin{proof}
    For a proof of the first assertion, see Theorem 2.19 in \citet{wainwright2019high}.
    From the first assertion, we have for any $\widetilde \rmx > 0$ and for a fixed $t \in \bbN$
    \begin{align*}
        \bbP \left( \bigg| \sum_{s=1}^{t} \xi_s \bigg| \geq c_1 \left( \sqrt{ V_n (\widetilde \rmx + 1) } + \alpha_{\max} (\widetilde \rmx + 1) \right) \right) \leq 2e^{-\widetilde \rmx},
    \end{align*}
    where $c_1 > 0$ is a universal constant.
    Let $\widetilde \rmx_t = \rmx + 2\log t + \log(\pi^2/3)$ with $\rmx > 0$.
    Then, we have
    \begin{align*}
        &\bbP \left( \bigg| \sum_{s=1}^{t} \xi_s \bigg| \geq c_1 \left( \sqrt{ V_n (\widetilde \rmx_t + 1) } + \alpha_{\max} (\widetilde \rmx_t + 1) \right) \text{ for some } t \in \bbN \right)  \\
        &\qquad \leq 
            \sum_{t=1}^{\infty}
            \bbP \left( \bigg| \sum_{s=1}^{t} \xi_s \bigg| \geq c_1 \left( \sqrt{ V_n (\widetilde \rmx_t + 1) } + \alpha_{\max} (\widetilde \rmx_t + 1) \right) \right) \\
        &\qquad \leq \sum_{t=1}^{\infty} 2e^{ - \widetilde \rmx_t}
        = \sum_{t=1}^{\infty} 2e^{ - (\rmx + 2\log t + \log(\pi^2/3))}
        = \sum_{t=1}^{\infty} \dfrac{6}{\pi^2 t^{2}} 
        = e^{-\rmx} \sum_{t=1}^{\infty} \dfrac{6}{\pi^2 t^{2}} 
        = e^{-\rmx}.
    \end{align*}
    Since 
    \begin{align*}
        c_1 \left( \sqrt{ V_n (\widetilde \rmx_t + 1) } + \alpha_{\max} (\widetilde \rmx_t + 1) \right)
        \leq 
        c_2 c_1 \left( \sqrt{ V_n (\rmx + \log t + 1) } + \alpha_{\max} (\rmx + \log t + 1) \right)
    \end{align*}
    for some universal constant $c_2 > 0$, we complete the proof of the second assertion.
\end{proof}

\begin{lemma} \label{lemma:elliptic_potential}
    Let $\tau \in \bbN$.
    Let $\bOmega_{\tau} \in \symmPD$ and $(\bH_t)_{t \in \bbN}$ with $\bH_t \in \symmPSD$ for any $t \in \bbN$. 
    Suppose that
    \begin{align*}
        \bOmega_t = \bOmega_{t-1} + \bH_t, \quad \forall t \in \bbN \text{ with } t > \tau.
    \end{align*}
    Then, for any $t \in \bbN$ with $\tau < t$, we have
    \begin{align*}
        \sum_{s= \tau+1}^{t} \langle \bOmega_{s}^{-1}, \bH_s \rangle 
        \leq 
        \logdet(\bOmega_t) - \logdet(\bOmega_{\tau}).
    \end{align*}
\end{lemma}
\begin{proof}
    Let $\tau \in \bbN$ and $t \in \bbN$ with $\tau < t$.
    For $s \in \bbN$ with $s \geq \tau$, let
    \begin{align*}
        \bA_s = \bOmega_s^{-1/2} \bH_s \bOmega_s^{-1/2}.
    \end{align*}
    Then, we have
    \begin{align*}
        \bOmega_{s-1} = \bOmega_{s} - \bH_s = \bOmega_s^{1/2} (\bI_p - \bA_s) \bOmega_s^{1/2},
    \end{align*}
    which implies that
    \begin{align*}
        \det(\bOmega_{s-1}) = \det(\bOmega_{s})\det(\bI_p - \bA_s).
    \end{align*}
    It follows that
    \begin{align*}
        \logdet(\bOmega_{s}) - \logdet(\bOmega_{s-1}) = -\logdet(\bI_p - \bA_s).
    \end{align*}
    Let $(\lambda_j(\bA_s))_{j=1}^{p}$ be a sequence of all eigenvalues of $\bA_s$.
    Since $\bA_s \in \symmPSD$ and $\bOmega_s \succ \bH_s$, we have $\lambda_j(\bA_s) \in [0, 1)$.
    Hence, we have
    \begin{align*}
        \operatorname{tr}(\bA_s)
        &= \sum_{j=1}^{p} \lambda_j(\bA_s)
        \leq \sum_{j=1}^{p} -\log \big( 1 - \lambda_j(\bA_s) \big)
        = -\log \prod_{j=1}^{p} \big( 1 - \lambda_j(\bA_s) \big) \\
        &= -\logdet ( 1 - \bA_s ),
    \end{align*}
    where the inequality holds by 
    \begin{align*}
        z \leq - \log(1- z), \quad \forall z \in [0, 1).
    \end{align*}
    Also, we have $\operatorname{tr}(\bA_s) = \operatorname{tr}(\bOmega_{s}^{-1} \bH_s) = \langle \bOmega_{s}^{-1}, \bH_s \rangle$. It follows that
    \begin{align*}
        \logdet(\bOmega_{s}) - \logdet(\bOmega_{s-1}) = \langle \bOmega_{s}^{-1}, \bH_s \rangle.
    \end{align*}
    By telescoping sum, therefore, we have
    \begin{align*}
        \sum_{s= \tau+1}^{t} \langle \bOmega_{s}^{-1}, \bH_s \rangle 
        \leq 
        \logdet(\bOmega_t) - \logdet(\bOmega_{\tau}).
    \end{align*}
    This completes the proof.
\end{proof}

\section{Algorithm derivations}

For $Q = N(\mu, \bOmega^{-1})$, we define the surrogate negative ELBO by
\begin{align*}
    F(Q) = F(\mu, \bOmega^{-1}) = \mathbb{E}_{Q} \big[ \widetilde \ell_t(\theta)\big] + K \big(Q ;\: \Pi_{t-1}\big).
\end{align*}
The quadratic structure of $\widetilde{\ell}_t(\cdot)$ yields
\begin{align*}
    \bbE_{Q}\big[\widetilde \ell_t(\theta)\big] = 
    \langle \nabla \ell_t(\theta_{t-1}) - \nabla^2 \ell_t(\theta_{t-1})\theta_{t-1}, \mu \rangle
    + \frac{1}{2} \operatorname{tr} \left(
    \nabla^2 \ell_t(\theta_{t-1})\big(\mu^{\otimes 2} + \bOmega^{-1}\big) \right)
    + C,
\end{align*}
where $C$ is a constant independent of $\mu$ and $\bOmega$. 
Moreover, the KL divergence between $Q$ and $\Pi_{t-1}$ is given by
\begin{align*}
   \frac{1}{2} 
   \Big( 
    \log |\bOmega_{t-1}^{-1}|
    -\log |\bOmega^{-1}|
    + \langle \bOmega_{t-1}, (\mu - \theta_{t-1})^{\otimes 2} \rangle
    + \operatorname{tr} ( \bOmega_{t-1}\bOmega^{-1}) - p 
    \Big).
\end{align*}
Combining these expressions, the gradients of $F(\mu, \bOmega^{-1})$ are given by
\begin{align*}
    \nabla_\mu \, F(\mu, \bOmega^{-1})
        &= \nabla \ell_t (\theta_{t-1})
            + \nabla^2 \ell_t(\theta_{t-1}) (\mu-\theta_{t-1}) 
            + \bOmega_{t-1} (\mu - \theta_{t-1}), \\
    \nabla_{\bOmega^{-1}} \, F(\mu, \bOmega^{-1})
        &= \frac{1}{2} \big( \nabla^2 \ell_t(\theta_{t-1}) -\bOmega + \bOmega_{t-1} \big).
\end{align*}
Solving the first-order optimality conditions yields the closed-form updates:
\begin{align*}
\begin{aligned} 
    \bOmega_t &= \bOmega_{t-1} + \nabla^2 \ell_t(\theta_{t-1}), \\
    \theta_t &= \theta_{t-1} - \bOmega_t^{-1} \nabla \ell_t  (\theta_{t-1}).
\end{aligned}    
\end{align*}

\end{appendix}

\end{document}